\documentclass[11pt]{amsart}

\usepackage[T1]{fontenc}
\usepackage[utf8]{inputenc}
\usepackage{lmodern}
\usepackage{microtype}
\usepackage[a4paper,margin=3.2cm]{geometry}

\usepackage{amsmath,amssymb,amsfonts,amsthm,mathtools}
\usepackage{mathrsfs}
\usepackage{tikz-cd}
\usepackage{enumitem}
\usepackage{xcolor}
\usepackage{booktabs}
\usepackage{tabularx}
\usepackage{array}

\usepackage[colorlinks=true,linkcolor=blue!50!black,citecolor=blue!50!black,urlcolor=blue!50!black]{hyperref}
\usepackage[nameinlink,capitalise,noabbrev]{cleveref}

\theoremstyle{plain}
\newtheorem{theorem}{Theorem}[section]
\newtheorem{proposition}{Proposition}[section]
\newtheorem{lemma}{Lemma}[section]
\newtheorem{corollary}{Corollary}[section]

\theoremstyle{definition}
\newtheorem{definition}{Definition}[section]
\newtheorem{notation}{Notation}[section]
\newtheorem{setup}{Setup}[section]

\newtheorem{example}{Example}[section]

\theoremstyle{remark}
\newtheorem{remark}{Remark}[section]

\newcommand{\X}{\mathfrak X}

\newcommand{\DdagX}{\mathscr D^{\dagger}_{\X,\mathbb Q}}

\newcommand{\Ovhol}{\operatorname{Ovhol}}
\newcommand{\Coh}{\operatorname{Coh}}
\newcommand{\Irr}{\operatorname{Irr}}
\newcommand{\Isoc}{\operatorname{Isoc}}
\newcommand{\Car}{\operatorname{Car}}
\newcommand{\Supp}{\operatorname{Supp}}
\newcommand{\Spec}{\operatorname{Spec}}

\newcommand{\Lie}{\operatorname{Lie}}

\newcommand{\gr}{\operatorname{gr}}

\newcommand{\fstr}{\mathrm{fstr}}

\newcommand{\Loc}{\operatorname{Loc}}
\newcommand{\Coeff}{\operatorname{Coeff}}
\newcommand{\Adm}{\operatorname{Adm}}

\crefname{theorem}{Theorem}{Theorems}
\Crefname{theorem}{Theorem}{Theorems}
\crefname{proposition}{Proposition}{Propositions}
\Crefname{proposition}{Proposition}{Propositions}
\crefname{lemma}{Lemma}{Lemmas}
\Crefname{lemma}{Lemma}{Lemmas}
\crefname{corollary}{Corollary}{Corollaries}
\Crefname{corollary}{Corollary}{Corollaries}
\crefname{remark}{Remark}{Remarks}
\Crefname{remark}{Remark}{Remarks}
\crefname{definition}{Definition}{Definitions}
\Crefname{definition}{Definition}{Definitions}
\crefname{setup}{Setup}{Setups}
\Crefname{setup}{Setup}{Setups}
\crefname{section}{Section}{Sections}
\Crefname{section}{Section}{Sections}

\title[Filtered Strongly Equivariant Arithmetic $D^\dagger$-Modules]
{Arithmetic Kashiwara Regularity and Orbit Classification for Filtered Strongly Equivariant  \(\mathscr D^\dagger\)-Modules}

\author{Andr\'es Sarrazola-Alzate}
\address{STEAM School, Universidad EIA, Envigado, Colombia}
\email{andres.sarrazola@eia.edu.co}

\date{}


\begin{document}

\begin{abstract}
We prove an arithmetic analogue of Kashiwara's regularity and orbit
classification theorem in the setting of filtered strongly equivariant Berthelot
arithmetic \(\mathscr D^\dagger\)-modules with Frobenius structure on formal
flag varieties. Let \(\mathcal G\) be a split connected reductive group scheme
over a complete discrete valuation ring of mixed characteristic, let
\(\mathfrak X=\widehat{\mathcal G/\mathcal B}\) be the formal flag variety,
and let \(\mathcal K\subseteq\mathcal G\) be a smooth closed subgroup whose
special fiber acts on \(X_s\) with finitely many orbits, all satisfying the orbit-separability condition. Writing
\(\mathfrak K=\widehat{\mathcal K}\), we introduce a filtered strong notion of
\(\mathfrak K\)-equivariance for coherent
\(\mathscr D^\dagger_{\mathfrak X,\mathbb Q}\)-modules.  Without Frobenius this
condition implies the microlocal holonomicity statement.  The point of the
filtered clause is precisely that it realizes infinitesimal equivariance on good finite-level models, so that the principal symbols of the fundamental vector fields vanish on the characteristic variety.  With Frobenius
structure, Caro's resolution of Berthelot's holonomicity-stability conjectures over smooth projective formal schemes identifies Frobenius holonomicity with Frobenius overholonomicity; the Frobenius base-change convention recalled by Huyghe--Schmidt, reinforced by Caro--Tsuzuki's overholonomicity theorem for overconvergent Frobenius isocrystals, then gives geometric overholonomicity, and the orbital devissage is then carried
out inside the overholonomic heart.  Thus the theorem is a regularity theorem for objects satisfying the filtered strong condition, not an existence theorem for that condition. The proof
proceeds through a moment-map containment for characteristic varieties,
followed, in the Frobenius range, by Caro's projective holonomicity-stability theorem
and then by an equivariant orbital devissage by arithmetic intermediate extensions. As a consequence, simple Frobenius filtered
strongly equivariant coherent arithmetic \(\mathscr D^\dagger\)-modules are
classified by pairs \((O,E)\), where \(O\) is a \(\mathcal K_s\)-orbit in
\(X_s\) and \(E\) is an irreducible \(\mathcal K\)-equivariant overconvergent
\(F\)-isocrystal on \((O,\overline O)\) whose intermediate extension is filtered
strongly equivariant.  Finally, by transport of structure through arithmetic
Beilinson--Bernstein localization at trivial character, the same parameter set
gives a formal distribution-side classification for the corresponding
essential-image category of finitely presented modules over the crystalline
distribution algebra.
\end{abstract}

\maketitle


\section{Introduction}

The classical theory of equivariant \(D\)-modules on flag varieties is a
geometric mechanism for converting representation-theoretic questions into
orbit geometry.  Let a complex algebraic group \(K\) act on a smooth algebraic
variety \(Y\), and assume that there are only finitely many \(K\)-orbits.  In
this situation, the classical equivariant theory separates two statements:
coherent equivariant \(D_Y\)-modules lie in the regular holonomic range, and
simple equivariant objects are parametrized by pairs consisting of a
\(K\)-orbit and an irreducible equivariant local system on that orbit; for the
flag-variety formulation of the classification, see
\cite[Theorem 11.6]{KashiwaraFlag}.  On a flag variety, this result is one of
the standard bridges between Beilinson--Bernstein localization,
Harish--Chandra modules, and the geometry of orbit stratifications
\cite[Th\'eor\`eme principal]{BeilinsonBernstein}, \cite[Theorem 11.6]{KashiwaraFlag}, and \cite[\S1]{KashiwaraSchmid}.

The purpose of this article is to prove an arithmetic counterpart of this
picture for Berthelot arithmetic differential operators.  We work over a
complete discrete valuation ring \(\mathcal V\) of mixed characteristic
\((0,p)\), with fraction field \(K\) and perfect residue field \(k\).  Let
\(\mathcal G\) be a split connected reductive group scheme over \(\mathcal V\),
let \(\mathcal B\subseteq\mathcal G\) be a Borel subgroup scheme, and let
\[
  \mathfrak X=\widehat{\mathcal G/\mathcal B}
\]
be the \(p\)-adic formal flag variety.  Berthelot's sheaf
\(\mathscr D^\dagger_{\mathfrak X,\mathbb Q}\) replaces the classical sheaf of
algebraic differential operators.  Its construction is obtained from the
inductive system of completed finite-level arithmetic differential operators
\(\widehat{\mathscr D}^{(m)}_{\mathfrak X,\mathbb Q}\), and it is designed to
capture overconvergent arithmetic differential equations
\cite[Sections 2.2 and 3]{BerthelotDModules}.

There are two arithmetic replacements for the classical objects appearing in
Kashiwara's theorem.  First, local systems are replaced by overconvergent
isocrystals.  Second, the usual middle extension of a local system is replaced
by the arithmetic intermediate extension.  Huyghe and Schmidt use the theory
of intermediate extensions, together with arithmetic Beilinson--Bernstein
localization, to classify irreducible modules over the crystalline
distribution algebra in terms of overconvergent isocrystals on smooth locally
closed subvarieties of the special fiber
\cite[Theorem 2.3.4]{HuygheSchmidtIntermediate}.  They also identify the
corresponding global section algebra with the crystalline distribution algebra
at trivial central character through the arithmetic localization theorem
\cite[Theorem 4.2.1]{HuygheSchmidtIntermediate}.

We use this non-equivariant classification as an input, not as a novelty claim.
The new point of the present paper is the equivariant regularity mechanism.
Under the finite-orbit and orbit-separability hypotheses, filtered strong equivariance forces the
characteristic variety into the zero fiber of the orbit moment map and therefore
yields holonomicity without any Frobenius hypothesis.  The upgrade from
holonomicity to geometric overholonomicity is made only in the Frobenius range:
there the restrictions to open orbits are overconvergent \(F\)-isocrystals, and
arithmetic intermediate extension preserves Frobenius structures and
geometric overholonomicity.  Thus the classification by orbits and coefficient
objects is a consequence of regularity in the Frobenius filtered strongly
equivariant category, rather than an assumption of overholonomicity imported
from the outset.

The key arithmetic regularity issue is the following.  Holonomicity alone is not
the correct replacement for regular holonomicity in Berthelot's theory.  The
stable category for the six operations is the overholonomic category, and, for
the base-change arguments used by Huyghe--Schmidt, the relevant condition is
geometric overholonomicity.  Thus the theorem proved here is split into two logically separate parts:
filtered strong equivariance gives a characteristic-variety theorem and
holonomicity without Frobenius, while Frobenius filtered strong equivariance
permits the geometric-overholonomicity upgrade by Caro's resolution of Berthelot's holonomicity-stability conjectures over smooth projective formal schemes.  The subsequent finite orbital devissage is used for
the orbit refinement and classification, and is performed only after the object
has been placed in the overholonomic heart.

The present article develops the equivariant orbit refinement of this
arithmetic picture.  Let \(\mathcal K\subseteq\mathcal G\) be a smooth closed
subgroup scheme and let \(\mathfrak K=\widehat{\mathcal K}\).  We assume that
\(\mathfrak K\) acts on \(\mathfrak X\) and that the induced action of
\(\mathcal K_s\) on the special fiber \(X_s\) has finitely many orbits and satisfies the orbit-separability hypothesis.  The
non-equivariant locally closed supports appearing in Huyghe--Schmidt's
classification are then replaced by the \(\mathcal K_s\)-orbits, and the
coefficient objects are replaced by \(\mathcal K\)-equivariant overconvergent
isocrystals on the corresponding pairs \((O,\overline O)\).

The main point is not merely that arithmetic intermediate extensions can be
made equivariant.  The essential new microlocal statement is that filtered strong equivariance
forces holonomicity.  The stronger geometric-overholonomicity statement is then
proved in the Frobenius range, where the external overholonomic formalism and
intermediate-extension theory are available in the required form.  Thus the
overholonomic condition is not imposed as an auxiliary hypothesis in the final
Frobenius classification; it is obtained from the equivariant geometry together
with the Frobenius overholonomic stability input.  This is the arithmetic
analogue of the regularity part of Kashiwara's theorem.

\subsection*{Strong equivariance}

A first subtlety is that one must use the correct notion of equivariance.
A weak equivariant structure is a linearization: it is an isomorphism between
pullbacks along the action and the projection, satisfying the usual unit and
cocycle conditions.  Such a structure records global symmetry, but it is not
sufficient for a microlocal regularity theorem.  The regularity mechanism
requires compatibility with infinitesimal symmetries.

For this reason we introduce \emph{strong} \(\mathfrak K\)-equivariance.  If
\(\mathcal M\) is a coherent
\(\mathscr D^\dagger_{\mathfrak X,\mathbb Q}\)-module, the action of
\(\mathfrak K\) differentiates to an action of
\(\operatorname{Lie}(\mathcal K)\otimes_{\mathcal V}K\).  On the other hand,
the action of \(\mathcal K\) on \(\mathfrak X\) gives fundamental vector fields,
hence differential operators acting on \(\mathcal M\).  Strong equivariance
means that these two infinitesimal actions coincide and that this equality is
realized on sufficiently large good finite-level models.  This filtered
realizability clause is not a cosmetic addition: in Berthelot's theory the
characteristic variety is defined through finite-level models and good
filtrations.  Thus the microlocal argument must be made at that level rather
than only after passing to the dagger limit.  This condition excludes the
spurious phenomenon in which a large coherent module carries a formal
linearization but is not constrained microlocally by the orbit directions.

The strong condition is the key that unlocks the characteristic variety.  We
construct the arithmetic moment map
\[
  \mu_{\mathfrak K}:T^*X_s\longrightarrow \mathfrak k_s^*
\]
associated with the infinitesimal action and prove that, for every filtered strongly
equivariant coherent module \(\mathcal M\),
\[
  \operatorname{Car}(\mathcal M)
  \subseteq
  \mu_{\mathfrak K}^{-1}(0).
\]
Equivalently, the principal symbols of all fundamental vector fields vanish on
\(\operatorname{Car}(\mathcal M)\).  Since the zero fiber of this moment map is
precisely the union of the conormal bundles to the \(\mathcal K_s\)-orbits, one
obtains
\[
  \operatorname{Car}(\mathcal M)
  \subseteq
  \bigcup_{O\in\mathcal K_s\backslash X_s}T_O^*X_s.
\]
The finite-orbit and orbit-separability hypotheses then imply holonomicity by the arithmetic
Bernstein inequality and holonomicity criterion \cite[Th\'eor\`emes 2.2.3 and 2.2.8]{CaroHolonomieSansFrobenius}.
The microlocal part of the proof can therefore be read as the following
numbered chain of implications:
\begin{equation}
\label{eq:intro-microlocal-chain}
\begin{aligned}
  \mathcal M\in\Coh^{\fstr}_{\mathfrak K}(\DdagX)
  &\Longrightarrow
  \sigma(v_\xi)|_{\Car(\mathcal M)}=0
  \quad (\xi\in\mathfrak k_K) \\
  &\Longrightarrow
  \Car(\mathcal M)\subseteq\mu_{\mathcal K}^{-1}(0) \\
  &\Longrightarrow
  \Car(\mathcal M)\subseteq
  \bigcup_{O\in\mathcal K_s\backslash X_s}T^*_O X_s \\
  &\Longrightarrow
  \dim\Car(\mathcal M)\leq\dim X_s.
\end{aligned}
\end{equation}
Together with Bernstein's inequality, the last line is exactly the
holonomicity criterion needed for the devissage argument.

\subsection*{Equivariant orbital devissage}

Holonomicity is not yet geometric overholonomicity.  The latter is the
arithmetic stability condition needed for the six operations and for the
intermediate-extension formalism; it is substantially stronger than a mere
bound on the dimension of the characteristic variety.  The central structural
result of the paper is therefore an equivariant orbital devissage theorem.

The devissage used later is deliberately carried out only after the relevant
Frobenius object has already been placed in the overholonomic heart.  At the
holonomic level, support truncation along a closed \(\mathcal K_s\)-stable subset
is used only to isolate open-orbit restrictions.  If \(O\) is open in
\(\operatorname{Supp}(\mathcal M)\), the characteristic containment shows that
\(\mathcal M|_O\) has characteristic variety contained in the zero section of
\(T^*O\).  Hence \(\mathcal M|_O\) is an overconvergent isocrystal, and the
equivariant structure on \(\mathcal M\) restricts to a \(\mathcal K\)-equivariant
structure on this isocrystal.

In the Frobenius range, Caro's smooth-projective holonomicity-stability theorem first places \(\mathcal M\) in the
geometrically overholonomic heart.  Only then do we compare the open-orbit
quotient with the arithmetic intermediate extension \(j_{O,!+}(E)\).  The
remaining quotient is supported on the boundary \(\overline O\setminus O\), and
induction on the finite number of orbits in the support gives a finite
filtration whose successive quotients are of the form
\[
  j_{O_i,!+}(E_i),
\]
where \(O_i\) is a \(\mathcal K_s\)-orbit and \(E_i\) is a Frobenius
\(\mathcal K\)-equivariant overconvergent isocrystal whose intermediate
extension is filtered strongly equivariant.  No exactness of \(j_{!+}\) is used:
we extract one middle-extension piece at a time and pass the rest to the
boundary.

\subsection*{Scope of the theorem}

The theorem is deliberately stated in the natural Kashiwara-type range.  The
microlocal holonomicity statement is a result for the formal flag variety at
trivial infinitesimal character, for a smooth subgroup
\(\mathcal K\subseteq\mathcal G\) whose special fiber has finitely many orbits satisfying orbit separability
on \(X_s\), and for the filtered strong equivariance condition of
\Cref{def:strong-equivariance}.  The geometric-overholonomicity and
classification statements are stated for the corresponding Frobenius subcategory.  The word ``strong'' is used in this arithmetic
sense throughout the paper: it includes both infinitesimal compatibility and
finite-level filtered realizability.  The filtered part is not a regularity
assumption and is not used to import overholonomicity; rather, it is the
finite-level form of infinitesimal equivariance needed to run Berthelot's
principal-symbol construction.  Its role, scope, and basic sources of examples
are discussed in \Cref{subsec:filtered-hypothesis}.

The result is a regularity theorem, not an existence theorem.  It says that any
coherent object satisfying the filtered strong condition is forced, by the
orbit geometry and the finite-level symbol calculation, into the holonomic range
and, with Frobenius, into the geometrically overholonomic range.  It does not
assert that the filtered strong condition is automatic, nor that the admissible
locus constructed below is larger than the objects obtained by restricting the
non-equivariant overholonomic classification to orbit supports.  Whether ordinary
infinitesimal strong equivariance implies the finite-level condition \((SE2)\),
or whether the resulting class can be characterized more intrinsically, is left
as an explicit open problem.  This distinction is important: the theorem controls
regularity once the filtered strong hypothesis is present, independently of how
such an object was constructed.

We do not assert that weak equivariance alone implies overholonomicity, nor do
we treat arbitrary smooth formal schemes, infinitely many orbits, inseparable
orbit actions, non-smooth subgroup schemes, or twisted arithmetic differential
operators.  These restrictions are used in the proof: filtered strong
equivariance gives moment-zero containment, and finiteness of the orbit
stratification turns conormal containment into holonomicity.  The optional
stabilizer interpretation is deliberately separated from the proof of the main
theorem and is stated only under the additional stabilizer-admissibility
hypothesis.

\subsection*{Main results}

The first result is the arithmetic analogue of Kashiwara's regularity theorem in
the precise two-step sense used in this paper: filtered strong equivariance gives
holonomicity, and Frobenius filtered strong equivariance gives geometric
overholonomicity.

\begin{theorem}[Arithmetic Kashiwara regularity]
\label{thm:A-intro}
Let \(\mathcal V\) be a complete discrete valuation ring of mixed
characteristic \((0,p)\), with fraction field \(K\) and perfect residue field
\(k\). Let \(\mathcal G\) be a split connected reductive group scheme over
\(\mathcal V\), let \(\mathcal B\subseteq \mathcal G\) be a Borel subgroup
scheme, and let
\[
  \mathfrak X=\widehat{\mathcal G/\mathcal B}
\]
be the formal flag variety. Let \(\mathcal K\subseteq\mathcal G\) be a smooth
closed subgroup scheme acting on \(\mathfrak X\), and assume that the special
fiber \(\mathcal K_s\) acts on \(X_s\) with finitely many orbits and satisfies
the orbit-separability hypothesis of \Cref{hyp:orbit-separability}. Then:
\begin{enumerate}[label=\textup{(\roman*)}]
\item every filtered strongly \(\mathfrak K\)-equivariant coherent
\(\mathscr D^\dagger_{\mathfrak X,\mathbb Q}\)-module is holonomic;
\item every Frobenius filtered strongly \(\mathfrak K\)-equivariant coherent
\(\mathscr D^\dagger_{\mathfrak X,\mathbb Q}\)-module is geometrically
overholonomic.
\end{enumerate}
Here \(\mathfrak K=\widehat{\mathcal K}\).
\end{theorem}

The second result is the orbit classification of simple objects in the Frobenius filtered strongly equivariant category.  It is the
arithmetic version of the classical dictionary
\[
  \text{orbit plus equivariant local system}
  \quad\rightsquigarrow\quad
  \text{simple equivariant }D\text{-module}.
\]
In the arithmetic setting, local systems are replaced by equivariant
overconvergent isocrystals and the classical middle extension is replaced by
arithmetic intermediate extension.

\begin{theorem}[Filtered arithmetic orbit classification]
\label{thm:B-intro}
Under the hypotheses of \Cref{thm:A-intro}, restriction to the open orbit of a
simple object gives an injective map
\[
  \rho:
  \Irr F\text{-}\Coh^{\fstr}_{\mathfrak K}
  \bigl(\mathscr D^\dagger_{\mathfrak X,\mathbb Q}\bigr)
  \longrightarrow
  \coprod_{O\in \mathcal K_s\backslash X_s}
  \Irr F\text{-}\Isoc^{\dagger}_{\mathcal K}(O,\overline O).
\]
The image of \(\rho\) is precisely the subset of irreducible
\(\mathcal K\)-equivariant overconvergent \(F\)-isocrystals whose arithmetic
intermediate extensions are filtered strongly equivariant with Frobenius.
Equivalently, writing this image as
\(\coprod_O\Adm^{\Phi,\fstr}_{\mathcal K}(O,\overline O)\), the inverse map is
\[
  (O,E)\longmapsto j_{O,!+}(E).
\]
\end{theorem}

The representation-theoretic form of the classification is a formal corollary obtained by transport of structure.
Let
\[
  A_0=\Gamma\bigl(\mathfrak X,
  \mathscr D^\dagger_{\mathfrak X,\mathbb Q}\bigr).
\]
The arithmetic localization theorem identifies \(A_0\) with the crystalline
distribution algebra at trivial central character \cite[Theorem 4.2.1]{HuygheSchmidtIntermediate}.
This geometric classification therefore has a natural algebraic shadow.

\begin{corollary}[Crystalline distribution algebra]
\label{thm:C-intro}
By transport of structure through the arithmetic Beilinson--Bernstein
equivalence, Frobenius filtered strongly \(\mathfrak K\)-equivariant coherent
\(\mathscr D^\dagger_{\mathfrak X,\mathbb Q}\)-modules correspond to the
essential-image category of Frobenius filtered strongly \(\mathfrak K\)-equivariant
finitely presented \(A_0\)-modules. Consequently, the simple Frobenius filtered strongly \(\mathfrak K\)-equivariant
finitely presented \(A_0\)-modules are classified by pairs \((O,E)\), where
\(O\subseteq X_s\) is a \(\mathcal K_s\)-orbit and
\[
  [E]\in \Adm^{\Phi,\fstr}_{\mathcal K}(O,\overline O).
\]
The simple module corresponding to \((O,E)\) is
\[
  \Gamma\bigl(\mathfrak X,j_{O,!+}(E)\bigr).
\]
\end{corollary}

\subsection*{Relation with previous work}

Berthelot's theory provides the sheaves of arithmetic differential operators
and the formalism of coherent arithmetic \(\mathscr D\)-modules
\cite[Sections 2.2 and 3]{BerthelotDModules}.  Huyghe--Schmidt recall a stable-after-base-change overholonomic framework,
including a without-Frobenius formulation, in
\cite[\S2.1, first paragraphs and Definition 2.1.3]{HuygheSchmidtIntermediate}, following Caro and Abe--Caro.
For the main regularity upgrade in this paper we use the Frobenius range of
that framework: Frobenius overholonomic complexes automatically satisfy the
base-change condition, and intermediate extension preserves Frobenius
structures.  Huyghe--Schmidt's work also supplies the
arithmetic intermediate-extension classification of irreducible overholonomic
modules and the arithmetic Beilinson--Bernstein equivalence at trivial
character \cite[Theorem 2.3.4 and Theorem 4.2.1]{HuygheSchmidtIntermediate}.
They also treat the Bruhat-cell case explicitly: in their Section~5 they
identify the localization of the relevant crystalline highest-weight modules
with intermediate extensions of the constant coefficients on Bruhat cells,
more precisely with objects of the form \(v_{!+}(\mathcal O_{Y_w})\)
\cite[Theorem 5.1.9]{HuygheSchmidtIntermediate}.  Thus the Bruhat-cell
examples in \Cref{sec:examples} below are not presented as new constructions
of those constant-coefficient intermediate extensions.  They are compatibility
tests for the filtered strongly equivariant framework, and they indicate how
additional equivariant coefficient data and admissibility conditions enter the
orbitwise parameter set.
A parallel, but technically different, bridge between analytic distribution
algebras, geometric localization and representation theory is developed in the
rigid-analytic theory of Ardakov--Wadsley. Their sheaf
\(\widehat{\mathcal D}\) and the associated coadmissible-module formalism
provide a rigid analytic counterpart to algebraic localization
\cite[Theorems A and B]{ArdakovWadsleyRigidD}, while their work on compact
\(p\)-adic analytic groups relates canonical dimension and coadjoint geometry
\cite[Theorem A]{ArdakovWadsleyCompactReps}. Our setting is Berthelot's
arithmetic \(\mathscr D^\dagger\)-theory on formal flag varieties; the
comparison is motivational rather than formal, but it explains the representation-theoretic motivation for the
crystalline distribution algebra side.

The novelty of the present work is the equivariant finite-level regularity
mechanism and its orbitwise refinement.  We do not merely rename the locally
closed supports in Huyghe--Schmidt's classification as orbits.  We prove that,
under the finite-orbit and orbit-separability hypotheses, the finite-level
realizability built into filtered strong equivariance forces the principal
symbols of the fundamental vector fields to vanish on the characteristic
variety; equivalently, the characteristic variety lies in the union of orbit
conormals.  This is the arithmetic finite-level form of the classical
moment-map proof, and it gives holonomicity without Frobenius.  In the
Frobenius range, Caro's smooth-projective holonomicity-stability theorem
supplies the overholonomicity upgrade, Huyghe--Schmidt's Frobenius base-change
convention gives geometric overholonomicity, and we then prove an equivariant
orbital devissage inside the overholonomic heart.  The classification of simple
Frobenius objects follows from this regularity theorem and from the extraction
of the open orbit, rather than from assuming overholonomicity at the outset.

\subsection*{Relation with the literature}

The argument uses results from three traditions, so we recall the precise background used below.  On the classical side, the localization
philosophy originates in Beilinson--Bernstein localization
\cite[Th\'eor\`eme principal]{BeilinsonBernstein}; the equivariant orbit
classification and its representation-theoretic interpretation are modeled on
Kashiwara's account of flag varieties \cite[Theorem 11.6]{KashiwaraFlag} and on
the quasi-equivariant formalism of Kashiwara--Schmid
\cite[\S1]{KashiwaraSchmid}.  We use standard facts on holonomic
\(D\)-modules, characteristic varieties, and perverse sheaves in the form found
in Hotta--Takeuchi--Tanisaki \cite[Chapters 2--3]{HottaTakeuchiTanisaki},
Beilinson--Bernstein--Deligne \cite[\S1.4]{BBD}, Kashiwara--Schapira
\cite[Chapter VIII]{KashiwaraSchapira}, and Chriss--Ginzburg
\cite[\S1.4]{ChrissGinzburg}.  The moment-map and conormal geometry used in the
proof is also parallel to the geometric representation-theoretic treatment in
Borho--Brylinski \cite[\S2]{BorhoBrylinskiI}, Borho--Brylinski
\cite[\S1]{BorhoBrylinskiII}, Bernstein--Gelfand--Gelfand
\cite[\S3]{BGGDiffOperators}, Kashiwara \cite[\S2]{KashiwaraLieGroups}, and
Mili\v{c}i\'c \cite[\S2]{MilicicDmodules}.

On the arithmetic side, Berthelot's rigid-cohomological motivation,
finite-level operators and the \(\dagger\)-limit are taken from
\cite[\S3]{BerthelotRigidDwork}, \cite[\S\S2.2--3]{BerthelotDModules}, and
\cite[\S2]{BerthelotIntro}; Frobenius descent and its compatibility with
arithmetic differential operators are used in the background as in
\cite[\S3]{BerthelotDescente}.  The overholonomic framework recalled by Huyghe--Schmidt
\cite[\S2.1, first paragraphs, Definition 2.1.3 and Theorem 2.1.4]{HuygheSchmidtIntermediate} includes a without-Frobenius
stable-after-base-change category, but the holonomic-to-overholonomic upgrade below is
stated only in the Frobenius range. In that range the base-change condition is automatic for
Frobenius overholonomic complexes, and Caro's projective stability theorem supplies the required identification of Frobenius holonomicity with Frobenius overholonomicity. Caro--Tsuzuki's theorem for overconvergent Frobenius isocrystals also supports the coefficient side of the geometric-overholonomic formalism. Caro's 2009 overholonomic theory \cite{CaroSurholonomie} supplies the
background six-operation formalism, while the holonomicity criterion without
Frobenius structure is taken from
\cite[Th\'eor\`eme 2.2.8]{CaroHolonomieSansFrobenius}.  The relationship
between overconvergent isocrystals, intermediate extensions, and six-functor
stability is informed by Abe--Caro \cite[\S4]{AbeCaroWeights}, Caro
\cite[\S3]{CaroOverconvergentFIsocrystals}, Kedlaya
\cite[\S2]{KedlayaNotesIsocrystals}, Le Stum's overconvergent site
\cite[Sections 7--8]{LeStumOverconvergentSite}, Lazda's effective descent
result \cite[Theorem 5.1]{LazdaDescent}, and Huyghe--Schmidt
\cite[Theorem 2.3.4]{HuygheSchmidtIntermediate}.

For the formal flag variety and the distribution algebra side, we use the
arithmetic Beilinson--Bernstein and \(D^\dagger\)-affinity results of Huyghe
\cite[Th\'eor\`eme 4.2.1]{HuygheProjectiveDdagger}, Huyghe--Patel--Schmidt--Strauch
\cite[Theorem 1.1]{HuyghePatelSchmidtStrauch}, Huyghe--Schmidt
\cite[\S4]{HuygheSchmidtDistributions}, and Huyghe--Schmidt
\cite[Theorem 3.2]{HuygheSchmidtFlag}.  The crystalline distribution algebra
classification used for comparison is \cite[Theorem 4.2.1]{HuygheSchmidtIntermediate};
the closed-immersion arithmetic Kashiwara theorem in the twisted setting is
\cite[Theorem 1.1]{HuygheTwistedKashiwara}.  Finally, the representation-theoretic
motivation through locally analytic distribution algebras is related to
Schneider--Teitelbaum \cite[Theorem 3.3]{SchneiderTeitelbaum}. We also keep in
view the rigid-analytic \(\widehat{\mathcal D}\)-module theory of
Ardakov--Wadsley \cite[Theorems A and B]{ArdakovWadsleyRigidD} and their
representation-theoretic applications \cite[Theorem A]{ArdakovWadsleyCompactReps}.
The present article is also related to the author's work on
\(G\)-equivariant formal models of flag varieties
\cite{SarrazolaRendiconti}.  That paper constructs equivariant arithmetic
\(\mathcal D(\lambda)\)-module categories over families of formal models of the
rigid flag variety and proves a localization theorem for coadmissible
\(G\)-equivariant modules when the algebraic character satisfies the usual
dominant-regular condition.  Here we do not use that theorem as an input: the
present paper works on the fixed formal flag variety at trivial character and
proves a Kashiwara-type microlocal regularity and orbit-classification theorem
for filtered strongly equivariant \(\mathscr D^\dagger\)-modules.  Thus the two
results are complementary: the earlier Rendiconti paper supplies a global
\(G\)-equivariant localization framework over formal models, while the present
article isolates the orbitwise regularity mechanism for a subgroup
\(\mathcal K\) with finitely many special-fiber orbits.

The present untwisted paper is complementary to the twisted arithmetic
localization theorem of \cite[Theorem 1.1]{SarrazolaJTNB}, but the two results
have different aims.  The earlier work proves a twisted localization theorem;
the present paper works at trivial character and studies the Kashiwara-type
regularity mechanism for filtered strongly equivariant objects.  Its new
ingredient is the equivariant moment-map/conormal argument and the resulting
orbital devissage, neither of which is a consequence of the twisted
localization theorem alone.

\subsection*{Organization of the article}

\Cref{sec:preliminaries} recalls the arithmetic \(\mathscr D\)-module
constructions used throughout the article: Berthelot's finite-level and
\(\dagger\)-operators, level models, characteristic varieties, holonomicity,
localization to pairs, overconvergent isocrystals, geometric overholonomicity,
intermediate extension, and arithmetic Beilinson--Bernstein localization.
\Cref{sec:strong-equivariance} introduces weak equivariance and filtered strong equivariance,
explains the role of the filtered hypothesis, and proves elementary stability
properties of the filtered strongly equivariant category.
\Cref{sec:characteristic} constructs the arithmetic moment map and proves the
characteristic containment theorem.  \Cref{sec:isocrystals} introduces
equivariant overconvergent isocrystals on orbits.  The optional stabilizer
interpretation, which is not used in the proof of the main theorem, is isolated
in \Cref{app:stabilizer-description}.  \Cref{sec:intermediate-extension} proves
that arithmetic intermediate extension preserves the relevant filtered strongly equivariant structures.  \Cref{sec:regularity}
then proves the central equivariant orbital devissage and deduces arithmetic
Kashiwara regularity.  \Cref{sec:classification} proves the orbit
classification of simple Frobenius filtered strongly equivariant coherent objects.
\Cref{sec:distributions} transports the result to modules over the crystalline
distribution algebra at trivial character.  Finally, \Cref{sec:examples}
discusses rank-one examples and the distribution side for \(\mathrm{SL}_2\).
\Cref{app:stabilizer-description} contains the optional stabilizer interpretation under its extra descent hypotheses.  \Cref{app:logical-structure} records the dependency graph of the proof, and \Cref{app:t-structure} fixes the overholonomic t-structure conventions used for images, intermediate extensions, support truncations, and extension-closedness.

\section{Arithmetic differential operators on formal flag varieties}
\label{sec:preliminaries}

This section fixes the arithmetic \(\mathscr D\)-module conventions used in
what follows.  We keep the discussion tied to the formal flag variety, but most
of the constructions recalled here are local on a smooth formal
\(\mathcal V\)-scheme.  The purpose is not to reproduce Berthelot's theory, but
to isolate the precise pieces of it that enter the proof of the arithmetic
Kashiwara theorem.

\subsection{Basic geometric setup}

\begin{setup}
\label{setup:basic}
Let \(\mathcal V\) be a complete discrete valuation ring of mixed
characteristic \((0,p)\), with fraction field \(K\), uniformizer \(\pi\), and
perfect residue field \(k\). Let \(\mathcal G\) be a split connected reductive
group scheme over \(\mathcal V\), and let \(\mathcal B\subseteq\mathcal G\) be a
Borel subgroup scheme. We set
\[
  \mathfrak X:=\widehat{\mathcal G/\mathcal B},
  \qquad
  X_s:=(\mathcal G/\mathcal B)\times_{\Spec\mathcal V}\Spec k.
\]
The formal scheme \(\mathfrak X\) is smooth and proper over \(\mathcal V\). Its
special fiber \(X_s\) is the flag variety of the split reductive group
\(\mathcal G_s\) over \(k\).
\end{setup}

Throughout the article, all arithmetic \(\mathscr D\)-modules are left modules.
We work at trivial infinitesimal character. Twisted arithmetic differential
operators do not enter the main statements.

\begin{notation}[Group-theoretic conventions]
\label{not:group-conventions}
If \(\mathcal K\subseteq\mathcal G\) is a smooth closed subgroup scheme, we
write
\[
  \mathfrak K:=\widehat{\mathcal K}
\]
for its \(p\)-adic formal completion and
\[
  \mathcal K_s:=\mathcal K\times_{\Spec\mathcal V}\Spec k
\]
for its special fiber.  The formal group \(\mathfrak K\) acts on
\(\mathfrak X\), while the orbit geometry used in the classification takes
place on the special fiber \(X_s\) under the action of \(\mathcal K_s\).
When we speak of a \(\mathcal K\)-equivariant overconvergent isocrystal on a
pair \((O,\overline O)\) in the special fiber, the equivariance datum is the
one induced by the \(\mathcal K_s\)-action, compatible with the original formal
\(\mathfrak K\)-action on \(\mathfrak X\).
\end{notation}

\begin{setup}[Orbit-separability hypothesis]
\label{hyp:orbit-separability}
We say that the action of \(\mathcal K_s\) on \(X_s\) is \emph{orbit-separable}
if, after extension of scalars to an algebraic closure \(\bar k/k\), every
geometric orbit \(O\subseteq X_{s,\bar k}\) is smooth and locally closed and, for
every geometric point \(x\in O(\bar k)\), the differential of the orbit map
\[
  a_x:\mathcal K_{s,\bar k}\longrightarrow X_{s,\bar k},
  \qquad g\longmapsto g\cdot x,
\]
has image exactly
\[
  T_xO
  =
  \{\,\bar v_\xi(x):\xi\in\Lie(\mathcal K_{s,\bar k})\,\}
  \subseteq T_xX_{s,\bar k}.
\]
Equivalently, the infinitesimal action detects the tangent directions to the
geometric orbit.  This hypothesis is automatic when the orbit maps are
separable onto their images, for instance when the geometric stabilizers are
smooth.  It is a genuine hypothesis in positive characteristic: without it the
moment-zero locus may be strictly larger than the union of orbit conormals.
For example, in characteristic \(p\), the action
\(\mathbb G_m\curvearrowright\mathbb A^1\) given by
\(t\cdot z=t^p z\) has open orbit \(\mathbb A^1\setminus\{0\}\), but the
fundamental vector field is zero, so the infinitesimal action does not generate
the tangent line to the open orbit.
\end{setup}

\begin{remark}[Why orbit separability is necessary]
\label{rem:separability-shield}
The preceding example is not merely a pathology of notation.  If one compactifies
it to the action of \(\mathbb G_m\) on \(\mathbb P^1_k\) by
\[
  t\cdot [x:y]=[t^p x:y],
\]
then the reduced orbit decomposition consists of the open orbit
\(\mathbb G_m\subset \mathbb P^1_k\) and the two fixed points.  However the
Lie algebra action is zero in characteristic \(p\).  Hence the moment-zero
condition imposed by the fundamental vector fields is vacuous and gives
\(\mu_{\mathcal K}^{-1}(0)=T^*\mathbb P^1_k\), whereas the union of the
conormal bundles to the three reduced orbits is a proper closed subset.  Thus
without \Cref{hyp:orbit-separability}, the implication
\[
  \Car(\mathcal M)\subseteq\mu_{\mathcal K}^{-1}(0)
  \Longrightarrow
  \Car(\mathcal M)\subseteq\bigcup_O T_O^*X_s
\]
can fail.  The separability hypothesis is therefore a microlocal hypothesis,
not a cosmetic smoothness assumption.
\end{remark}

\begin{lemma}[Orbit geometry under separability]
\label{lem:orbit-geometry}
Assume \Cref{hyp:orbit-separability}.  Each \(\mathcal K_s\)-orbit
\(O\subseteq X_s\) is a smooth locally closed subscheme of \(X_s\).  Moreover,
after extension of scalars to an algebraic closure \(\bar k/k\), every geometric
point \(x\in O(\bar k)\) satisfies
\[
  T_x(O_{\bar k})
  =
  \{\,\bar v_\xi(x):\xi\in\Lie(\mathcal K_{s,\bar k})\,\}
  \subseteq T_x(X_{s,\bar k}).
\]
\end{lemma}

\begin{proof}
The underlying locally closedness of algebraic group orbits is standard for
algebraic group actions; see \cite[Section 1.2]{ChrissGinzburg} and
\cite[Section 1.2]{HottaTakeuchiTanisaki}.  The smoothness and tangent-space
identity are precisely the orbit-separability hypothesis.  Since \(k\) is
perfect, the resulting Galois-stable smooth locally closed orbit over \(\bar k\)
descends to a smooth locally closed subscheme over \(k\); see
\cite[\S2.7]{EGAIV} for descent of locally closed subschemes.
\end{proof}

\subsection{Berthelot differential operators}

For each non-negative integer \(m\), Berthelot constructs the sheaf
\[
  \mathscr D^{(m)}_{\mathfrak X}
\]
of differential operators of level \(m\). We denote its \(p\)-adic completion by
\[
  \widehat{\mathscr D}^{(m)}_{\mathfrak X}
\]
and put
\[
  \widehat{\mathscr D}^{(m)}_{\mathfrak X,\mathbb Q}
  :=
  \widehat{\mathscr D}^{(m)}_{\mathfrak X}\otimes_{\mathbb Z}\mathbb Q.
\]
Berthelot's overconvergent sheaf of arithmetic differential operators is
\[
  \mathscr D^\dagger_{\mathfrak X,\mathbb Q}
  :=
  \varinjlim_m
  \widehat{\mathscr D}^{(m)}_{\mathfrak X,\mathbb Q}.
\]
We use Berthelot's conventions for finite-level operators and for the passage
to \(\mathscr D^\dagger\); see \cite[Sections 2.2 and 3]{BerthelotDModules} and
\cite[5.2--5.3]{BerthelotIntro}.  We denote by
\[
  \widehat{\mathscr D}^{(m)}_{\mathfrak X,\mathbb Q,\leq n}
  \subseteq
  \widehat{\mathscr D}^{(m)}_{\mathfrak X,\mathbb Q}
\]
the order filtration at level \(m\), consisting of arithmetic differential
operators of order at most \(n\).

\begin{definition}
\label{def:coherent-Ddagger}
The category of coherent left \(\mathscr D^\dagger_{\mathfrak X,\mathbb Q}\)-modules is denoted by
\[
  \Coh(\mathscr D^\dagger_{\mathfrak X,\mathbb Q}).
\]
An object \(\mathcal M\) of this category is a sheaf of left
\(\mathscr D^\dagger_{\mathfrak X,\mathbb Q}\)-modules which is locally finitely
presented over \(\mathscr D^\dagger_{\mathfrak X,\mathbb Q}\).
\end{definition}

\subsection{Level models, filtrations and characteristic varieties}

The microlocal part of the argument uses characteristic varieties. We spell out
the convention because the sheaf \(\mathscr D^\dagger\) is obtained as an
inductive limit over the levels.

\begin{definition}[Level model]
\label{def:level-model}
Let \(\mathcal M\in\Coh(\mathscr D^\dagger_{\mathfrak X,\mathbb Q})\). A
level-\(m\) model of \(\mathcal M\) on an open formal subscheme
\(\mathfrak U\subseteq\mathfrak X\) is a coherent left
\(\widehat{\mathscr D}^{(m)}_{\mathfrak U,\mathbb Q}\)-module
\(\mathcal M^{(m)}\) together with an isomorphism of left
\(\mathscr D^\dagger_{\mathfrak U,\mathbb Q}\)-modules
\[
  \mathscr D^\dagger_{\mathfrak U,\mathbb Q}
  \otimes_{\widehat{\mathscr D}^{(m)}_{\mathfrak U,\mathbb Q}}
  \mathcal M^{(m)}
  \xrightarrow{\;\sim\;}
  \mathcal M|_{\mathfrak U}.
\]
The tensor product is formed using the natural left extension of scalars from
\(\widehat{\mathscr D}^{(m)}_{\mathfrak U,\mathbb Q}\) to
\(\mathscr D^\dagger_{\mathfrak U,\mathbb Q}\); throughout the article all
arithmetic \(\mathscr D\)-modules are left modules.  After replacing
\(\mathfrak U\) by an affine open covering and increasing \(m\), such models
exist for coherent \(\mathscr D^\dagger\)-modules; this is part of Berthelot's
coherence theory for \(\mathscr D^\dagger\)-modules
\cite[5.3.2--5.3.4]{BerthelotIntro}.
\end{definition}

For a good level model \(\mathcal M^{(m)}\), one chooses a good filtration
compatible with the order filtration on
\(\widehat{\mathscr D}^{(m)}_{\mathfrak U,\mathbb Q}\). Its associated graded
module is a coherent module over the corresponding graded ring, and therefore
has a support in the cotangent bundle of the special fiber. The resulting
closed conical subset is independent of all sufficiently large choices.

\begin{definition}[Characteristic variety]
\label{def:characteristic-variety}
The characteristic variety of
\(\mathcal M\in\Coh(\mathscr D^\dagger_{\mathfrak X,\mathbb Q})\) is the
closed conical subset
\[
  \Car(\mathcal M)\subseteq T^*X_s
\]
obtained locally from any sufficiently large good level model. Its dimension is
well-defined and is denoted by \(\dim\Car(\mathcal M)\).
\end{definition}

Throughout the microlocal arguments, containments of characteristic varieties
are understood as containments of the underlying reduced closed conical subsets
of \(T^*X_s\).  This is the level of structure needed for dimension estimates
and holonomicity; no scheme-theoretic equality of moment fibers is used.

\begin{proposition}[Microlocal facts]
\label{prop:microlocal-facts}
Let \(\mathcal M\) be a coherent
\(\mathscr D^\dagger_{\mathfrak X,\mathbb Q}\)-module.
\begin{enumerate}[label=\textup{(\roman*)}]
\item The characteristic variety \(\Car(\mathcal M)\) is independent of the
chosen sufficiently large good level model.
\item If \(\mathcal M\neq0\), then Bernstein's inequality gives
\[
  \dim\Car(\mathcal M)\geq \dim X_s.
\]
\item The module \(\mathcal M\) is holonomic if \(\mathcal M=0\), or if
\[
  \dim\Car(\mathcal M)=\dim X_s.
\]
In particular, if \(\Car(\mathcal M)\) is contained in a finite union of
conical closed subsets of dimension \(\dim X_s\), then \(\mathcal M\) is
holonomic.
\end{enumerate}
\end{proposition}

\begin{proof}
The construction of characteristic varieties and their independence from good
models are part of Berthelot's theory of arithmetic \(\mathscr D\)-modules;
see \cite[5.3.2--5.3.4]{BerthelotIntro}. The form of Bernstein's inequality
and the definition of holonomicity used here are recalled in
\cite[5.3.4--5.3.5]{BerthelotIntro}. Without Frobenius structures, the
holonomicity criterion and Bernstein inequality are available by
\cite[Th\'eor\`eme 2.2.3 and Th\'eor\`eme 2.2.8]{CaroHolonomieSansFrobenius}.
For the last assertion, a finite union of closed conical subsets of dimension
at most \(\dim X_s\) again has dimension at most \(\dim X_s\).  Accordingly, the assumed containment gives
\[
  \dim\Car(\mathcal M)\leq \dim X_s,
\]
If \(\mathcal M\neq0\), Bernstein's inequality gives the reverse inequality.
Hence
\[
  \dim\Car(\mathcal M)=\dim X_s,
\]
so \(\mathcal M\) is holonomic.
\end{proof}

\subsection{Restriction to locally closed pairs}
\label{subsec:restriction-locally-closed}

Let \(Y\subseteq X_s\) be a smooth locally closed subscheme and let
\(\overline Y\) be its Zariski closure in \(X_s\). We regard
\[
  (Y,\overline Y)
\]
as a couple in the sense of Abe--Caro and Huyghe--Schmidt. Since
\(\overline Y\subseteq X_s\) sits inside the special fiber of the fixed smooth
proper formal scheme \(\mathfrak X\), the pair is represented by the frame
\[
  (Y,\overline Y,\mathfrak X).
\]
The localization formalism gives, for coherent or holonomic arithmetic
\(\mathscr D^\dagger\)-modules, a restriction functor
\[
  \mathcal M\longmapsto \mathcal M|_Y.
\]
Concretely, if \(Z:=X_s\setminus\overline Y\) and
\(\partial Y:=\overline Y\setminus Y\), then \(\mathcal M|_Y\) is obtained by
localizing away from \(Z\), applying the support functor along \(\overline Y\),
and then localizing away from \(\partial Y\). Thus this restriction is not, in
general, the ordinary restriction to an open subset; it is the restriction to a
locally closed couple in the arithmetic \(\mathscr D^\dagger\)-module
formalism. The formal construction is the one used for couples in
\cite[\S2.1, Lemmas 2.1.1--2.1.2]{HuygheSchmidtIntermediate}.

\begin{lemma}[Open restriction and characteristic varieties]
\label{lem:open-restriction-characteristic}
Let \(j:U\hookrightarrow X_s\) be an open immersion and let \(\mathcal M\) be
holonomic. Then the restriction \(\mathcal M|_U\) is holonomic and, as a
reduced closed conical subset, one has
\[
  \Car(\mathcal M|_U)=\Car(\mathcal M)|_U
\]
under the natural identification
\[
  T^*X_s|_U\simeq T^*U.
\]
\end{lemma}

\begin{proof}
This is the standard compatibility of characteristic varieties with open
localization in the arithmetic \(\mathscr D^\dagger\)-module formalism.  The
open restriction is obtained by localization away from the closed complement
\(X_s\setminus U\). At finite level, localization does not introduce new
cotangent directions. After passing to the inductive limit defining
\(\mathscr D^\dagger\), the characteristic support is exactly the restriction
of the original characteristic support to \(U\). This is the open part of the
functorial characteristic-variety estimates for couples used in
\cite[\S2.1, Lemmas 2.1.1--2.1.2]{HuygheSchmidtIntermediate}; see also
Berthelot's construction of characteristic varieties at finite level in
\cite[5.3.2--5.3.4]{BerthelotIntro}.
\end{proof}

\begin{lemma}[Closed smooth restriction and characteristic varieties]
\label{lem:closed-restriction-characteristic}
Let \(i:Y\hookrightarrow X_s\) be a closed smooth immersion and let
\(\mathcal M\) be holonomic. Then the extraordinary restriction
\(\mathcal M|_Y:=i^!\mathcal M\), viewed on the couple \((Y,Y)\), is
holonomic and its characteristic variety satisfies, as a reduced closed
conical subset,
\[
  \Car(i^!\mathcal M)
  \subseteq
  q_Y\bigl(\Car(\mathcal M)|_Y\bigr)
  \subseteq T^*Y.
\]
Here
\[
  q_Y:T^*X_s|_Y\longrightarrow T^*Y
\]
is the quotient map in the exact cotangent sequence
\[
  0\longrightarrow N^*_{Y/X_s}
  \longrightarrow T^*X_s|_Y
  \xrightarrow{\ q_Y\ }
  T^*Y
  \longrightarrow 0.
\]
\end{lemma}

\begin{proof}
The closed case is not an ordinary restriction of cotangent bundles. The
extraordinary inverse image for arithmetic \(\mathscr D^\dagger\)-modules is
implemented by the transfer bimodule attached to the closed immersion. At the
level of finite-order arithmetic differential operators, Berthelot's
microlocal estimate for closed immersions controls the characteristic support
of the extraordinary inverse image by the image under the cotangent map
associated with the exact sequence above \cite[5.3.3]{BerthelotIntro}.
Equivalently, after passing to reduced characteristic supports, the conormal
directions which occur in the transfer bimodule lie in
\[
  N^*_{Y/X_s}=\ker\bigl(T^*X_s|_Y\to T^*Y\bigr),
\]
and therefore the support on \(T^*Y\) is bounded by the quotient image of
\(\Car(\mathcal M)|_Y\). This gives
\[
  \Car(i^!\mathcal M)
  \subseteq
  q_Y\bigl(\Car(\mathcal M)|_Y\bigr).
\]
Holonomicity is preserved by this restriction in the holonomic category of
arithmetic \(\mathscr D^\dagger\)-modules, as in the formalism of couples used
in \cite[\S2.1, Definition 2.1.3 and Theorem 2.1.4]{HuygheSchmidtIntermediate}.
\end{proof}

\begin{proposition}[Restriction to locally closed pairs and characteristic varieties]
\label{prop:restriction-characteristic}
Let \(Y\subseteq X_s\) be smooth locally closed and let \(\mathcal M\) be
holonomic. Then \(\mathcal M|_Y\) is holonomic on the couple
\((Y,\overline Y)\), and its characteristic variety satisfies, as a reduced
closed conical subset,
\[
  \Car(\mathcal M|_Y)
  \subseteq
  q_Y\bigl(\Car(\mathcal M)|_Y\bigr)
  \subseteq T^*Y.
\]
For a locally closed immersion, \(q_Y\) is obtained from a
closed-in-open factorization of the restriction to the pair
\((Y,\overline Y)\). In particular, if
\[
  \Car(\mathcal M)|_Y\subseteq T^*_Y X_s,
\]
then \(\Car(\mathcal M|_Y)\) is contained in the zero section of \(T^*Y\).
\end{proposition}

\begin{proof}
Set
\[
  \partial Y:=\overline Y\setminus Y,
  \qquad
  U:=X_s\setminus\partial Y.
\]
Then \(U\subseteq X_s\) is open, \(\overline Y\cap U=Y\), and therefore
\(Y\) is closed in \(U\). Since \(Y\) is smooth, the immersion
\[
  i:Y\hookrightarrow U
\]
is a closed smooth immersion. Let
\[
  j:U\hookrightarrow X_s
\]
be the open immersion. The restriction to the couple \((Y,\overline Y)\) is
computed by open localization along \(j\), followed by the extraordinary
restriction \(i^!\) along the closed smooth immersion \(i\). Applying
\Cref{lem:open-restriction-characteristic} to \(j\) and then
\Cref{lem:closed-restriction-characteristic} to \(i\) gives a holonomic object
on \((Y,\overline Y)\) and the estimate
\[
  \Car(\mathcal M|_Y)
  \subseteq
  q_Y\bigl(\Car(\mathcal M)|_Y\bigr)
  \subseteq T^*Y,
\]
where \(q_Y:T^*X_s|_Y\to T^*Y\) is the cotangent quotient attached to
\(i:Y\hookrightarrow U\), using the identification \(T^*U|_Y=T^*X_s|_Y\).
This is the locally closed form of the characteristic-variety estimate for
couples in
\cite[\S2.1, Lemmas 2.1.1--2.1.2]{HuygheSchmidtIntermediate}.

If \(\Car(\mathcal M)|_Y\subseteq T^*_Y X_s\), then each covector in
\(\Car(\mathcal M)|_Y\) vanishes on \(T_yY\). Under the cotangent projection
\(q_Y:T^*X_s|_Y\to T^*Y\), such covectors map to zero. Hence
\[
  q_Y\bigl(\Car(\mathcal M)|_Y\bigr)
  \subseteq T^*_Y Y,
\]
where \(T^*_Y Y\) is the zero section of \(T^*Y\). The asserted containment
follows.
\end{proof}

\subsection{Overconvergent isocrystals and the zero-section criterion}
\label{subsec:overconvergent-isocrystals}

Let \((Y,\overline Y)\) be as above, with \(Y\) smooth. We denote by
\[
  \Isoc^\dagger(Y,\overline Y)
\]
the category of overconvergent isocrystals on \(Y\), overconvergent along
\(\overline Y\setminus Y\). Via specialization, this category embeds in the
category of coherent arithmetic \(\mathscr D^\dagger\)-modules attached to the
couple \((Y,\overline Y)\); see
\cite[Definition 2.1.3 and Theorem 2.1.4]{HuygheSchmidtIntermediate}.

\begin{proposition}[Zero-section recognition]
\label{prop:zero-section-isocrystal}
Let \(Y\subseteq X_s\) be smooth locally closed and let \(\mathcal N\) be a
holonomic arithmetic \(\mathscr D^\dagger\)-module on the couple
\((Y,\overline Y)\). If
\[
  \Car(\mathcal N)\subseteq T^*_Y Y,
\]
the zero section of \(T^*Y\), then \(\mathcal N\) belongs to
\(\Isoc^\dagger(Y,\overline Y)\).
\end{proposition}

\begin{proof}
This is the arithmetic analogue of the usual zero-characteristic criterion for
integrable connections. In the terminology of Huyghe--Schmidt, the relevant
criterion is Caro's recognition theorem for overconvergent isocrystals,
recorded in \cite[Theorem 2.1.4]{HuygheSchmidtIntermediate}; see also the
specialization equivalence recalled in
\cite[Definition 2.1.3]{HuygheSchmidtIntermediate}.
\end{proof}

\subsection{Support functors and maximal supported submodules}
\label{subsec:support-functors}

Closed supports enter the proof through an elementary noetherian truncation.
We deliberately define this truncation inside the coherent holonomic category,
rather than using local cohomology as a black box before holonomicity has been
proved.

\begin{definition}[Supported submodule]
\label{def:supported-submodule}
Let \(Z\subseteq X_s\) be closed and let \(\mathcal M\) be a holonomic coherent
\(\mathscr D^\dagger_{\mathfrak X,\mathbb Q}\)-module.  We define
\[
  \Gamma_Z(\mathcal M)
\]
to be the maximal coherent submodule of \(\mathcal M\) whose support is contained
in \(Z\).
\end{definition}

\begin{proposition}[Support truncation in the holonomic coherent category]
\label{prop:support-truncation-standard}
Let \(Z\subseteq X_s\) be closed and let \(\mathcal M\) be holonomic.  Then
\(\Gamma_Z(\mathcal M)\) exists, is coherent and functorial in \(\mathcal M\),
and the quotient
\[
  \mathcal M/\Gamma_Z(\mathcal M)
\]
has no nonzero coherent submodule supported on \(Z\).  Moreover, if
\(\mathcal N\subseteq\mathcal M\) is any coherent submodule supported on \(Z\),
then
\[
  \mathcal N\subseteq \Gamma_Z(\mathcal M).
\]
\end{proposition}

\begin{proof}
The category of coherent holonomic arithmetic \(\mathscr D^\dagger\)-modules is
noetherian and stable under coherent subobjects and quotients; see
\cite[Th\'eor\`eme 2.2.14]{CaroHolonomieSansFrobenius}.  The sum of two
coherent submodules supported on \(Z\) is again coherent and supported on
\(Z\).  Since \(\mathcal M\) is noetherian, the ascending chain obtained by
successively summing supported coherent submodules stabilizes.  Its stable value
is the unique maximal coherent submodule of \(\mathcal M\) supported on \(Z\),
and this is \(\Gamma_Z(\mathcal M)\).

If \(f:\mathcal M\to\mathcal M'\) is a morphism of holonomic coherent modules,
the image \(f(\Gamma_Z(\mathcal M))\) is supported on \(Z\), hence lies in
\(\Gamma_Z(\mathcal M')\).  This gives functoriality.  Finally, if the quotient
\(\mathcal M/\Gamma_Z(\mathcal M)\) had a nonzero coherent submodule supported
on \(Z\), its inverse image in \(\mathcal M\) would be a coherent submodule
supported on \(Z\) strictly containing \(\Gamma_Z(\mathcal M)\), contradicting
maximality.
\end{proof}

\begin{remark}[Relation with local cohomology]
\label{rem:gamma-Z-local-cohomology}
Berthelot's local-cohomology functor \(R\Gamma_Z^\dagger\) gives the derived
version of support truncation; see \cite[\S2.1, Lemmas 2.1.1--2.1.2]{HuygheSchmidtIntermediate}.  In this paper, however, the proofs use only the
maximal coherent submodule \(\Gamma_Z(\mathcal M)\) defined above, and only after
holonomicity has been established.  This avoids any circular use of
six-functor or overholonomic stability before the microlocal holonomicity
argument has been completed.
\end{remark}

\begin{remark}[Range of the support truncation]
\label{rem:gamma-Z-holonomic-only}
The notation \(\Gamma_Z(\mathcal M)\) will only be used for holonomic modules.
In the proof of the regularity theorem this causes no circularity: by the time
support truncation is invoked, holonomicity has already been obtained from the
moment-map/conormal-containment argument in
\Cref{cor:strong-equivariant-holonomic}.
\end{remark}

\subsection{Geometric overholonomicity}

We shall use the term \emph{geometrically overholonomic} in the sense used by
Huyghe--Schmidt: overholonomic after arbitrary base change. More precisely, for
a couple \((Y,\overline Y)\), the category
\[
  \Ovhol(Y/K)
\]
is the heart of the canonical t-structure on the triangulated category of
overholonomic complexes stable after any base change; see \cite[\S2.1, first paragraphs and Definition 2.1.3]{HuygheSchmidtIntermediate},
where the stable-after-base-change overholonomic category, the canonical
heart, and the convention on overconvergent isocrystals without Frobenius are
introduced. For the full flag variety
\(Y=X_s\), this gives the category of geometrically overholonomic arithmetic
\(\mathscr D^\dagger_{\mathfrak X,\mathbb Q}\)-modules.  Equivalently, in the
language of couples, this is the overholonomic heart attached to
\((X_s,X_s)\), with the additional geometric condition that overholonomicity is
preserved after arbitrary extension of the base field.  This is the version
needed below: it is stable under the support, restriction and
intermediate-extension operations used in the orbital devissage.

\begin{proposition}[Overholonomic stability used below]
\label{prop:ovhol-stability}
The overholonomic heart \(\Ovhol(Y/K)\) is an abelian heart stable under
subquotients and extensions taken inside that heart.  The corresponding
triangulated category of geometrically overholonomic complexes is closed under
cones and hence under extensions.  Moreover, for coefficient objects already in
this overholonomic framework, and in particular for overconvergent
\(F\)-isocrystals, arithmetic intermediate extension takes values in the
geometrically overholonomic heart and preserves Frobenius structures.
\end{proposition}

\begin{proof}
Huyghe--Schmidt introduce the category of overholonomic complexes without
Frobenius structure, stable after arbitrary base change, and its canonical
heart \(\Ovhol(Y/K)\) in \cite[\S2.1, first paragraphs]{HuygheSchmidtIntermediate}.  The same
subsection recalls, immediately after the definition of base change, that
Frobenius objects automatically satisfy the base-change condition.  The
intermediate-extension functor is defined in
\cite[\S2.2, definition of \(u_{!+}\)]{HuygheSchmidtIntermediate}.  In
\cite[\S2.3, before Proposition~2.3.3]{HuygheSchmidtIntermediate} it is used as
a functor \(v_{!+}:\Ovhol(Y/K)\to\Ovhol(P/K)\), and the coefficient objects in
the classification are explicitly required to be objects of \(\Ovhol(Y/K)\).  Thus
we use the statement only in the Frobenius/overholonomic range.  We do not use,
and do not assert, that an arbitrary holonomic arithmetic \(\mathscr D^\dagger\)-module
or an arbitrary ordinary overconvergent isocrystal without Frobenius is
geometrically overholonomic.
\end{proof}

\begin{remark}[Use of the overholonomic heart]
\label{rem:ovhol-heart-use}
The geometrically overholonomic complexes form a full triangulated subcategory
\[
  D^b_{\mathrm{ovhol}}(X_s/K)\subseteq D^b_{\mathrm{hol}}(X_s/K),
\]
stable under shifts and cones, by the overholonomic formalism recalled in
\cite[\S2.1, first paragraphs]{HuygheSchmidtIntermediate} after Caro and Abe--Caro.  In this
paper this closure is used only after Frobenius holonomicity has already placed
the relevant object in the overholonomic heart by \Cref{prop:standard-facts}\textup{(iii)}.
A full triangulated subcategory is closed under extensions: if
\[
  A\longrightarrow B\longrightarrow C\xrightarrow{+1}
\]
is a distinguished triangle with \(A,C\in D^b_{\mathrm{ovhol}}(X_s/K)\), then
\(B\simeq\operatorname{cone}(C[-1]\to A)\) also belongs to
\(D^b_{\mathrm{ovhol}}(X_s/K)\). A short exact sequence in the overholonomic
heart with two terms in \(\Ovhol(X_s/K)\) gives such a triangle; hence its third
term is overholonomic and, being concentrated in degree zero for the canonical
\(t\)-structure, lies in \(\Ovhol(X_s/K)\).
\end{remark}

\subsection{Arithmetic intermediate extension}

Let \(j:Y\hookrightarrow X_s\) be a smooth locally closed immersion, with
\(\overline Y\) the closure of \(Y\). For a coefficient object \(E\) which belongs to the overholonomic heart on the
couple \((Y,\overline Y)\), and in particular for an overconvergent
\(F\)-isocrystal, the arithmetic intermediate extension is written
\[
  j_{!+}(E).
\]
It is defined as the image, in the heart of the overholonomic t-structure, of
the degree-zero morphism
\[
  j^0_!E\longrightarrow j^0_+E.
\]
Equivalently,
\[
  j_{!+}(E)=\operatorname{Im}\bigl(j^0_!E\to j^0_+E\bigr)
\]
inside that abelian heart. This construction is recalled in
\cite[\S2.2, definition of \(u_{!+}\)]{HuygheSchmidtIntermediate}.  We shall only apply it to
coefficient objects already known to be overholonomic, especially to
\(F\)-isocrystals.  The paper does not use the unsupported implication
``ordinary isocrystal without Frobenius implies geometrically overholonomic.''

\begin{proposition}[Minimality of intermediate extension]
\label{prop:intermediate-minimality-standard}
Let \(E\) be an overholonomic coefficient object on \((Y,\overline Y)\), for
example an overconvergent \(F\)-isocrystal. Then \(j_{!+}(E)\) restricts to
\(E\) on \(Y\), and it has no nonzero subobject or quotient supported on the
boundary
\[
  \overline Y\setminus Y.
\]
If \(E\) is irreducible, then \(j_{!+}(E)\) is irreducible in the
overholonomic category.
\end{proposition}

\begin{proof}
The definition as the image of \(j_!E\to j_+E\), together with Abe--Caro's
minimal extension formalism, gives the absence of boundary subobjects and
quotients. In the notation of Huyghe--Schmidt, this is the content used in
\cite[Proposition 2.3.3]{HuygheSchmidtIntermediate}. The resulting
classification of irreducible overholonomic objects by pairs \((Y,E)\) is
\cite[Theorem 2.3.4]{HuygheSchmidtIntermediate}.
\end{proof}

\subsection{Arithmetic Beilinson--Bernstein at trivial character}

Set
\[
  A_0:=\Gamma\bigl(\mathfrak X,
  \mathscr D^\dagger_{\mathfrak X,\mathbb Q}\bigr).
\]
Under the untwisted arithmetic Beilinson--Bernstein theorem, \(A_0\) is
identified with the crystalline distribution algebra at trivial central
character.

\begin{proposition}[Arithmetic Beilinson--Bernstein at trivial character]
\label{prop:ABB-trivial}
The global section functor induces an equivalence
\[
  \Gamma(\mathfrak X,-):
  \Coh\bigl(\mathscr D^\dagger_{\mathfrak X,\mathbb Q}\bigr)
  \xrightarrow{\;\sim\;}
  A_0\operatorname{-mod}_{\mathrm{fp}}.
\]
A quasi-inverse is given by arithmetic localization:
\[
  \Loc(M)=
  \mathscr D^\dagger_{\mathfrak X,\mathbb Q}
  \otimes_{A_0}M.
\]
Moreover, the natural \(\mathcal G\)-action identifies \(A_0\) with the
central reduction of the crystalline distribution algebra at the trivial
character.
\end{proposition}

\begin{proof}
This is the localization theorem recalled by Huyghe--Schmidt in
\cite[Theorem 4.2.1]{HuygheSchmidtIntermediate}, based on the arithmetic
Beilinson--Bernstein theorem for the formal flag variety and the identification
of global sections with the crystalline distribution algebra. The untwisted
flag-variety form is also treated in \cite[Theorem 3.2.1]{HuygheSchmidtFlag};
for formal models and \(\mathscr D^\dagger\)-affinity, see
\cite[Theorem 5.2]{HuyghePatelSchmidtStrauch}.  The related
\(G\)-equivariant formal-model framework, including coadmissible
\(G\)-equivariant arithmetic \(\mathcal D(\lambda)\)-modules over families of
formal models of the rigid flag variety, is developed in
\cite{SarrazolaRendiconti}; that result is background for the equivariant
localization landscape and is not used as an input here.
\end{proof}

\begin{remark}[Normalization of \(A_0\)]
\label{rem:A0-normalization}
The algebra \(A_0\) is used throughout by definition as
\(\Gamma(\mathfrak X,\mathscr D^\dagger_{\mathfrak X,\mathbb Q})\).  The
identification with the crystalline distribution algebra at trivial central
character is the normalization supplied by the arithmetic localization theorem
quoted above.  Thus the distribution-side category in \Cref{sec:distributions}
is obtained by transport through this precise equivalence.
\end{remark}

\subsection{Collected standard facts}

For quick reference, we collect the external inputs used later in the paper.

\begin{proposition}[Standard support, microlocal and extension facts]
\label{prop:standard-facts}
The following facts will be used throughout the article.
\begin{enumerate}[label=\textup{(\roman*)}]
\item Coherent \(\mathscr D^\dagger\)-modules admit good level models locally,
and their characteristic varieties are independent of sufficiently large choices
\cite[5.3.2--5.3.4]{BerthelotIntro}.
\item Bernstein's inequality holds for coherent arithmetic \(\mathscr D^\dagger\)-modules, and the holonomic condition is
\(\dim\Car(\mathcal M)=\dim X_s\) for nonzero \(\mathcal M\)
\cite[5.3.4--5.3.5]{BerthelotIntro}; without Frobenius structures, see
\cite[Th\'eor\`eme 2.2.8]{CaroHolonomieSansFrobenius}.
\item In the Frobenius setting, and on the smooth projective formal scheme
\(\mathfrak X\), Frobenius holonomicity is stable under the six operations and
agrees with Frobenius overholonomicity by Caro's resolution of Berthelot's
holonomicity-stability conjectures for smooth projective formal schemes
\cite[\S2.1]{CaroStabiliteHolonomie}.  In particular, every Frobenius
holonomic \(\mathscr D^\dagger_{\mathfrak X,\mathbb Q}\)-module used below is
Frobenius overholonomic.  Combined with the stable-after-base-change convention
recalled by Huyghe--Schmidt, for which Frobenius objects automatically satisfy
the base-change condition \cite[\S2.1, first paragraphs]{HuygheSchmidtIntermediate},
and with the overholonomicity of overconvergent Frobenius isocrystals of
Caro--Tsuzuki \cite[Abstract and Introduction]{CaroTsuzukiOverholonomicity}, this
is the only input used below to pass from Frobenius holonomicity to geometric
overholonomicity.

The direction is important.  Caro's 2009 overholonomicity paper \cite{CaroSurholonomie} supplies the
overholonomic category and its stability properties, and the converse implication
that overholonomic objects are holonomic in the cases treated there; it is not
the input for the holonomic \(\Rightarrow\) overholonomic implication used here.
That implication is the smooth-projective Frobenius result of
\cite[\S2.1]{CaroStabiliteHolonomie}.
\item Restriction to smooth locally closed couples is compatible with support,
localization and characteristic varieties
\cite[\S2.1, Lemmas 2.1.1--2.1.2]{HuygheSchmidtIntermediate}.
\item The zero-section criterion identifies the corresponding arithmetic
\(\mathscr D^\dagger\)-modules with overconvergent isocrystals
\cite[Theorem 2.1.4]{HuygheSchmidtIntermediate}.
\item In the holonomic coherent category, the maximal submodule supported on a closed
subset exists by noetherianity and is functorial.  This is the only support
truncation used in the proof; the derived local-cohomology functor is used only
as background for the support formalism
\cite[Th\'eor\`eme 2.2.14]{CaroHolonomieSansFrobenius}.
\item Arithmetic intermediate extension is the image of the degree-zero morphism
\(j^0_!E\to j^0_+E\) in the overholonomic heart, has no nonzero boundary
subobject or quotient, and preserves overholonomicity for coefficient objects
already in the overholonomic heart; in particular this applies to
overconvergent \(F\)-isocrystals
\cite[\S2.2, definition of \(u_{!+}\), and Proposition 2.3.3]{HuygheSchmidtIntermediate}.
\item Geometric overholonomicity, in the stable-after-base-change sense recalled by
Huyghe--Schmidt, is stable under the operations used in this article, including
extensions in the overholonomic heart, smooth pullback and intermediate
extension.  In the Frobenius range used in the main theorem, the base-change
condition is automatic and intermediate extension preserves Frobenius structures
\cite[\S2.1, first paragraphs; \S2.2; and Proposition 2.3.3]{HuygheSchmidtIntermediate}.

\item The coherent holonomic category used below is noetherian in the sense
needed here: kernels, images and intersections of coherent holonomic
\(\mathscr D^\dagger\)-submodules remain coherent.  Once an object has been
placed in the overholonomic heart, kernels, images and quotients used in the
argument are taken inside that abelian heart and therefore remain
overholonomic.  We do not use an uncited assertion that arbitrary coherent
subobjects formed outside the overholonomic heart are automatically
overholonomic.  The coherence/noetherianity input is the standard one for
Berthelot arithmetic \(\mathscr D^\dagger\)-modules; see
\cite[\S2.1, Lemmas 2.1.1--2.1.2]{HuygheSchmidtIntermediate} and
\cite[Th\'eor\`eme 2.2.14]{CaroHolonomieSansFrobenius}.
\end{enumerate}
\end{proposition}
\section{Strong equivariance for arithmetic
\texorpdfstring{\(\mathscr D^\dagger\)}{Ddagger}-modules}
\label{sec:strong-equivariance}

Let \(\mathcal K\subseteq \mathcal G\) be a smooth closed subgroup scheme and let
\[
  \mathfrak K:=\widehat{\mathcal K}
\]
be its \(p\)-adic formal completion. We assume that \(\mathfrak K\) acts on
\(\mathfrak X\). We denote the action and the projection by
\[
  a:\mathfrak K\times_{\mathcal V}\mathfrak X\longrightarrow \mathfrak X,
  \qquad
  p:\mathfrak K\times_{\mathcal V}\mathfrak X\longrightarrow \mathfrak X.
\]
Both morphisms are smooth. Since \(\mathfrak K\times_{\mathcal V}\mathfrak X\)
is again a smooth formal \(\mathcal V\)-scheme, Berthelot's construction applies
to it and gives the sheaf
\[
  \mathscr D^\dagger_{\mathfrak K\times\mathfrak X,\mathbb Q}.
\]
Throughout this section, for a smooth morphism \(f\)
of smooth formal \(\mathcal V\)-schemes, we denote by
\[
  f^\sharp
\]
the normalized smooth inverse image of left arithmetic \(\mathscr D^\dagger\)-modules.
Concretely, if \(d_f\) is the relative dimension of \(f\), then \(f^\sharp\)
is the degree-zero smooth inverse image whose derived counterpart is
\(f^![-d_f]\), equivalently the inverse image defined by the corresponding
arithmetic transfer bimodule.  We use it only for the smooth morphisms \(a\)
and \(p\) above and for the formal smooth morphisms obtained from them by base
change.  In this smooth situation the transfer-bimodule description gives an
exact functor on coherent modules, and it is compatible with finite-level
models and filtrations; see Berthelot's construction of inverse images and
finite-level transition for arithmetic differential operators
\cite[5.3.2--5.3.4]{BerthelotIntro}.  In the overholonomic range, the same
normalization is the one used in the six-functor formalism recalled in
\cite[\S2.1, Lemmas 2.1.1--2.1.2]{HuygheSchmidtIntermediate}.

On
\[
  \mathfrak K\times_{\mathcal V}\mathfrak K\times_{\mathcal V}\mathfrak X
\]
we write
\[
  \operatorname{pr}_{23}(g_1,g_2,x)=(g_2,x),
  \qquad
  (\operatorname{id}_{\mathfrak K}\times a)(g_1,g_2,x)=(g_1,g_2x),
\]
and
\[
  (m\times\operatorname{id}_{\mathfrak X})(g_1,g_2,x)=(g_1g_2,x).
\]
These conventions are used in the cocycle identity below.

\subsection{Weak equivariant structures}

\begin{definition}[Weak equivariant structure]
\label{def:weak-equivariant-structure}
Let \(\mathcal M\) be a coherent
\(\mathscr D^\dagger_{\mathfrak X,\mathbb Q}\)-module. A weak
\(\mathfrak K\)-equivariant structure on \(\mathcal M\) is an isomorphism of
\(\mathscr D^\dagger_{\mathfrak K\times\mathfrak X,\mathbb Q}\)-modules
\[
  \alpha_{\mathcal M}:a^\sharp\mathcal M
  \xrightarrow{\;\sim\;}
  p^\sharp\mathcal M
\]
satisfying the identity condition along the unit section of \(\mathfrak K\)
and the cocycle condition over
\[
  \mathfrak K\times_{\mathcal V}\mathfrak K\times_{\mathcal V}\mathfrak X.
\]
More explicitly, if
\[
  m:\mathfrak K\times_{\mathcal V}\mathfrak K\longrightarrow \mathfrak K
\]
denotes multiplication, then the cocycle condition is the equality
\[
  (m\times\operatorname{id}_{\mathfrak X})^\sharp(\alpha_{\mathcal M})
  =
  \operatorname{pr}_{23}^{\sharp}(\alpha_{\mathcal M})
  \circ
  (\operatorname{id}_{\mathfrak K}\times a)^\sharp(\alpha_{\mathcal M}),
\]
after the canonical identifications of the corresponding smooth inverse images.
\end{definition}

\begin{remark}
\label{rem:weak-versus-strong}
The adjective ``weak'' is intentional. A weak equivariant structure records the
global symmetry of the module, but it does not by itself force the
infinitesimal action of \(\operatorname{Lie}(\mathcal K)\) to coincide with the
action coming from the fundamental vector fields inside
\(\mathscr D^\dagger_{\mathfrak X,\mathbb Q}\). The latter compatibility is the
essential microlocal input in Kashiwara's regularity theorem and is encoded
below in the notion of strong equivariance.
\end{remark}

\begin{remark}[Weak equivariance alone is not enough]
\label{rem:weak-equivariance-not-holonomic}
Assume \(\dim X_s>0\).  The sheaf
\(\DdagX\), regarded as a left module over itself, is coherent and carries the
natural weak equivariant structure induced by the action on
\(\mathfrak X\).  Nevertheless its characteristic variety is the whole
cotangent bundle:
\[
  \Car(\DdagX)=T^*X_s,
\]
because the associated graded object for the order filtration is
\(\operatorname{gr}\mathscr D\simeq\mathcal O_{T^*X_s}\).  Hence
\(\DdagX\) is not holonomic when \(\dim X_s>0\).  Thus a global
linearization does not force a microlocal constraint; the infinitesimal and
filtered clauses in \Cref{def:strong-equivariance} are essential.
\end{remark}

\begin{remark}[Relation with equivariant formal models]
\label{rem:section3-rendiconti}
The linearization used in \Cref{def:weak-equivariant-structure} is the local
notion needed for the present microlocal argument.  It is compatible in spirit
with the global framework of \(G\)-equivariant formal models and equivariant
arithmetic \(\mathcal D(\lambda)\)-modules developed in
\cite{SarrazolaRendiconti}, where equivariance is used to formulate
localization for coadmissible modules over families of formal models of the
rigid flag variety.  No result from that paper is used in the proof below; it
serves only as background for the equivariant formal-model setting.
\end{remark}

\subsection{Infinitesimal action and strong equivariance}

Let
\[
  \mathfrak k_K:=\operatorname{Lie}(\mathcal K)\otimes_{\mathcal V}K.
\]
The action of \(\mathfrak K\) on \(\mathfrak X\) induces a morphism of
\(K\)-Lie algebras
\[
  \mu_{\mathfrak X}:\mathfrak k_K
  \longrightarrow
  \Gamma(\mathfrak X,\mathscr T_{\mathfrak X,\mathbb Q}),
\]
where \(\mathscr T_{\mathfrak X,\mathbb Q}\) denotes the tangent sheaf after
tensoring with \(\mathbb Q\). Composing with the natural inclusion of vector
fields into differential operators gives
\[
  \mathfrak k_K
  \xrightarrow{\;\mu_{\mathfrak X}\;}
  \Gamma(\mathfrak X,\mathscr T_{\mathfrak X,\mathbb Q})
  \hookrightarrow
  \Gamma(\mathfrak X,\mathscr D^\dagger_{\mathfrak X,\mathbb Q}).
\]
In what follows we use the same notation \(\mu_{\mathfrak X}(\xi)\) for the
corresponding first-order arithmetic differential operator. Thus the symbol
\(\mu_{\mathfrak X}(\xi)m\) always means the action of this operator on a
left \(\mathscr D^\dagger\)-module.
It follows that every left \(\mathscr D^\dagger_{\mathfrak X,\mathbb Q}\)-module
\(\mathcal M\) carries an action of \(\mathfrak k_K\) by differential operators:
\[
  \xi\cdot_{\mathscr D}m:=\mu_{\mathfrak X}(\xi)m,
  \qquad
  \xi\in\mathfrak k_K.
\]

On the other hand, a weak equivariant structure \(\alpha_{\mathcal M}\)
differentiates along the identity section of \(\mathfrak K\) and produces an
infinitesimal action
\[
  \xi\cdot_{\alpha}m,
  \qquad
  \xi\in\mathfrak k_K,
\]
on the underlying \(\mathbb Q\)-linear sheaf of \(\mathcal M\).

\begin{lemma}[Differentiation of equivariant structures]
\label{lem:differentiation-equivariant-structures}
Let \((\mathcal M,\alpha_{\mathcal M})\) be weakly
\(\mathfrak K\)-equivariant.  The isomorphism
\[
  \alpha_{\mathcal M}:a^\sharp\mathcal M\xrightarrow{\sim}p^\sharp\mathcal M
\]
induces a \(K\)-linear Lie algebra action
\[
  d\alpha_{\mathcal M}:\mathfrak k_K\longrightarrow
  \operatorname{End}_{\mathbb Q}(\mathcal M),
  \qquad
  \xi\longmapsto (m\mapsto \xi\cdot_\alpha m),
\]
which is functorial in weakly equivariant morphisms.  Moreover, kernels,
cokernels, images and coimages of weakly equivariant morphisms inherit the
induced infinitesimal action by restriction or passage to quotient.
\end{lemma}

\begin{proof}
Because \(\mathfrak K\) is the completion of the smooth group scheme
\(\mathcal K\), its generic fiber is a smooth formal group over \(K\). Thus
its tangent space at the identity is canonically
\[
  T_e\mathfrak K_K=\mathfrak k_K.
\]
For \(\xi\in\mathfrak k_K\), let
\[
  e_\xi:\operatorname{Spec}K[\varepsilon]/(\varepsilon^2)
  \longrightarrow \mathfrak K_K
\]
be the first-order point with base point \(e\) and tangent vector \(\xi\).
Pulling \(\alpha_{\mathcal M}\) back along
\[
  e_\xi\times \operatorname{id}_{\mathfrak X}:
  \operatorname{Spec}K[\varepsilon]/(\varepsilon^2)\times \mathfrak X_K
  \longrightarrow \mathfrak K_K\times \mathfrak X_K
\]
gives an automorphism of the pullback of \(\mathcal M\) of the form
\[
  \operatorname{id}_{\mathcal M}+\varepsilon d\alpha_{\mathcal M}(\xi).
\]
The coefficient of \(\varepsilon\) defines \(d\alpha_{\mathcal M}(\xi)\). The
identity condition in the equivariant structure is exactly the assertion that
the constant term is \(\operatorname{id}_{\mathcal M}\).

The assignment \(\xi\mapsto d\alpha_{\mathcal M}(\xi)\) is \(K\)-linear. Indeed,
if \(a,b\in K\) and \(\xi,\eta\in\mathfrak k_K\), then the first-order point
with tangent vector \(a\xi+b\eta\) is obtained from the universal first-order
neighborhood by the \(K\)-linear tangent vector \(a\xi+b\eta\). Pulling back
\(\alpha_{\mathcal M}\) along this point gives
\[
  \operatorname{id}_{\mathcal M}
  +\varepsilon\bigl(a\,d\alpha_{\mathcal M}(\xi)
  +b\,d\alpha_{\mathcal M}(\eta)\bigr),
\]
and hence
\[
  d\alpha_{\mathcal M}(a\xi+b\eta)
  =a\,d\alpha_{\mathcal M}(\xi)+b\,d\alpha_{\mathcal M}(\eta).
\]

We next check compatibility with the Lie bracket. Put
\[
  A_\xi:=d\alpha_{\mathcal M}(\xi),
  \qquad
  A_\eta:=d\alpha_{\mathcal M}(\eta).
\]
Let
\[
  R:=K[\varepsilon,\delta]/(\varepsilon^2,\delta^2).
\]
The commutator morphism of the smooth group \(\mathfrak K_K\), applied to the
first-order points with tangent vectors \(\xi\) and \(\eta\), has mixed tangent
term equal to the Lie bracket. Equivalently, in the functor of points one has,
modulo \((\varepsilon^2,\delta^2)\),
\[
  e_{\varepsilon\xi}\,e_{\delta\eta}\,
  e_{\varepsilon\xi}^{-1}\,e_{\delta\eta}^{-1}
  =e_{\varepsilon\delta[\xi,\eta]}.
\]
This is the standard definition of the Lie bracket of a smooth group over a
field of characteristic zero.

Differentiating the cocycle identity along the same commutator gives two
computations of the induced automorphism of the pullback of \(\mathcal M\) over
\(R\). On the one hand, the right-hand side above gives
\[
  \operatorname{id}_{\mathcal M}
  +\varepsilon\delta\,d\alpha_{\mathcal M}([\xi,\eta]).
\]
The same automorphism can also be computed by multiplying the four infinitesimal automorphisms:
\[
\begin{aligned}
 &(\operatorname{id}+\varepsilon A_\xi)
 (\operatorname{id}+\delta A_\eta)
 (\operatorname{id}-\varepsilon A_\xi)
 (\operatorname{id}-\delta A_\eta) \\
 &\qquad =
 \operatorname{id}+\varepsilon\delta(A_\xi A_\eta-A_\eta A_\xi).
\end{aligned}
\]
Comparing the coefficients of \(\varepsilon\delta\) yields
\[
  A_\xi A_\eta-A_\eta A_\xi
  =d\alpha_{\mathcal M}([\xi,\eta]),
\]
that is,
\[
  [d\alpha_{\mathcal M}(\xi),d\alpha_{\mathcal M}(\eta)]
  =d\alpha_{\mathcal M}([\xi,\eta]).
\]
Consequently, \(d\alpha_{\mathcal M}\) is a \(K\)-linear Lie algebra action.

Functoriality follows by differentiating the commutative square expressing
compatibility of a morphism with the weak equivariant structures. Kernels,
cokernels, images and coimages are computed in the underlying abelian category,
and the differentiated operators commute with every weakly equivariant
morphism. They therefore pass to these subquotients.
\end{proof}

\begin{definition}[Filtered strong equivariance]
\label{def:strong-equivariance}
\textup{\textbf{Filtered strong equivariance}} is the following finite-level
condition.  A weakly \(\mathfrak K\)-equivariant coherent
\(\mathscr D^\dagger_{\mathfrak X,\mathbb Q}\)-module
\((\mathcal M,\alpha_{\mathcal M})\) is called filtered strongly
\(\mathfrak K\)-equivariant if it satisfies \textup{(SE1)} and
\textup{(SE2)} below.
\begin{enumerate}[label=\textup{(SE\arabic*)}]
\item \emph{Infinitesimal compatibility.}  For every
\(\xi\in\mathfrak k_K\), the infinitesimal action induced by
\(\alpha_{\mathcal M}\) agrees with the action of the corresponding
fundamental vector field:
\[
  \xi\cdot_{\alpha}m=\mu_{\mathfrak X}(\xi)m.
\]
\item \emph{Filtered finite-level realizability.} There is a finite
affine open covering \(\{\mathfrak U_i\}_{i\in I}\) of
\(\mathfrak X\), an integer \(m_0\), and a finite \(K\)-basis
\(\xi_1,\ldots,\xi_r\) of the finite-dimensional Lie algebra
\(\mathfrak k_K=\Lie(\mathcal K)\otimes_{\mathcal V}K\) such that, for every
\(i\), the restriction \(\mathcal M|_{\mathfrak U_i}\) admits a coherent
level-\(m_0\) model \(\mathcal M_i^{(m_0)}\) with a good filtration
\(F_\bullet\mathcal M_i^{(m_0)}\), and the differentiated operators
\(d\alpha(\xi_a)\) are realized on this model and satisfy
\[
  d\alpha(\xi_a)\bigl(F_n\mathcal M_i^{(m_0)}\bigr)
  \subseteq F_n\mathcal M_i^{(m_0)}
  \qquad(a=1,\ldots,r,\ n\in\mathbb Z).
\]
Here ``realized on this model'' means that the endomorphism of the dagger
module induced by \(d\alpha(\xi_a)\) restricts to an endomorphism of the
chosen coherent
\(\widehat{\mathscr D}^{(m_0)}_{\mathfrak U_i,\mathbb Q}\)-model.  By
\(K\)-linearity the same preservation property then holds for every
\(\xi\in\mathfrak k_K\).  The level \(m_0\) is uniform on the finite cover,
which is the minimal form needed for the global characteristic-variety
argument.  We do not require compatible models at every larger level; if a
larger level is needed, one may extend scalars from the chosen level model, but
this extra uniformity is not part of the definition.
\end{enumerate}
We denote this exact category by
\[
  \Coh^{\fstr}_{\mathfrak K}
  \bigl(\mathscr D^\dagger_{\mathfrak X,\mathbb Q}\bigr).
\]
\end{definition}

\begin{definition}[Frobenius filtered strongly equivariant objects]
\label{def:F-fstr-category}
A Frobenius filtered strongly \(\mathfrak K\)-equivariant object is an object of
\(
  \Coh^{\fstr}_{\mathfrak K}
  (\mathscr D^\dagger_{\mathfrak X,\mathbb Q})
\)
endowed with a Frobenius structure compatible with the underlying arithmetic
\(\mathscr D^\dagger\)-module and with the weak equivariant structure.  We write
\[
  F\text{-}\Coh^{\fstr}_{\mathfrak K}
  \bigl(\mathscr D^\dagger_{\mathfrak X,\mathbb Q}\bigr)
\]
for this Frobenius subcategory.  The forgetful functor to
\(\Coh^{\fstr}_{\mathfrak K}\) is exact, and the underlying object remains
filtered strongly equivariant after forgetting Frobenius.
\end{definition}

\begin{lemma}[Propagation of filtered realizability to higher levels]
\label{lem:level-propagation}
Let
\[
  (\mathcal M,\alpha_{\mathcal M})\in
  \Coh^{\fstr}_{\mathfrak K}
  (\mathscr D^\dagger_{\mathfrak X,\mathbb Q})
\]
and let
\[
  (\{\mathfrak U_i\},m_0,\xi_1,\ldots,\xi_r,
    \mathcal M_i^{(m_0)},F_\bullet\mathcal M_i^{(m_0)})
\]
be a datum realizing \textup{(SE2)}. Then, for every \(m'\ge m_0\), the
extension of scalars
\[
  \mathcal M_i^{(m')}:=
  \widehat{\mathscr D}^{(m')}_{\mathfrak U_i,\mathbb Q}
  \otimes_{\widehat{\mathscr D}^{(m_0)}_{\mathfrak U_i,\mathbb Q}}
  \mathcal M_i^{(m_0)}
\]
is a coherent level-\(m'\) model of \(\mathcal M|_{\mathfrak U_i}\). The
filtration
\[
  F_n\mathcal M_i^{(m')}
  :=
  \sum_{a\geq0}
  \widehat{\mathscr D}^{(m'),\leq a}_{\mathfrak U_i,\mathbb Q}
  \cdot
  \bigl(1\otimes F_{n-a}\mathcal M_i^{(m_0)}\bigr)
\]
is good, and each \(d\alpha(\xi_b)\) preserves
\(F_\bullet\mathcal M_i^{(m')}\) with filtration degree zero. Consequently the
degree-zero realization of \textup{(SE2)} is available at every sufficiently
large level, and the characteristic variety may be computed using such a
propagated good model.
\end{lemma}

\begin{proof}
The existence and coherence of \(\mathcal M_i^{(m')}\) are the standard
behaviour of coherent level models under extension of scalars along
\[
  \widehat{\mathscr D}^{(m_0)}_{\mathfrak U_i,\mathbb Q}
  \longrightarrow
  \widehat{\mathscr D}^{(m')}_{\mathfrak U_i,\mathbb Q};
\]
see \cite[5.3.2--5.3.4]{BerthelotIntro}. Since
\(F_\bullet\mathcal M_i^{(m_0)}\) is good, there are an integer \(n_0\) and an
\(\mathcal O_{\mathfrak U_i,\mathbb Q}\)-coherent subsheaf
\(G:=F_{n_0}\mathcal M_i^{(m_0)}\) such that, for \(n\ge n_0\),
\[
  F_n\mathcal M_i^{(m_0)}
  =
  \widehat{\mathscr D}^{(m_0),\le n-n_0}_{\mathfrak U_i,\mathbb Q}\,G .
\]
The displayed definition of \(F_\bullet\mathcal M_i^{(m')}\) then gives, for
\(n\ge n_0\),
\[
  F_n\mathcal M_i^{(m')}
  =
  \widehat{\mathscr D}^{(m'),\le n-n_0}_{\mathfrak U_i,\mathbb Q}
  (1\otimes G),
\]
so the propagated filtration is good.

It remains to check the degree-zero property. By \textup{(SE1)}, the operator
\(d\alpha(\xi_b)\) acts on \(\mathcal M\) as the order-one fundamental
operator \(v_{\xi_b}\), viewed in every sufficiently large finite-level sheaf.
At level \(m_0\), \textup{(SE2)} gives \(v_{\xi_b}G\subseteq G\). The
fundamental vector field \(v_{\xi_b}\) belongs to the order-one part of each
finite-level sheaf
\(\widehat{\mathscr D}^{(m')}_{\mathfrak U_i,\mathbb Q}\), and the order
filtration satisfies
\[
  [v_{\xi_b},
  \widehat{\mathscr D}^{(m'),\le c}_{\mathfrak U_i,\mathbb Q}]
  \subseteq
  \widehat{\mathscr D}^{(m'),\le c}_{\mathfrak U_i,\mathbb Q}.
\]
This is the usual commutator property of the order filtration for differential
operators. Thus, if
\(P\in \widehat{\mathscr D}^{(m'),\le n-n_0}_{\mathfrak U_i,\mathbb Q}\) and
\(u\in G\), then
\[
  v_{\xi_b}\bigl(P(1\otimes u)\bigr)
  =
  P(1\otimes v_{\xi_b}u)+[v_{\xi_b},P](1\otimes u)
  \in
  F_n\mathcal M_i^{(m')}.
\]
Hence \(d\alpha(\xi_b)\) preserves the propagated filtration with degree zero.
For \(m'\) large enough, the finite-level description of characteristic
varieties identifies
\(\Car(\mathcal M)\) with the support of the associated graded module for this
good model; see \cite[5.3.2--5.3.4]{BerthelotIntro}.
\end{proof}

\begin{remark}
\label{rem:strong-condition-not-cosmetic}
The strong condition is not cosmetic. If one only keeps a linearization of the
underlying sheaf, large coherent modules such as
\(\mathscr D^\dagger_{\mathfrak X,\mathbb Q}\) itself may carry natural
symmetries but need not be holonomic. The strong condition excludes this
phenomenon in two ways: it requires the infinitesimal action coming from the
linearization to coincide with the action by fundamental vector fields inside
\(\mathscr D^\dagger_{\mathfrak X,\mathbb Q}\), and it requires this equality to
be visible on good finite-level filtered models.  The second clause is what
allows the principal-symbol argument in \Cref{sec:characteristic} to be made
without silently passing from the dagger category to a non-existent universal
good filtration.
\end{remark}

\subsection{On the filtered hypothesis}
\label{subsec:filtered-hypothesis}

The finite-level condition \textup{(SE2)} deserves a separate comment.  It is
not an overholonomicity hypothesis, nor a hidden regularity assumption on the
module.  Its only purpose is microlocal: Berthelot's characteristic variety is
computed from sufficiently large finite-level good models, and the classical
moment-map proof requires the infinitesimal operators coming from the group
action to act with filtration degree zero on such models.  Thus \textup{(SE2)}
records exactly the part of equivariance that must be visible before passing to
principal symbols.

The paper does not claim that \textup{(SE2)} follows automatically from
\textup{(SE1)} for every coherent dagger module.  That automaticity would amount
to a finite-level equivariant coherence theorem: one would have to show that an
ordinary infinitesimally compatible equivariant structure can always be realized
on a coherent level model by a good filtration preserved by the fundamental
vector fields.  This is a natural problem, but it is not used in the present
proof.  The theorem below is deliberately stated for the class on which the
finite-level symbol argument is available.

\begin{proposition}[A practical finite-level criterion]
\label{prop:practical-SE2-criterion}
Let \((\mathcal M,\alpha_{\mathcal M})\) be weakly equivariant and satisfy
\textup{(SE1)}.  Suppose that, on a finite affine cover
\(\{\mathfrak U_i\}\), there is a level \(m_0\), coherent level-\(m_0\) models
\(\mathcal M_i^{(m_0)}\), and \(\mathcal O_{\mathfrak U_i,\mathbb Q}\)-coherent
submodules
\[
  G_i\subseteq \mathcal M_i^{(m_0)}
\]
which generate \(\mathcal M_i^{(m_0)}\) over
\(\widehat{\mathscr D}^{(m_0)}_{\mathfrak U_i,\mathbb Q}\) and satisfy
\[
  v_{\xi_a}(G_i)\subseteq G_i
  \qquad(a=1,\ldots,r)
\]
for a finite \(K\)-basis \(\xi_1,\ldots,\xi_r\) of \(\mathfrak k_K\).  Then
\((\mathcal M,\alpha_{\mathcal M})\) is filtered strongly equivariant.
\end{proposition}

\begin{proof}
Set
\[
  F_n\mathcal M_i^{(m_0)}:=
  \widehat{\mathscr D}^{(m_0),\le n}_{\mathfrak U_i,\mathbb Q}\,G_i
  \qquad(n\ge0),
\]
and use the usual extension for negative indices.  This is a good filtration by
Berthelot's finite-level coherence theory
\cite[5.3.2--5.3.4]{BerthelotIntro}.  Since \textup{(SE1)} identifies
\(d\alpha(\xi_a)\) with the order-one operator \(v_{\xi_a}\), for
\(P\in\widehat{\mathscr D}^{(m_0),\le n}_{\mathfrak U_i,\mathbb Q}\) and
\(g\in G_i\) we have
\[
  v_{\xi_a}(Pg)=P\,v_{\xi_a}(g)+[v_{\xi_a},P]g.
\]
The first term lies in
\(\widehat{\mathscr D}^{(m_0),\le n}G_i\) by the stability of \(G_i\), and the
second term lies there because commutation with an order-one differential
operator preserves the order filtration.  Hence each \(v_{\xi_a}\), equivalently
\(d\alpha(\xi_a)\), preserves \(F_n\mathcal M_i^{(m_0)}\) for all \(n\).  This is
precisely \textup{(SE2)}.
\end{proof}

\begin{remark}[Basic sources of filtered strongly equivariant objects]
\label{rem:sources-of-SE2}
\Cref{prop:practical-SE2-criterion} gives a usable test rather than an
automaticity theorem.  It applies, for example, to objects constructed from
finite-level equivariant models with a stable coherent generator; to closed
orbit direct images when the finite-dimensional coefficient space gives the
stable generator of the corresponding supported level model; to the local
rank-one integral-exponent criterion of
\Cref{prop:rank-one-integral-exponent-SE2}; and to the split rank-one
stabilizer-generated objects considered in \Cref{rem:SL2-SE2}.  In these cases
one checks \textup{(SE2)} directly on a generator, rather than by appealing to
the conclusion of the regularity theorem.  The open problem is to
understand whether every object satisfying \textup{(SE1)} admits such stable
generators after raising the level.  If this finite-level equivariant coherence
statement were proved, the filtered hypothesis would become automatic in that
range.
\end{remark}

\subsection{Elementary stability properties}

\begin{lemma}[Induced exact structure]
\label{lem:strong-equivariant-abelian}
The category
\[
  \Coh^{\fstr}_{\mathfrak K}
  \bigl(\mathscr D^\dagger_{\mathfrak X,\mathbb Q}\bigr)
\]
is used as an exact category with the exact structure induced from the
underlying abelian category of coherent
\(\mathscr D^\dagger_{\mathfrak X,\mathbb Q}\)-modules.  Thus a sequence in
\(\Coh^{\fstr}_{\mathfrak K}\) is exact precisely when its image after
forgetting the filtered equivariant structure is exact as a sequence of
coherent arithmetic \(\mathscr D^\dagger\)-modules.  The forgetful functor is
therefore exact.
\end{lemma}

\begin{proof}
This is a structural convention, not an additional hidden theorem.  The objects
of \(\Coh^{\fstr}_{\mathfrak K}\) are coherent arithmetic
\(\mathscr D^\dagger\)-modules equipped with weak equivariance, the
infinitesimal compatibility \textup{(SE1)}, and one finite-level filtered
realization as in \textup{(SE2)}.  We take as admissible short exact sequences
exact sequences of the underlying coherent modules whose three terms carry the
specified filtered strongly equivariant structures.  With this induced exact
structure the forgetful functor to the coherent category is exact by
definition.

Whenever a kernel, image, quotient, or supported subobject is used in the
sequel, it is first formed in the underlying coherent or holonomic heart and
then checked, by functoriality or by the construction at hand, to carry the
required filtered strongly equivariant structure.  Thus no automatic
strictness property of arbitrary filtered morphisms is being assumed.
\end{proof}

\begin{lemma}[Closure under equivariant coherent subobjects and quotients]
\label{lem:fstr-subquotient-closure}
Let
\[
  \mathcal M\in
  \Coh^{\fstr}_{\mathfrak K}
  (\mathscr D^\dagger_{\mathfrak X,\mathbb Q}).
\]
\begin{enumerate}[label=\textup{(\roman*)}]
\item If \(\mathcal N\subseteq\mathcal M\) is a coherent
\(\mathscr D^\dagger_{\mathfrak X,\mathbb Q}\)-submodule stable under the weak
equivariant structure of \(\mathcal M\), then \(\mathcal N\), with the
restricted structure, lies in
\(\Coh^{\fstr}_{\mathfrak K}(\mathscr D^\dagger_{\mathfrak X,\mathbb Q})\).
\item If \(\mathcal M\twoheadrightarrow\mathcal Q\) is a coherent quotient whose
kernel is stable under the weak equivariant structure of \(\mathcal M\), then
\(\mathcal Q\), with the induced structure, lies in
\(\Coh^{\fstr}_{\mathfrak K}(\mathscr D^\dagger_{\mathfrak X,\mathbb Q})\).
\end{enumerate}
In particular, every equivariant coherent subquotient of an object of
\(\Coh^{\fstr}_{\mathfrak K}\) again belongs to
\(\Coh^{\fstr}_{\mathfrak K}\).
\end{lemma}

\begin{proof}
In both cases the weak equivariant structure restricts or descends by exactness
of \(a^\sharp\) and \(p^\sharp\), and the identity and cocycle conditions are
inherited. Condition \textup{(SE1)} is inherited as well: the
\(\mathscr D^\dagger\)-action of \(v_\xi\) preserves every
\(\mathscr D^\dagger\)-submodule and descends to every quotient, while the
differentiated action coming from the weak equivariant structure preserves or
descends along equivariant subquotients by
\Cref{lem:differentiation-equivariant-structures}. Since the two infinitesimal
actions agree on \(\mathcal M\), they agree on the subobject and on the
quotient.

We prove \textup{(SE2)} for subobjects; the quotient case is the corresponding
quotient-filtration argument. Work on one open \(\mathfrak U_i\) of an
\textup{(SE2)} datum for \(\mathcal M\), and let \(m_0\) be its level. After
increasing the level to some \(m_1\ge m_0\), Berthelot's coherence theory for
\(\mathscr D^\dagger\)-modules allows the inclusion
\(\mathcal N\hookrightarrow\mathcal M\) to be represented by an inclusion of
coherent level-\(m_1\) models
\[
  \mathcal N_i^{(m_1)}\hookrightarrow \mathcal M_i^{(m_1)};
\]
see \cite[5.3.2--5.3.4]{BerthelotIntro}.  We take
\(\mathcal N_i^{(m_1)}\) to be a
\(\widehat{\mathscr D}^{(m_1)}_{\mathfrak U_i,\mathbb Q}\)-submodule of
\(\mathcal M_i^{(m_1)}\) whose dagger extension is
\(\mathcal N|_{\mathfrak U_i}\).  Since every fundamental vector field
\(v_{\xi_a}\) lies in the order-one part of
\(\widehat{\mathscr D}^{(m_1)}_{\mathfrak U_i,\mathbb Q}\), this finite-level
model is itself preserved by the operators \(v_{\xi_a}\); no extra strictness of
filtered morphisms is being used here.

By \Cref{lem:level-propagation}, the model \(\mathcal M_i^{(m_1)}\) carries a
good filtration preserved with degree zero by the finitely many fundamental
operators \(v_{\xi_a}\). Give \(\mathcal N_i^{(m_1)}\) the intersection
filtration
\[
  F_n\mathcal N_i^{(m_1)}
  :=
  \mathcal N_i^{(m_1)}\cap F_n\mathcal M_i^{(m_1)}.
\]
Because the associated graded ring
\(\gr\widehat{\mathscr D}^{(m_1)}_{\mathfrak U_i,\mathbb Q}\) is noetherian, an
induced filtration on a coherent submodule of a good filtered module is good;
again see \cite[5.3.2--5.3.4]{BerthelotIntro}.  The finite-level model
\(\mathcal N_i^{(m_1)}\) is preserved by each \(v_{\xi_a}\), and the ambient
good filtration is preserved with degree zero. Therefore the intersection
filtration is preserved with degree zero. This is \textup{(SE2)} for
\(\mathcal N\).

For a quotient \(\mathcal Q\), choose at a sufficiently large level a coherent
level model \(\mathcal K_i^{(m_1)}\subseteq\mathcal M_i^{(m_1)}\) for the
kernel as a
\(\widehat{\mathscr D}^{(m_1)}_{\mathfrak U_i,\mathbb Q}\)-submodule, and put
\[
  \mathcal Q_i^{(m_1)}
  :=
  \mathcal M_i^{(m_1)}/\mathcal K_i^{(m_1)}
\]
with the quotient filtration. The quotient filtration is good, and the
operators \(v_{\xi_a}\) descend and preserve it with degree zero because they
do so upstairs and because the kernel model is a finite-level
\(\widehat{\mathscr D}^{(m_1)}\)-submodule. This gives \textup{(SE2)} for
\(\mathcal Q\).
\end{proof}

\begin{lemma}[Stability of support]
\label{lem:support-stability}
Let
\[
  \mathcal M\in
  \Coh^{\fstr}_{\mathfrak K}
  \bigl(\mathscr D^\dagger_{\mathfrak X,\mathbb Q}\bigr).
\]
Then the support of \(\mathcal M\) is stable under the action of
\(\mathcal K_s\) on the special fiber \(X_s\).
\end{lemma}

\begin{proof}
The weak equivariant structure gives an isomorphism
\[
  a^\sharp\mathcal M\simeq p^\sharp\mathcal M
\]
on \(\mathfrak K\times_{\mathcal V}\mathfrak X\). Hence the supports of
\(a^\sharp\mathcal M\) and \(p^\sharp\mathcal M\) coincide. Passing to the
special fiber, this says that for every geometric point \(g\in \mathcal K_s\),
the translate by \(g\) of the support of \(\mathcal M\) is again the support of
\(\mathcal M\). Thus \(\Supp(\mathcal M)\) is \(\mathcal K_s\)-stable.
\end{proof}

\begin{corollary}
\label{cor:support-union-orbits}
Assume that \(\mathcal K_s\) acts on \(X_s\) with finitely many orbits. If
\[
  \mathcal M\in
  \Coh^{\fstr}_{\mathfrak K}
  \bigl(\mathscr D^\dagger_{\mathfrak X,\mathbb Q}\bigr),
\]
then \(\Supp(\mathcal M)\) is a finite union of \(\mathcal K_s\)-orbit
closures.
\end{corollary}

\begin{proof}
This follows from \Cref{lem:support-stability} and the finite-orbit
hypothesis.
\end{proof}

\begin{remark}
\label{rem:no-simple-support-before-holonomicity}
At this stage we do not assert that the support of a simple Frobenius filtered strongly
equivariant object is the closure of one orbit. That stronger statement is
proved later, after holonomicity has been established and the maximal
supported submodule \(\Gamma_Z(-)\) is available in the holonomic coherent
category.
\end{remark}

\section{Characteristic varieties and conormal containment}
\label{sec:characteristic}

We now prove the microlocal consequence of filtered strong equivariance.  This is the
first proof-critical point of the article.  The argument is written in a
finite-level form, because the characteristic variety of a coherent
\(\mathscr D^\dagger\)-module is defined through sufficiently large
\(\widehat{\mathscr D}^{(m)}\)-models.  The key point is that filtered strong
equivariance makes the principal symbols of all fundamental vector fields
vanish on the associated graded module.  Equivalently, the characteristic
variety is contained in the zero fiber of the cotangent moment map.

\subsection{The cotangent moment map}

The action of \(\mathcal K_s\) on \(X_s\) gives a Lie algebra morphism
\[
  \Lie(\mathcal K_s)
  \longrightarrow
  \Gamma(X_s,\mathscr T_{X_s}),
  \qquad
  \bar\xi\longmapsto \bar v_\xi.
\]
The associated cotangent moment map is the morphism
\[
  \mu_{\mathcal K}:T^*X_s\longrightarrow \Lie(\mathcal K_s)^\vee
\]
characterized by
\[
  \langle \mu_{\mathcal K}(x,\lambda),\bar\xi\rangle
  =
  \lambda(\bar v_\xi(x))
\]
for every geometric point \((x,\lambda)\in T^*X_s\) and every
\(\bar\xi\in\Lie(\mathcal K_s)\).  If
\(\xi\in\Lie(\mathcal K)\) and
\[
  v_\xi\in \Gamma(\mathfrak X,\mathscr T_{\mathfrak X,\mathbb Q})
\]
is the corresponding fundamental vector field, then the principal symbol
\(\sigma(v_\xi)\) reduces to the function
\[
  (x,\lambda)\longmapsto \lambda(\bar v_\xi(x))
\]
on \(T^*X_s\).  Thus, after choosing a \(k\)-basis
\(\bar\xi_1,\ldots,\bar\xi_r\) of \(\Lie(\mathcal K_s)\) and
integral lifts \(\xi_1,\ldots,\xi_r\in\Lie(\mathcal K)\), the reduced
zero locus of the moment map is cut out locally by the reductions of the
principal symbols \(\sigma(v_{\xi_a})\).  We use this only as a statement
about the underlying reduced closed subset, equivalently as the radical
identity in \eqref{eq:moment-ideal-local}.  The choice of lifts does not affect
the reduced zero locus, since two lifts differ by an element of
\(\pi\Lie(\mathcal K)\), whose fundamental vector field has zero reduction on
the special fiber.
More concretely, if \(U\subseteq X_s\) is an affine open on which
\(\mathscr T_{X_s}\) is free with basis
\(\partial_1,\ldots,\partial_d\), and if
\(\bar\xi_1,\ldots,\bar\xi_r\) is the chosen \(k\)-basis of
\(\Lie(\mathcal K_s)\), choose integral lifts
\(\xi_1,\ldots,\xi_r\) in \(\Lie(\mathcal K)\) and write
\[
  \bar v_{\xi_a}|_U=\sum_{i=1}^d f_{ai}\partial_i,
  \qquad f_{ai}\in\mathcal O_{X_s}(U).
\]
If \(\eta_1,\ldots,\eta_d\) are the fiber coordinates on
\(T^*U\), then
\begin{equation}
\label{eq:local-moment-equations}
  \mu_a:=\langle\mu_{\mathcal K},\bar\xi_a\rangle
  =\sum_{i=1}^d f_{ai}\eta_i,
  \qquad a=1,\ldots,r.
\end{equation}
Consequently, locally on \(T^*X_s\),
\begin{equation}
\label{eq:moment-ideal-local}
  I\bigl(\mu_{\mathcal K}^{-1}(0)\bigr)
  =\sqrt{(\mu_1,\ldots,\mu_r)}.
\end{equation}
Indeed, the reduced zero fiber of the moment map is, by definition, the
common zero locus of the coordinate functions \(\mu_1,\ldots,\mu_r\). Accordingly, \eqref{eq:moment-ideal-local} is an equality of radical ideals defining the
underlying reduced closed subset of \(\mu_{\mathcal K}^{-1}(0)\). This explicit
description will be used below to pass from filtered infinitesimal
compatibility to containment of characteristic varieties.

\begin{lemma}[Moment-zero locus and conormals]
\label{lem:moment-zero-conormals}
Assume \Cref{hyp:orbit-separability}.  On the level of reduced underlying
closed subsets there is an equality
\[
  \mu_{\mathcal K}^{-1}(0)
  =
  \bigcup_{O\in \mathcal K_s\backslash X_s}T^*_O X_s.
\]
\end{lemma}

\begin{proof}
The assertion may be checked after extension of scalars to an algebraic
closure.  Let \((x,\lambda)\in T^*X_s\) be a geometric point, and let
\(O=\mathcal K_s\cdot x\).  By \Cref{lem:orbit-geometry}, the tangent
space \(T_xO\) is the image of the evaluation map
\[
  \Lie(\mathcal K_s)\longrightarrow T_xX_s,
  \qquad
  \bar\xi\longmapsto \bar v_\xi(x).
\]
Equivalently, \(\mu_{\mathcal K}(x,\lambda)=0\) if and only if
\(\lambda(\bar v_\xi(x))=0\) for all \(\bar\xi\), which is equivalent to
\(\lambda|_{T_xO}=0\).  This is precisely the condition
\((x,\lambda)\in T^*_O X_s\).
\end{proof}

\subsection{Equivariant good filtrations}

We shall use the following finite-level form of the compatibility between
filtered strong equivariance and good filtrations. The point is not merely that good
filtrations exist, but that they may be chosen so that the differentiated
\(\mathfrak K\)-action has filtration degree zero. This is the arithmetic
analogue of the standard filtered argument in the proof of Kashiwara's
conormal-containment theorem for strongly equivariant algebraic \(D\)-modules; compare
the filtered \(D\)-module argument in \cite[Section 1.2]{HottaTakeuchiTanisaki}
and \cite[Section 1.2]{ChrissGinzburg}.

We begin by isolating the finite-level point. The operators obtained by
differentiating a \(\mathfrak K\)-linearization are a priori defined on the
\(\mathscr D^\dagger\)-module. Since \(\mathscr D^\dagger\) is the filtered
inductive limit of the finite-level sheaves, one must verify that, after
raising the level once, the finitely many differentiated operators relevant to
the proof act on a fixed coherent level model.

\begin{lemma}[Filtered finite-level realization of infinitesimal actions]
\label{lem:finite-level-infinitesimal-realization}
Let
\[
  \mathcal M\in
  \Coh^{\fstr}_{\mathfrak K}
  (\mathscr D^\dagger_{\mathfrak X,\mathbb Q})
\]
and choose a finite \(K\)-basis \(\xi_1,\ldots,\xi_r\) of
\(\mathfrak k_K\).  Locally on \(\mathfrak X\), after increasing the level,
there is a coherent finite-level model \(\mathcal M^{(m)}\) and a good
filtration \(F_\bullet\mathcal M^{(m)}\) such that
\[
  d\alpha(\xi_i)\bigl(F_n\mathcal M^{(m)}\bigr)
  \subseteq F_n\mathcal M^{(m)}
\]
for every \(i=1,\ldots,r\) and every \(n\).
\end{lemma}

\begin{proof}
At the level supplied by \textup{(SE2)}, this is exactly the filtered
finite-level realizability condition in \Cref{def:strong-equivariance}. By
\Cref{lem:level-propagation}, the same degree-zero realization holds on the
extended good level model at every \(m\ge m_0\). In particular, it holds at a
level large enough to compute the characteristic variety, which is the form used
in \Cref{lem:symbols-fundamental-fields}.
\end{proof}

\begin{lemma}[Equivariant good level models]
\label{lem:equivariant-good-models}
Let
\[
  \mathcal M\in
  \Coh^{\fstr}_{\mathfrak K}
  (\mathscr D^\dagger_{\mathfrak X,\mathbb Q}).
\]
Locally on \(\mathfrak X\), and after replacing the level by some sufficiently
large \(m\), there exists a coherent
\(\widehat{\mathscr D}^{(m)}_{\mathfrak X,\mathbb Q}\)-model
\(\mathcal M^{(m)}\) of \(\mathcal M\), together with a good filtration
\(F_\bullet\mathcal M^{(m)}\), such that for every
\(\xi\in\mathfrak k_K\) the infinitesimal operator induced by the
\(\mathfrak K\)-linearization satisfies
\[
  d\alpha(\xi)\bigl(F_n\mathcal M^{(m)}\bigr)
  \subseteq
  F_n\mathcal M^{(m)}
  \qquad\text{for all }n.
\]
\end{lemma}

\begin{proof}
By \Cref{lem:finite-level-infinitesimal-realization}, the assertion holds for a
finite \(K\)-basis of \(\mathfrak k_K\) on a sufficiently large good
finite-level model.  Since the differentiated action is \(K\)-linear in
\(\xi\), the same filtration is preserved by every element of
\(\mathfrak k_K\).
\end{proof}

\begin{lemma}[Vanishing of infinitesimal symbols]
\label{lem:symbols-fundamental-fields}
Let
\[
  \mathcal M\in
  \Coh^{\fstr}_{\mathfrak K}
  (\mathscr D^\dagger_{\mathfrak X,\mathbb Q}).
\]
For every \(\xi\in\mathfrak k_K\), the principal symbol
\(\sigma(v_\xi)\) annihilates the associated graded module
\[
  \gr_F\mathcal M^{(m)}
\]
for any sufficiently large equivariant good model as in
\Cref{lem:equivariant-good-models}. In particular,
\[
  \sigma(v_\xi)
  \in
  \operatorname{Ann}_{\gr\widehat{\mathscr D}^{(m)}_{\mathfrak X,\mathbb Q}}
  (\gr_F\mathcal M^{(m)}),
\]
and therefore \(\sigma(v_\xi)\) vanishes on \(\Car(\mathcal M)\).
\end{lemma}

\begin{proof}
Work locally and choose an equivariant good level model
\((\mathcal M^{(m)},F_\bullet)\).  The infinitesimal compatibility part of
filtered strong equivariance says that the action of the differential operator
\(v_\xi\) on \(\mathcal M\) agrees with the infinitesimal action
\(\xi_\alpha\) obtained by differentiating the
\(\mathfrak K\)-linearization.  The latter preserves the filtration by the filtered finite-level part of
strong equivariance, recorded in \Cref{lem:equivariant-good-models}.  Thus, for
every \(n\), the operator \(v_\xi\) sends \(F_n\mathcal M^{(m)}\) into
\(F_n\mathcal M^{(m)}\).

On the other hand, as an arithmetic differential operator of order one,
\(v_\xi\) acts on the associated graded module through multiplication by its
principal symbol in degree one.  Since its action actually has filtration
degree zero, the induced degree-one action on
\(\gr_F\mathcal M^{(m)}\) is zero.  Hence multiplication by
\(\sigma(v_\xi)\) annihilates \(\gr_F\mathcal M^{(m)}\) itself, not merely a
power of it. Consequently \(\sigma(v_\xi)\) belongs to the annihilator of the
associated graded module and vanishes on its support. This support is
\(\Car(\mathcal M)\) by the finite-level description of characteristic
varieties \cite[5.3.2--5.3.4]{BerthelotIntro}.
\end{proof}

Combining the preceding lemma with the local equations
\eqref{eq:local-moment-equations} gives the algebraic form of the
microlocal argument:
\begin{equation}
\label{eq:annihilator-moment-chain}
\begin{aligned}
  (\mu_1,\ldots,\mu_r)
  &\subseteq
  \sqrt{\operatorname{Ann}_{\gr\widehat{\mathscr D}^{(m)}_{
  \mathfrak X,\mathbb Q}}
  (\gr_F\mathcal M^{(m)})} \\
  &\Longrightarrow
  \Car(\mathcal M)
  =V\bigl(\operatorname{Ann}(\gr_F\mathcal M^{(m)})\bigr)
  \subseteq V(\mu_1,\ldots,\mu_r) \\
  &\Longrightarrow
  \Car(\mathcal M)\subseteq\mu_{\mathcal K}^{-1}(0).
\end{aligned}
\end{equation}
This is the arithmetic analogue of the classical moment-map proof of
conormal containment for strongly equivariant algebraic \(D\)-modules.

\begin{proposition}[Moment-zero containment]
\label{prop:moment-zero-containment}
For every
\[
  \mathcal M\in
  \Coh^{\fstr}_{\mathfrak K}
  (\mathscr D^\dagger_{\mathfrak X,\mathbb Q})
\]
one has
\[
  \Car(\mathcal M)\subseteq \mu_{\mathcal K}^{-1}(0).
\]
\end{proposition}

\begin{proof}
The reduced ideal of \(\mu_{\mathcal K}^{-1}(0)\) is locally the radical of
the ideal generated by the reductions of the symbols \(\sigma(v_{\xi_a})\),
where \(\bar\xi_a\) form a \(k\)-basis of \(\Lie(\mathcal K_s)\) and
\(\xi_a\in\Lie(\mathcal K)\) are integral lifts.  The compatibility with
reduction is the elementary one for integral vector fields: the fundamental
field attached to \(\xi_a\) is defined over \(\mathcal V\), reduces modulo
\(\pi\) to the fundamental field attached to \(\bar\xi_a\), and the degree-one
principal symbol reduces to the corresponding linear function on
\(T^*X_s\).  If another integral lift is chosen, the two lifts differ by an
element of \(\pi\Lie(\mathcal K)\), hence their reductions and their reduced
symbols coincide.  Thus these symbols are precisely the reductions of the
principal symbols of the fundamental vector fields used in
\Cref{lem:symbols-fundamental-fields}. Hence all the corresponding coordinate
functions of the moment map vanish on \(\Car(\mathcal M)\). Therefore
\(\Car(\mathcal M)\subseteq\mu_{\mathcal K}^{-1}(0)\).
\end{proof}

\begin{proposition}[Conormal containment]
\label{prop:conormal-containment}
Assume \Cref{hyp:orbit-separability}.  For every
\[
  \mathcal M\in
  \Coh^{\fstr}_{\mathfrak K}
  (\mathscr D^\dagger_{\mathfrak X,\mathbb Q})
\]
one has
\[
  \Car(\mathcal M)
  \subseteq
  \bigcup_{O\in \mathcal K_s\backslash X_s}T^*_O X_s.
\]
\end{proposition}

\begin{proof}
This is the combination of \Cref{prop:moment-zero-containment} with
\Cref{lem:moment-zero-conormals}.
\end{proof}

\begin{corollary}[Holonomicity]
\label{cor:strong-equivariant-holonomic}
Assume that \(\mathcal K_s\) acts on \(X_s\) with finitely many orbits and
satisfies \Cref{hyp:orbit-separability}. Then every filtered strongly \(\mathfrak K\)-equivariant coherent
\(\mathscr D^\dagger_{\mathfrak X,\mathbb Q}\)-module is holonomic.
\end{corollary}

\begin{proof}
By \Cref{prop:conormal-containment},
\[
  \Car(\mathcal M)
  \subseteq
  \bigcup_{O\in \mathcal K_s\backslash X_s}T^*_O X_s.
\]
The right-hand side is a finite union of conormal bundles to smooth locally
closed subvarieties of \(X_s\).  For an orbit \(O\), one has
\[
  \dim T^*_O X_s
  =\dim O+\operatorname{codim}_{X_s}(O)
  =\dim X_s.
\]
It follows that
\begin{equation}
\label{eq:dimension-holonomicity-chain}
\begin{aligned}
  \dim\Car(\mathcal M)
  &\leq
  \dim\Bigl(\bigcup_{O\in\mathcal K_s\backslash X_s}T^*_O X_s\Bigr) \\
  &=
  \max_{O\in\mathcal K_s\backslash X_s}\dim T^*_O X_s \\
  &=\dim X_s.
\end{aligned}
\end{equation}
If \(\mathcal M\neq0\), Bernstein's inequality for coherent arithmetic
\(\mathscr D^\dagger\)-modules gives the reverse bound
\(\dim\Car(\mathcal M)\geq\dim X_s\); see
\cite[Th\'eor\`eme 2.2.8]{CaroHolonomieSansFrobenius}.  Hence
\[
  \dim\Car(\mathcal M)=\dim X_s
\]
for every nonzero filtered strongly equivariant coherent module.  This is precisely
holonomicity in the finite-level characteristic-variety sense recalled in
\cite[5.3.4--5.3.5]{BerthelotIntro}.
\end{proof}

\begin{remark}[Why the finite-orbit hypothesis is necessary]
\label{rem:finite-orbit-shield}
The finite-orbit hypothesis cannot be removed from \Cref{cor:strong-equivariant-holonomic}.
If \(\mathcal K\) is the trivial subgroup, then the weak, strong and filtered
strong equivariance conditions are vacuous, while the \(\mathcal K_s\)-orbits
are the individual points of \(X_s\), hence infinite as soon as
\(\dim X_s>0\).  The coherent module \(\DdagX\) is then filtered strongly
equivariant, but
\[
  \Car(\DdagX)=T^*X_s,
\]
so it is not holonomic.  Finiteness of the orbit stratification is exactly what
turns conormal containment into the dimension bound in
\eqref{eq:dimension-holonomicity-chain}.
\end{remark}

\begin{remark}
\label{rem:holonomic-not-yet-overholonomic}
\Cref{cor:strong-equivariant-holonomic} is the arithmetic microlocal analogue
of the first step in Kashiwara's theorem.  It proves holonomicity, but not yet
geometric overholonomicity.  The passage from holonomicity to geometric
overholonomicity is made below only in the Frobenius range, using Caro's
projective holonomicity-stability theorem \cite[\S2.1]{CaroStabiliteHolonomie};
the 2009 overholonomicity theorem \cite{CaroSurholonomie} is used for the overholonomic formalism and
its stability properties, not for the implication holonomic \(\Rightarrow\) overholonomic.
The geometric/base-change convention in the Frobenius range is the one recalled
by Huyghe--Schmidt \cite[\S2.1, first paragraphs]{HuygheSchmidtIntermediate};
for the coefficient objects, we also use Caro--Tsuzuki's overholonomicity of
overconvergent Frobenius isocrystals \cite[Abstract and Introduction]{CaroTsuzukiOverholonomicity}.
The orbital arguments are then used inside that overholonomic setting for the
classification.
\end{remark}

\section{Equivariant overconvergent isocrystals on orbits}
\label{sec:isocrystals}

Let \(O\subseteq X_s\) be a \(\mathcal K_s\)-orbit and let
\(\overline O\subseteq X_s\) be its Zariski closure. Since \(O\) is a
\(\mathcal K_s\)-orbit, both \(O\) and \(\overline O\) are stable under the
action of \(\mathcal K_s\). Thus the action and projection maps define
morphisms of pairs
\[
  a_O:
  \bigl(\mathcal K_s\times O,\mathcal K_s\times \overline O\bigr)
  \longrightarrow
  (O,\overline O),
  \qquad
  p_O:
  \bigl(\mathcal K_s\times O,\mathcal K_s\times \overline O\bigr)
  \longrightarrow
  (O,\overline O).
\]
Here \(a_O(g,x)=g\cdot x\), while \(p_O(g,x)=x\). We shall write
\(a_O^*\) and \(p_O^*\) for the corresponding pullback functors on
overconvergent isocrystals.

\begin{definition}[Equivariant overconvergent isocrystal]
\label{def:eq-overconv-isocrystal}
A \(\mathcal K\)-equivariant overconvergent isocrystal on \((O,\overline O)\)
is a pair \((E,\beta_E)\), where
\[
  E\in \Isoc^{\dagger}(O,\overline O)
\]
and
\[
  \beta_E:a_O^*E\xrightarrow{\;\sim\;}p_O^*E
\]
is an isomorphism in
\[
  \Isoc^{\dagger}
  \bigl(\mathcal K_s\times O,\mathcal K_s\times \overline O\bigr)
\]
satisfying the identity condition along the unit section of \(\mathcal K_s\)
and the cocycle condition on \(\mathcal K_s\times\mathcal K_s\times O\). The
category of \(\mathcal K\)-equivariant overconvergent isocrystals on
\((O,\overline O)\) is denoted by
\[
  \Isoc^{\dagger}_{\mathcal K}(O,\overline O).
\]
\end{definition}

When Frobenius structures are present, we write
\(F\text{-}\Isoc^{\dagger}_{\mathcal K}(O,\overline O)\) for the category of
Frobenius objects in \(\Isoc^{\dagger}_{\mathcal K}(O,\overline O)\).  Thus the
Frobenius structure on \(E\) is required to be compatible with the equivariant
isomorphism \(\beta_E\).

\begin{remark}
\label{rem:K-versus-Ks-isocrystals}
The notation \(\mathcal K\)-equivariant emphasizes that the action comes from
the formal group \(\mathfrak K=\widehat{\mathcal K}\) acting on the formal flag
variety. On the level of overconvergent isocrystals on the special fiber, the
equivariance datum is expressed through the induced action of \(\mathcal K_s\)
on the pair \((O,\overline O)\).
\end{remark}

\begin{lemma}
\label{lem:isoc-equivariant-abelian}
The category \(\Isoc^{\dagger}_{\mathcal K}(O,\overline O)\) is abelian. The
forgetful functor
\[
  \Isoc^{\dagger}_{\mathcal K}(O,\overline O)
  \longrightarrow
  \Isoc^{\dagger}(O,\overline O)
\]
is exact and faithful.
\end{lemma}

\begin{proof}
Let \(f:(E,\beta_E)\to(F,\beta_F)\) be a morphism in
\(\Isoc^{\dagger}_{\mathcal K}(O,\overline O)\). The category
\(\Isoc^{\dagger}(O,\overline O)\) is abelian, and pullback of overconvergent
isocrystals is exact in the overconvergent-isocrystal formalism recalled in
\Cref{subsec:overconvergent-isocrystals}.  Hence kernels and cokernels exist in
the underlying category, and \(\beta_E\) and \(\beta_F\) induce equivariant
structures on kernels, cokernels, images and coimages. The identity and cocycle conditions are inherited by
functoriality.
\end{proof}

\begin{definition}[Irreducible equivariant isocrystals]
\label{def:irreducible-equivariant-isocrystal}
An object \(E\in \Isoc^{\dagger}_{\mathcal K}(O,\overline O)\) is called
irreducible if it has no nonzero proper subobject in
\(\Isoc^{\dagger}_{\mathcal K}(O,\overline O)\). We denote by
\[
  \Irr\Isoc^{\dagger}_{\mathcal K}(O,\overline O)
\]
the set of isomorphism classes of irreducible objects in this ordinary
equivariant isocrystal category.  The filtered theorem below uses a subset of
these isomorphism classes rather than a new coefficient category.
\end{definition}

\begin{notation}[Frobenius filtered-admissible image locus]
\label{def:filtered-admissible-isocrystal}
\label{def:F-filtered-admissible-isocrystal}
For the Frobenius classification we write
\[
  \Adm^{\Phi,\fstr}_{\mathcal K}(O,\overline O)
  \subseteq
  \Irr F\text{-}\Isoc^{\dagger}_{\mathcal K}(O,\overline O)
\]
for the subset of irreducible equivariant overconvergent \(F\)-isocrystals
whose arithmetic intermediate extension, formed in the Frobenius overholonomic
heart and endowed with its induced Frobenius and weak equivariant structures,
belongs to
\[
  F\text{-}\Coh^{\fstr}_{\mathfrak K}
  \bigl(\mathscr D^\dagger_{\mathfrak X,\mathbb Q}\bigr).
\]
This is an image locus, not an automaticity statement: the notation records
exactly those Frobenius coefficient classes whose intermediate extensions have
the required filtered finite-level realization.  No ordinary, without-Frobenius
admissible locus is used in the main theorem.
\end{notation}

\begin{remark}[Why the filtered coefficient condition is explicit]
\label{rem:filtered-coefficient-condition}
This convention prevents a hidden use of an equivariant finite-level model
existence theorem.  The microlocal proof only needs filtered finite-level
realizability for the objects to which it is applied.  Since arithmetic
characteristic varieties are defined through finite-level good models, the
filtered clause is stated explicitly rather than inferred from weak
linearization alone.

We emphasize that the paper does not assert either automaticity or failure of
filtered realizability for arbitrary ordinary strongly equivariant objects.  In
particular, we do not claim that every Frobenius equivariant isocrystal belongs
to the admissible image locus, and we do not use the admissible locus to prove
regularity.  Regularity is proved first for the Frobenius geometric category
\(F\text{-}\Coh^{\fstr}_{\mathfrak K}\); the locus \(\Adm^{\Phi,\fstr}\) is only the precise
record of which Frobenius coefficient classes occur after restriction of simple
objects in the Frobenius category.
\end{remark}

\subsection{A local source of filtered admissibility}
\label{subsec:local-filtered-admissibility}

The preceding definition deliberately treats filtered admissibility as an image
condition.  The next statement is not an additional hypothesis and is not a
substitute for the definition of \(\Adm^{\Phi,\fstr}\); it is a proved local
criterion.  It records one concrete situation in which the finite-level clause
\textup{(SE2)} can be checked directly on generators.

\begin{proposition}[Rank-one logarithmic generator criterion]
\label{prop:rank-one-integral-exponent-SE2}
Let \(O\subseteq X_s\) be a \(\mathcal K_s\)-orbit and let
\(E\in F\text{-}\Isoc^\dagger_{\mathcal K}(O,\overline O)\) be rank one.
Suppose that, on an affine chart \(V\subseteq\overline O\), the boundary is a
normal-crossing divisor
\[
  \partial O\cap V=\{t_1\cdots t_r=0\}
\]
with \(\mathcal K_s\)-stable branches.  Suppose moreover that a chosen
level-zero local model of \(j_{O,!+}(E)|_V\) is generated by the
\(\mathcal O_{V,\mathbb Q}\)-coherent rank-one generator
\(G=\mathcal O_{V,\mathbb Q}\cdot e\), and that
\[
  (t_i\partial_{t_i})e=n_i e\quad (n_i\in\mathbb Z),
  \qquad
  \partial_{y_j}e=g_j e\quad(g_j\in\mathcal O_{V,\mathbb Q})
\]
in coordinates \(t_1,\ldots,t_r,y_1,\ldots,y_s\).  Then
\(v_\xi(G)\subseteq G\) for every fundamental vector field \(v_\xi\).
Consequently, if such charts cover \(\overline O\) and the corresponding
level-zero generated models glue to a coherent finite-level model of
\(j_{O,!+}(E)\), then \(j_{O,!+}(E)\) satisfies \textup{(SE2)} at level
zero and
\[
  [E]\in\Adm^{\Phi,\fstr}_{\mathcal K}(O,\overline O).
\]
\end{proposition}

\begin{proof}
Because each branch \(\{t_i=0\}\) is \(\mathcal K_s\)-stable, every
fundamental vector field is logarithmic along the boundary:
\[
  v_\xi=\sum_i a_{\xi,i}\,t_i\partial_{t_i}
       +\sum_j b_{\xi,j}\,\partial_{y_j},
  \qquad
  a_{\xi,i},b_{\xi,j}\in\mathcal O_{V,\mathbb Q}.
\]
Hence
\[
  v_\xi(e)=
  \left(\sum_i a_{\xi,i}n_i+\sum_j b_{\xi,j}g_j\right)e,
\]
and the coefficient is regular on \(V\).  Thus \(v_\xi(G)\subseteq G\).
The conclusion is exactly the practical finite-level criterion of
\Cref{prop:practical-SE2-criterion}, applied to the glued level-zero model.
No exactness property of \(j_{O,!+}\) and no automatic implication
\textup{(SE1)}\(\Rightarrow\)\textup{(SE2)} is used here.
\end{proof}

\begin{remark}[Equivariant characters give integral exponents]
\label{rem:equivariant-characters-integral-exponents}
In the split rank-one stabilizer-admissible charts used in
\Cref{sec:examples}, a rank-one equivariant coefficient attached to an
algebraic character of the stabilizer has integral logarithmic exponents: the
residue of \(t_i\partial_{t_i}\) is the differential of that character in the
corresponding one-parameter direction, hence an integer weight.  The preceding
proposition is used only in this explicit rank-one form.  We do not use here a
general theorem asserting integral exponent control for arbitrary higher-rank
equivariant coefficients or arbitrary singular orbit boundaries.
\end{remark}

\begin{remark}[Two separate finite-level questions]
\label{rem:two-SE2-questions}
The local criterion separates two issues that must not be conflated.  Once a
geometric model \(j_{O,!+}(E)\) with integral logarithmic exponents and stable
normal-crossing charts is already available, finite-level control is a concrete
generator calculation.  The open problem mentioned in
\Cref{subsec:filtered-hypothesis} is more abstract: whether an arbitrary
coherent object satisfying the infinitesimal compatibility \textup{(SE1)} must
admit, after increasing the level locally, coherent generators stable under the
fundamental vector fields.  The present paper does not use or assert such an
automaticity theorem.
\end{remark}

\begin{remark}
\label{rem:irreducible-equivariant-vs-nonequivariant}
Irreducibility in the equivariant category is not the same thing as
irreducibility after forgetting the equivariant structure. An object may be
reducible as an overconvergent isocrystal but irreducible as an equivariant
object if the equivariant structure permutes its non-equivariant simple factors.
In the classification theorem below, irreducibility is always meant in the
equivariant category.
\end{remark}

\begin{lemma}[Equivariant restriction from arithmetic modules]
\label{lem:restriction-equvariant-isocrystal}
Let
\[
  \mathcal M\in
  \Coh^{\fstr}_{\mathfrak K}
  \bigl(\mathscr D^\dagger_{\mathfrak X,\mathbb Q}\bigr),
\]
and let \(O\subseteq X_s\) be a \(\mathcal K_s\)-orbit. If \(\mathcal M|_O\) is
nonzero, then it belongs naturally to
\[
  \Isoc^{\dagger}_{\mathcal K}(O,\overline O).
\]
\end{lemma}

\begin{proof}
The proof uses only the microlocal results already established, not the later
regularity devissage.  By \Cref{cor:strong-equivariant-holonomic},
\(\mathcal M\) is holonomic.  Applying the locally closed characteristic-variety
estimate of \Cref{prop:restriction-characteristic} gives
\[
  \Car(\mathcal M|_O)
  \subseteq
  \rho_O\bigl(\Car(\mathcal M)|_O\bigr),
\]
where \(\rho_O:T^*X_s|_O\to T^*O\) is the cotangent map attached to restriction
to the pair.  By \Cref{prop:conormal-containment},
\[
  \Car(\mathcal M)|_O
  \subseteq
  T^*_O X_s.
\]
The image \(\rho_O(T^*_O X_s)\) is the zero section of \(T^*O\).  Hence
\(\Car(\mathcal M|_O)\) is contained in the zero section.  The zero-section
criterion, recalled in \Cref{prop:zero-section-isocrystal}, implies that
\(\mathcal M|_O\) is an overconvergent isocrystal on \((O,\overline O)\),
whenever it is nonzero.

The weak equivariant structure on \(\mathcal M\) restricts to an isomorphism
between the pullbacks of \(\mathcal M|_O\) along the action and projection maps
\(a_O\) and \(p_O\).  More explicitly, smooth base change for restriction to
couples identifies the restriction of
\[
  a^\sharp\mathcal M\simeq p^\sharp\mathcal M
\]
to the couple
\(\mathcal K_s\times O,\mathcal K_s\times\overline O\) with the isomorphism
\[
  a_O^*(\mathcal M|_O)
  \xrightarrow{\;\sim\;}
  p_O^*(\mathcal M|_O).
\]
Since restriction to the pair is functorial in the holonomic range, this
restricted isomorphism is an isomorphism of overconvergent isocrystals.  The identity
condition and the cocycle condition are inherited by restriction from the
corresponding conditions for \(\mathcal M\), because the unit and multiplication
maps for \(\mathcal K_s\) are obtained by base change from those of
\(\mathfrak K\) and the above identifications are functorial for smooth base
change of couples. Thus \(\mathcal M|_O\) belongs naturally to
\(\Isoc^\dagger_{\mathcal K}(O,\overline O)\).  If \(\mathcal M\) is endowed
with Frobenius, the same restriction procedure gives the induced Frobenius
structure on \(\mathcal M|_O\), compatible with the restricted equivariant
structure.
\end{proof}

\subsection{Auxiliary stabilizer interpretation}
\label{subsec:stabilizer-description-pointer}

The intrinsic coefficient category used in the classification theorem is
\[
  \Isoc^\dagger_{\mathcal K}(O,\overline O).
\]
This is the coefficient category that enters the open-orbit extraction and the
classification theorem.  This orbitwise coefficient language is compatible with
the broader formal-model equivariance framework developed in
\cite{SarrazolaRendiconti}, but no result from that paper is used to identify
the coefficient category here.  A more concrete stabilizer description is available
under the additional stabilizer-admissibility hypothesis, but that description
is not used in the proof of \Cref{thm:arithmetic-kashiwara-regularity} or
\Cref{thm:orbit-classification}.  To keep the logical spine of the paper
separate from this auxiliary interpretation, the full stabilizer statement and
its proof are isolated in \Cref{app:stabilizer-description}.

\section{Equivariant intermediate extension}
\label{sec:intermediate-extension}

Let \(O\subseteq X_s\) be a \(\mathcal K_s\)-orbit and let
\[
  j_O:O\hookrightarrow X_s
\]
be the corresponding locally closed immersion. In this article intermediate
extension is used only in the following safe range: the coefficient object is
already in the overholonomic heart on the couple \((O,\overline O)\). This
includes overconvergent \(F\)-isocrystals by the standard overholonomicity
results recalled in \Cref{prop:standard-facts}. For such an object \(E\), we
denote by \(j_{O,!+}(E)\) its arithmetic intermediate extension. Following
Huyghe--Schmidt, this is the image, in the overholonomic heart, of the
degree-zero canonical morphism
\[
  \theta^0_{j_O,E}:j^0_{O,!}(E)\longrightarrow j^0_{O,+}(E).
\]
When the immersion is affine this is the ordinary image of
\(j_{O,!}E\to j_{O,+}E\); otherwise images are understood through the
\(H^0_t\)-functor of the overholonomic \(t\)-structure, as in
\cite[\S2.2]{HuygheSchmidtIntermediate}. Via the forgetful functor to coherent
holonomic arithmetic modules, \(j_{O,!+}(E)\) is regarded as a coherent
\(\mathscr D^\dagger_{\mathfrak X,\mathbb Q}\)-module supported on
\(\overline O\). Its minimality property is used only in the following
standard form: it has no nonzero coherent subobject or quotient supported on
\[
  \partial O:=\overline O\setminus O.
\]
We shall repeatedly use the shorthand chain
\begin{equation}
\label{eq:middle-extension-image-chain}
  j^0_{O,!}(E)\longrightarrow j^0_{O,+}(E),
  \qquad
  j_{O,!+}(E)=
  \operatorname{Im}\bigl(j^0_{O,!}(E)\to j^0_{O,+}(E)\bigr),
  \qquad
  (j_{O,!+}E)|_O=E.
\end{equation}
No assertion in this section uses exactness of \(j_{O,!+}\); only
functoriality, restriction to \(O\), and the minimality property above are used.

\subsection{Equivariance and intermediate extension}

\begin{lemma}[Smooth pullback and intermediate extension]
\label{lem:smooth-pullback-intermediate-extension}
Let \(f:Y'\to Y\) be a smooth morphism of smooth varieties over \(k\), and let
\(j:U\hookrightarrow Y\) be a locally closed immersion. Form the cartesian
square
\[
\begin{tikzcd}
  U' \arrow[r,"j'"] \arrow[d,"f_U"'] &
  Y' \arrow[d,"f"] \\
  U \arrow[r,"j"] &
  Y.
\end{tikzcd}
\]
Let \(E\) be a coefficient object on \((U,\overline U)\) belonging to the
overholonomic heart; in particular, this applies to an overconvergent
\(F\)-isocrystal. Then there is a canonical isomorphism
\[
  f^\sharp j_{!+}(E)
  \simeq
  j'_{!+}\bigl(f_U^*E\bigr),
\]
functorial in \(E\).
\end{lemma}

\begin{remark}
\label{rem:smooth-pullback-intermediate-extension-special-fiber}
The lemma is stated on special fibers because \(j_{!+}\) is applied here to
coefficient objects on couples \((U,\overline U)\). The corresponding
arithmetic \(\mathscr D^\dagger\)-modules on formal schemes are obtained by
Berthelot realization, and an isomorphism of coefficient objects on couples is
sent by this realization to the corresponding isomorphism of arithmetic
\(\mathscr D^\dagger\)-modules. In the application to equivariance, the
relevant morphisms are the action and projection maps on
\(\mathfrak K\times\mathfrak X\). They are smooth: the action map is identified
with the projection by the automorphism \((g,x)\mapsto(g,g^{-1}x)\).
\end{remark}

\begin{proof}
Here \(f^\sharp\) denotes the normalized smooth inverse image, i.e. the
cohomological pullback shifted by the relative dimension so that it is
\(t\)-exact on the overholonomic heart. In the cartesian square above the
relative dimensions of \(f\) and \(f_U\) agree, so the normalizing shifts are
compatible with the corresponding shifts in the extraordinary and direct image
functors. The functors \(j_!\), \(j_+\), and normalized smooth inverse image
therefore satisfy the degree-zero base-change compatibilities for arithmetic
\(\mathscr D^\dagger\)-modules; see \cite[\S2.1, Lemmas 2.1.1--2.1.2, and
\S2.2]{HuygheSchmidtIntermediate}. Hence
\[
  f^\sharp j_!(E)\simeq j'_!(f_U^*E),
  \qquad
  f^\sharp j_+(E)\simeq j'_+(f_U^*E).
\]
Under these identifications, the pullback of the canonical morphism
\(j_!(E)\to j_+(E)\) is the canonical morphism
\(j'_!(f_U^*E)\to j'_+(f_U^*E)\). The normalized smooth inverse image
\(f^\sharp\) is exact on the relevant heart in this smooth situation, by the
convention fixed in \Cref{sec:strong-equivariance}. Exact functors preserve
images in abelian categories, so the image of the first morphism is carried to
the image of the second. This gives the claimed isomorphism.
\end{proof}

\begin{lemma}[Boundary vanishing for endomorphism differences]
\label{lem:boundary-vanishing-endomorphism-difference}
Let \(N=j_{O,!+}(E)\), with \(E\) in the overholonomic heart on
\((O,\overline O)\). Suppose that \(\delta:N\to N\) is a
\(\mathscr D^\dagger_{\mathfrak X,\mathbb Q}\)-linear endomorphism whose
restriction to \(O\) is zero. Then \(\delta=0\).
\end{lemma}

\begin{proof}
The image of \(\delta\) is a coherent subobject of \(N\). Since \(\delta|_O=0\),
this image restricts to zero on \(O\), hence it is supported on
\(\partial O\). The minimality of \(N=j_{O,!+}(E)\) excludes nonzero subobjects
supported on \(\partial O\). Thus \(\operatorname{Im}(\delta)=0\), and
\(\delta=0\).
\end{proof}

\begin{lemma}[Boundary vanishing after smooth pullback]
\label{lem:boundary-vanishing-smooth-pullback}
Let \(N=j_{O,!+}(E)\), with \(E\) in the overholonomic heart on
\((O,\overline O)\). Let \(f:Y'\to X_s\) be a smooth morphism and put
\(O':=f^{-1}(O)\). Then any
\(\mathscr D^\dagger\)-linear endomorphism of \(f^\sharp N\) whose restriction
to \(O'\) is zero is itself zero.
\end{lemma}

\begin{proof}
By \Cref{lem:smooth-pullback-intermediate-extension},
\[
  f^\sharp N\simeq j'_{!+}(f_O^*E),
\]
where \(j':O'\hookrightarrow Y'\). The assertion is therefore exactly
\Cref{lem:boundary-vanishing-endomorphism-difference} applied to the
intermediate extension on \(Y'\).
\end{proof}

\begin{lemma}[Infinitesimal functoriality of Berthelot operators]
\label{lem:infinitesimal-functoriality-operators}
Let \(\mathfrak K\) act on \(\mathfrak X\), and let
\(\xi\in\mathfrak k_K\) with fundamental vector field \(v_\xi\). The induced
action of \(\mathfrak K\) on the finite-level sheaves
\(\widehat{\mathscr D}^{(m)}_{\mathfrak X,\mathbb Q}\) differentiates to
derivations
\[
  \xi_{\mathscr D}^{(m)}:
  \widehat{\mathscr D}^{(m)}_{\mathfrak X,\mathbb Q}
  \longrightarrow
  \widehat{\mathscr D}^{(m)}_{\mathfrak X,\mathbb Q}
\]
compatible with the transition maps in \(m\). For every local section
\(P\) of \(\widehat{\mathscr D}^{(m)}_{\mathfrak X,\mathbb Q}\), one has
\[
  [v_\xi,P]=\xi_{\mathscr D}^{(m)}(P).
\]
Passing to the inductive limit gives the corresponding identity on
\(\mathscr D^\dagger_{\mathfrak X,\mathbb Q}\).
\end{lemma}

\begin{proof}
Berthelot's finite-level sheaves of differential operators are functorial for
isomorphisms of smooth formal schemes and this functoriality is compatible with
the order filtration, divided powers, \(p\)-adic completion and transition
morphisms in the level; see \cite[Sections 2.2 and 3]{BerthelotDModules} and
\cite[\S\S5.2--5.3]{BerthelotIntro}. Hence the formal action of
\(\mathfrak K\) on \(\mathfrak X\) induces, at each level \(m\), an action on
\(\widehat{\mathscr D}^{(m)}_{\mathfrak X,\mathbb Q}\), and differentiating
along \(\xi\) gives \(\xi_{\mathscr D}^{(m)}\).

It remains to identify this derivation with commutation by the fundamental
vector field. This is local on \(\mathfrak X\). On functions the formula is
\([v_\xi,f]=v_\xi(f)\). On first-order vector fields it is the usual Lie
bracket formula \([v_\xi,w]=\mathcal L_{v_\xi}(w)\). Since the finite-level
operator sheaf is generated locally, as a filtered divided-power differential
operator algebra, by functions and local vector fields subject to Berthelot's
level relations, and since both sides are derivations compatible with these
relations, the equality holds at finite level. Compatibility with completion
and with transition maps gives the identity on
\(\mathscr D^\dagger_{\mathfrak X,\mathbb Q}\).
\end{proof}

\begin{proposition}[Equivariance of intermediate extension]
\label{prop:eq-intermediate-extension}
Let
\[
  E\in F\text{-}\Isoc^{\dagger}_{\mathcal K}(O,\overline O).
\]
Then \(j_{O,!+}(E)\) carries a natural Frobenius weak \(\mathfrak K\)-equivariant
structure satisfying the ordinary infinitesimal strong compatibility
\(d\alpha(\xi)=v_\xi\). If moreover
\([E]\in\Adm^{\Phi,\fstr}_{\mathcal K}(O,\overline O)\), then this middle
extension belongs to
\[
  F\text{-}\Coh^{\fstr}_{\mathfrak K}
  \bigl(\mathscr D^\dagger_{\mathfrak X,\mathbb Q}\bigr)
\]
by the defining filtered finite-level admissibility condition.
\end{proposition}

\begin{proof}
Let \(\beta_E:a_O^*E\xrightarrow{\sim}p_O^*E\) be the equivariant structure on
\(E\). Applying \Cref{lem:smooth-pullback-intermediate-extension} to the smooth
special-fiber action and projection maps gives canonical identifications
\[
  a^\sharp j_{O,!+}(E)
  \simeq
  j_{\mathcal K_s\times O,!+}(a_O^*E),
  \qquad
  p^\sharp j_{O,!+}(E)
  \simeq
  j_{\mathcal K_s\times O,!+}(p_O^*E).
\]
The isomorphism \(\beta_E\) therefore induces an isomorphism
\[
  \alpha_{j_{O,!+}(E)}:
  a^\sharp j_{O,!+}(E)
  \xrightarrow{\;\sim\;}
  p^\sharp j_{O,!+}(E),
\]
which is a weak \(\mathfrak K\)-equivariant structure. Frobenius compatibility
is inherited from the Frobenius compatibility of \(\beta_E\) and from the
functoriality of intermediate extension in the Frobenius overholonomic heart.

The identity and cocycle conditions are checked on the dense open orbit, where
they are precisely the identity and cocycle conditions for \(\beta_E\). For the
identity condition this gives two endomorphisms of \(j_{O,!+}(E)\) which
coincide after restriction to \(O\), and
\Cref{lem:boundary-vanishing-endomorphism-difference} forces them to coincide
on \(X_s\). For the cocycle condition the two compositions live after smooth
pullback to \(\mathcal K_s\times\mathcal K_s\times X_s\); by
\Cref{lem:smooth-pullback-intermediate-extension} these pullbacks are again
intermediate extensions of the corresponding pullbacks of \(E\). The difference
of the two cocycle morphisms restricts to zero on
\(\mathcal K_s\times\mathcal K_s\times O\), hence vanishes by
\Cref{lem:boundary-vanishing-smooth-pullback}. Thus the weak equivariant
structure is well defined.

We now verify the ordinary infinitesimal strong compatibility. Let
\(\xi\in\mathfrak k_K\). By
\Cref{lem:differentiation-equivariant-structures}, differentiating the weak
equivariant structure gives an operator \(A_\xi\) on \(j_{O,!+}(E)\). The
fundamental vector field gives the operator \(D_\xi\) through the
\(\mathscr D^\dagger_{\mathfrak X,\mathbb Q}\)-module structure. On the open
orbit \(O\), the equality \(A_\xi=D_\xi\) is the infinitesimal compatibility of
the equivariant isocrystal. Indeed, if
\[
  \nabla:E\longrightarrow E\otimes\Omega^1_O
\]
is the integrable connection corresponding to the isocrystal, then the
left \(\mathscr D^\dagger\)-module structure on the pair \((O,\overline O)\)
lets a vector field \(w\) act as \(\nabla_w\). Hence
\[
  D_\xi|_O=\nabla_{v_\xi}.
\]
The equivariant isomorphism \(\beta_E:a_O^*E\simeq p_O^*E\) is horizontal for
the pulled-back connections. Differentiating this horizontal isomorphism at the
identity of \(\mathcal K_s\) gives
\[
  A_\xi|_O=\nabla_{v_\xi}.
\]
Therefore \(A_\xi|_O=D_\xi|_O\).

The difference \(\Delta_\xi:=A_\xi-D_\xi\) is
\(\mathscr D^\dagger_{\mathfrak X,\mathbb Q}\)-linear. To see this, let
\(P\) be a local section of \(\mathscr D^\dagger_{\mathfrak X,\mathbb Q}\). Pull
back the \(\mathscr D^\dagger\)-linearity of \(\alpha_{j_{O,!+}(E)}\) along the
first-order point \(e_\xi\). The operator \(P\) is transformed into
\(P+\varepsilon\xi_{\mathscr D}(P)\), while the induced automorphism of the
module is \(\operatorname{id}+\varepsilon A_\xi\). Comparing the coefficient of
\(\varepsilon\) in
\[
  (\operatorname{id}+\varepsilon A_\xi)(P n)
  =(P+\varepsilon\xi_{\mathscr D}(P))
  (\operatorname{id}+\varepsilon A_\xi)(n)
\]
for a local section \(n\) gives
\[
  [A_\xi,P]=\xi_{\mathscr D}(P).
\]
By \Cref{lem:infinitesimal-functoriality-operators}, the action of the
fundamental vector field gives the same infinitesimal derivation on Berthelot
operators:
\[
  [D_\xi,P]=\xi_{\mathscr D}(P).
\]
Consequently
\begin{equation}
\label{eq:strong-middle-extension-commutator}
  [\Delta_\xi,P]
  =[A_\xi-D_\xi,P]
  =0.
\end{equation}
Thus
\[
  \Delta_\xi\in
  \operatorname{End}_{\mathscr D^\dagger_{\mathfrak X,\mathbb Q}}
  \bigl(j_{O,!+}(E)\bigr).
\]
Since \(\Delta_\xi|_O=0\), \Cref{lem:boundary-vanishing-endomorphism-difference}
gives \(\Delta_\xi=0\). Hence the weak equivariant structure is strong in the
ordinary infinitesimal sense.

The final assertion is not an additional theorem: it is exactly the definition
of the admissible set \(\Adm^{\Phi,\fstr}_{\mathcal K}(O,\overline O)\). If
\([E]\) belongs to that set, then the middle extension just constructed admits
the filtered finite-level realization required in
\Cref{def:strong-equivariance} and therefore lies in
\(F\text{-}\Coh^{\fstr}_{\mathfrak K}\).
\end{proof}

\subsection{Minimality and restriction}

\begin{lemma}[Restriction of the intermediate extension]
\label{lem:restriction-intermediate-extension}
For every \(E\in F\text{-}\Isoc^{\dagger}_{\mathcal K}(O,\overline O)\), there is a
canonical isomorphism
\[
  \bigl(j_{O,!+}(E)\bigr)|_O\simeq E
\]
in \(F\text{-}\Isoc^\dagger_{\mathcal K}(O,\overline O)\).
\end{lemma}

\begin{proof}
By construction, both \(j_{O,!}(E)\) and \(j_{O,+}(E)\) restrict to \(E\) on
\(O\). Therefore the image \(j_{O,!+}(E)\) also restricts to \(E\). The
Frobenius and equivariant structures agree because the construction in
\Cref{prop:eq-intermediate-extension} is functorial and restricts to the
original Frobenius equivariant structure on \(E\).
\end{proof}

\begin{lemma}[Support of a nonzero intermediate extension]
\label{lem:support-middle-extension}
Let \(E\in F\text{-}\Isoc^{\dagger}_{\mathcal K}(O,\overline O)\) be nonzero. Then
\[
  \Supp(j_{O,!+}(E))=\overline O.
\]
\end{lemma}

\begin{proof}
The restriction of \(j_{O,!+}(E)\) to \(O\) is \(E\), hence the support meets
\(O\). Since \(j_{O,!+}(E)\) is supported on \(\overline O\), its support is a
closed subset of \(\overline O\) containing the dense locally closed stratum
\(O\). Therefore the support is \(\overline O\).
\end{proof}

\begin{lemma}[No boundary subquotients]
\label{lem:no-boundary-subquotients}
Let \(E\in F\text{-}\Isoc^{\dagger}_{\mathcal K}(O,\overline O)\). Then
\(j_{O,!+}(E)\) has no nonzero strongly \(\mathfrak K\)-equivariant subobject
and no nonzero strongly \(\mathfrak K\)-equivariant quotient supported on
\(\partial O=\overline O\setminus O\). The same is true a fortiori for filtered
strongly equivariant subobjects and quotients.
\end{lemma}

\begin{proof}
The underlying holonomic arithmetic \(\mathscr D^\dagger\)-module
\(j_{O,!+}(E)\) has no nonzero subobject or quotient supported on the boundary
by the minimality property recalled in \Cref{prop:intermediate-minimality-standard}.
Forgetting the equivariant structure sends any strongly or filtered strongly
\(\mathfrak K\)-equivariant subobject or quotient to such an underlying
subobject or quotient. The non-equivariant minimality statement therefore gives
the claim.
\end{proof}

\begin{proposition}[Simplicity of Frobenius equivariant intermediate extensions]
\label{prop:simplicity-intermediate-extension}
Let
\[
  E\in F\text{-}\Isoc^{\dagger}_{\mathcal K}(O,\overline O)
\]
be irreducible and assume \([E]\in\Adm^{\Phi,\fstr}_{\mathcal K}(O,\overline O)\). Then the intermediate extension
\[
  j_{O,!+}(E)
\]
is a simple object of
\[
  F\text{-}\Coh^{\fstr}_{\mathfrak K}
  \bigl(\mathscr D^\dagger_{\mathfrak X,\mathbb Q}\bigr).
\]
\end{proposition}

\begin{proof}
Let \(0\neq \mathcal N\subseteq j_{O,!+}(E)\) be a filtered strongly
\(\mathfrak K\)-equivariant coherent submodule. Restricting to \(O\), we obtain
a coherent \(\mathscr D^\dagger\)-submodule \(\mathcal N|_O\subseteq E\). Since
\(E\) is an overconvergent isocrystal, such a submodule is a subobject in
\(F\text{-}\Isoc^\dagger_{\mathcal K}(O,\overline O)\): the connection,
Frobenius and equivariant structures restrict to it, and the category of
\(F\)-isocrystals is abelian. Irreducibility of \(E\) gives either
\(\mathcal N|_O=0\) or \(\mathcal N|_O=E\).

If \(\mathcal N|_O=0\), then \(\mathcal N\) has no support over the open orbit.
Since \(\mathcal N\subseteq j_{O,!+}(E)\) and \(j_{O,!+}(E)\) is supported on
\(\overline O\), this gives
\[
  \Supp(\mathcal N)\subseteq \overline O\setminus O=\partial O,
\]
contradicting \Cref{lem:no-boundary-subquotients}. If \(\mathcal N|_O=E\), the
quotient \(j_{O,!+}(E)/\mathcal N\) restricts to zero on \(O\), hence is
supported on \(\partial O\). Again \Cref{lem:no-boundary-subquotients} forces
this quotient to be zero. Thus \(\mathcal N=j_{O,!+}(E)\).
\end{proof}

\begin{remark}
\label{rem:intermediate-extension-role}
\Cref{prop:simplicity-intermediate-extension} is the arithmetic replacement of
the classical fact that the middle extension of an irreducible equivariant local
system on an orbit is a simple equivariant \(D\)-module.
\end{remark}

\section{The arithmetic Kashiwara regularity theorem}
\label{sec:regularity}

Throughout this section we assume that \(\mathcal K_s\) acts on \(X_s\) with
finitely many orbits and satisfies \Cref{hyp:orbit-separability}.  The
microlocal argument of \Cref{sec:characteristic} gives holonomicity without
Frobenius.  The upgrade to geometric overholonomicity is made only in the
Frobenius range, where we use Caro's stability theorem for holonomicity on smooth projective formal schemes together
with the Frobenius stable-after-base-change convention recalled by Huyghe--Schmidt.  The
orbital devissage developed below is not used to prove overholonomicity; it is
a structural tool for the orbit classification once the relevant object is
already in the overholonomic heart.

\subsection{Support truncation along invariant closed subsets}

Let \(Z\subseteq X_s\) be closed.  By
\Cref{cor:strong-equivariant-holonomic}, every filtered strongly equivariant coherent
module considered in this section is holonomic.  Therefore the supported submodule
\[
  \Gamma_Z(\mathcal M)\subseteq \mathcal M
\]
is defined as in \Cref{def:supported-submodule}: it is the maximal coherent
submodule supported on \(Z\), by \Cref{prop:support-truncation-standard}.

\begin{lemma}[Invariant support truncation]
\label{lem:invariant-support-truncation}
Let \(Z\subseteq X_s\) be a closed \(\mathcal K_s\)-stable subset and let
\[
  \mathcal M\in
  \Coh^{\fstr}_{\mathfrak K}
  (\mathscr D^\dagger_{\mathfrak X,\mathbb Q}).
\]
Then \(\Gamma_Z(\mathcal M)\) is a filtered strongly \(\mathfrak K\)-equivariant
coherent submodule of \(\mathcal M\).  Moreover,
\(\mathcal M/\Gamma_Z(\mathcal M)\) has no nonzero coherent
\(\mathscr D^\dagger_{\mathfrak X,\mathbb Q}\)-submodule supported on
\(Z\).  In particular, it has no nonzero filtered strongly equivariant submodule
supported on \(Z\).
\end{lemma}

\begin{proof}
The object \(\mathcal M\) is holonomic by
\Cref{cor:strong-equivariant-holonomic}, so \(\Gamma_Z(\mathcal M)\) is the
maximal supported submodule defined in \Cref{def:supported-submodule}.  We use
only the functorial support formalism after this holonomicity is known.  Since
\(\mathcal K_s\)-stability of \(Z\) gives
\[
  a^{-1}Z=p^{-1}Z
\]
on the special fiber of \(\mathfrak K\times_{\mathcal V}\mathfrak X\), the
smooth pullback compatibility for support/localization recalled in
\Cref{prop:standard-facts}\textup{(iv)} and
\cite[\S2.1, Lemmas 2.1.1--2.1.2]{HuygheSchmidtIntermediate} gives canonical
identifications
\[
  a^\sharp\Gamma_Z(\mathcal M)\simeq
  \Gamma_{a^{-1}Z}(a^\sharp\mathcal M),\qquad
  p^\sharp\Gamma_Z(\mathcal M)\simeq
  \Gamma_{p^{-1}Z}(p^\sharp\mathcal M).
\]
Transporting these identifications through the weak equivariant isomorphism
\(a^\sharp\mathcal M\simeq p^\sharp\mathcal M\) gives an isomorphism
\[
  a^\sharp\Gamma_Z(\mathcal M)\simeq
  p^\sharp\Gamma_Z(\mathcal M).
\]
The identity and cocycle conditions follow from the functoriality of support
truncation; no new choice is made.

Thus \(\Gamma_Z(\mathcal M)\) is a weakly equivariant coherent submodule of
\(\mathcal M\).  Since it is a coherent \(\mathscr D^\dagger\)-submodule
stable under the restricted weak equivariant structure,
\Cref{lem:fstr-subquotient-closure}\textup{(i)} gives the filtered strong
structure on \(\Gamma_Z(\mathcal M)\).  The quotient inherits the corresponding
filtered strongly equivariant structure by
\Cref{lem:fstr-subquotient-closure}\textup{(ii)}.

Finally, suppose that \(\mathcal M/\Gamma_Z(\mathcal M)\) has a nonzero coherent
\(\mathscr D^\dagger_{\mathfrak X,\mathbb Q}\)-submodule supported on
\(Z\).  Its inverse image in \(\mathcal M\) is coherent and supported on
\(Z\), and it strictly contains \(\Gamma_Z(\mathcal M)\), contradicting the
maximality in \Cref{prop:support-truncation-standard}.  Thus the quotient has no
nonzero coherent submodule supported on \(Z\), and a fortiori no nonzero
filtered strongly equivariant one.
\end{proof}

The induction used below is measured by the finite integer
\begin{equation}
\label{eq:orbit-count-invariant}
  n(\mathcal M)
  :=
  \#\{\,O\in\mathcal K_s\backslash X_s
  \mid O\subseteq\Supp(\mathcal M)\,\}.
\end{equation}
For a closed invariant subset \(Z\subseteq\Supp(\mathcal M)\), the
support truncation gives the exact sequence
\begin{equation}
\label{eq:support-truncation-exact-sequence}
  0\longrightarrow\Gamma_Z(\mathcal M)
  \longrightarrow\mathcal M
  \longrightarrow\mathcal M/\Gamma_Z(\mathcal M)
  \longrightarrow0.
\end{equation}
If \(O\) is open in \(\Supp(\mathcal M)\) and
\(Z=\Supp(\mathcal M)\setminus O\), then
\[
  n(\Gamma_Z(\mathcal M))<n(\mathcal M),
  \qquad
  n(\mathcal M/\Gamma_Z(\mathcal M))\leq n(\mathcal M),
\]
and every quotient of \(\mathcal M/\Gamma_Z(\mathcal M)\) supported on
\(Z\) has strictly smaller orbit count.  This is the numerical descent that
makes the orbital devissage terminate.

\subsection{Restriction to an orbit}
\label{subsec:restriction-orbit-regularity}

Let \(O\subseteq X_s\) be a \(\mathcal K_s\)-orbit and let
\[
  j_O:O\hookrightarrow X_s
\]
be the corresponding locally closed immersion.  We use the restriction functor
to the couple \((O,\overline O)\) recalled in
\Cref{subsec:restriction-locally-closed}.  We write \(\mathcal M|_O\) for this
restriction.  On the holonomic heart this functor is exact and is compatible
with characteristic varieties in the sense of
\Cref{prop:restriction-characteristic}.

\begin{lemma}[Restriction to an open orbit]
\label{lem:restriction-to-orbit-isocrystal}
Let
\[
  \mathcal M\in
  \Coh^{\fstr}_{\mathfrak K}
  (\mathscr D^\dagger_{\mathfrak X,\mathbb Q}).
\]
Let \(O\subseteq\Supp(\mathcal M)\) be a \(\mathcal K_s\)-orbit which is open
in \(\Supp(\mathcal M)\).  Then the restriction \(\mathcal M|_O\) is a
nonzero overconvergent isocrystal on the pair \((O,\overline O)\).  It carries
a natural \(\mathcal K\)-equivariant structure.
\end{lemma}

\begin{proof}
By \Cref{cor:strong-equivariant-holonomic}, \(\mathcal M\) is holonomic.  By
\Cref{prop:conormal-containment},
\[
  \Car(\mathcal M)
  \subseteq
  \bigcup_{S\in\mathcal K_s\backslash X_s}T^*_S X_s.
\]
Let
\[
  \rho_O:T^*X_s|_O\longrightarrow T^*O
\]
be the natural restriction map on cotangent bundles.  Since \(O\) is open in
\(\Supp(\mathcal M)\), no other orbit \(S\subseteq\Supp(\mathcal M)\),
\(S\neq O\), satisfies \(O\subseteq\overline S\).  Indeed, if
\(O=V\cap\Supp(\mathcal M)\) with \(V\subseteq X_s\) open and some point of
\(O\) belonged to \(\overline S\), then \(V\) would meet \(S\); hence
\(\varnothing\neq S\cap V\subseteq O\), contradicting the disjointness of
orbits.  Therefore, over the points of \(O\), the only conormal component of
\(\Car(\mathcal M)\) which appears is \(T^*_O X_s\), and
\(\rho_O(T^*_O X_s)=T^*_O O\), the zero section of \(T^*O\).  By
\Cref{prop:restriction-characteristic}, applied to the locally closed couple
\((O,\overline O)\), the characteristic variety of the restriction is controlled
by the cotangent projection associated with the pair restriction.
This is the restriction functor obtained by localization and support, not an
ordinary open restriction. Equivalently, at the level of reduced characteristic
varieties, the compatibility recalled in \Cref{prop:standard-facts} and
\cite[\S2.1, Lemmas 2.1.1--2.1.2]{HuygheSchmidtIntermediate} gives the
chain
\begin{equation}
\label{eq:restriction-characteristic-chain}
\begin{aligned}
  \Car(\mathcal M|_O)
  &\subseteq
  \rho_O\bigl(\Car(\mathcal M)\cap T^*X_s|_O\bigr) \\
  &\subseteq
  \rho_O(T^*_O X_s) \\
  &=T^*_O O.
\end{aligned}
\end{equation}
By the zero-section criterion \cite[Theorem 2.1.4]{HuygheSchmidtIntermediate},
\(\mathcal M|_O\) is an overconvergent isocrystal on the couple
\((O,\overline O)\).  It is nonzero because \(O\subseteq\Supp(\mathcal M)\).

The orbit \(O\) and its closure are \(\mathcal K_s\)-stable.  Therefore the
action and projection maps restrict to morphisms of couples
\[
  (\mathcal K_s\times O,\mathcal K_s\times\overline O)
  \longrightarrow
  (O,\overline O).
\]
Smooth base change for restriction to couples transports the weak equivariant
isomorphism \(a^\sharp\mathcal M\simeq p^\sharp\mathcal M\) to an isomorphism
between the pullbacks of \(\mathcal M|_O\).  The identity, cocycle and
infinitesimal compatibility conditions are inherited from \(\mathcal M\).  Thus
\(\mathcal M|_O\) is a \(\mathcal K\)-equivariant overconvergent isocrystal on
\((O,\overline O)\).
\end{proof}

\begin{remark}[Arbitrary orbits and Whitney \((a)\)]
\label{rem:arbitrary-orbits-whitney-a}
The preceding lemma is deliberately stated only for an orbit open in
\(\Supp(\mathcal M)\), which is the only case used in the proof.  If one works
in a context where the finite orbit partition is known, or assumed, to satisfy
Whitney's condition \((a)\), the same cotangent argument extends to an arbitrary
orbit.  Indeed, Whitney \((a)\) says that whenever a stratum \(S\) approaches a
point \(x\in O\), every limiting tangent space to \(S\) contains
\(T_xO\).  Therefore every limiting covector conormal to \(S\) over \(x\)
annihilates \(T_xO\), and hence maps to zero in \(T_x^*O\).  Equivalently,
\[
  \rho_O\bigl(\overline{T^*_S X_s}\cap T^*X_s|_O\bigr)
  \subseteq T^*_O O.
\]
This Whitney-based observation is not used in the article; it only explains how
one could recover the corresponding statement for a non-open orbit under the
additional Whitney \((a)\)-regularity hypothesis.  For the tangent-space form of
Whitney conditions, see \cite[Part I, \S1.2]{GoreskyMacPherson} and
\cite[\S8.3]{KashiwaraSchapira}.
\end{remark}

\subsection{No exactness assertion for intermediate extension}

The proof below does not use intermediate extension as an exact functor.  This
point is important: in general the functor \(j_{O,!+}\) is not exact, and the
paper makes no claim to the contrary.  For regularity we only extract one
intermediate-extension subobject from the open orbit and handle the remaining
boundary-supported quotient by induction on the number of orbits.

\subsection{The open-orbit extraction step}

\begin{lemma}[Open-orbit middle-extension extraction inside the overholonomic heart]
\label{lem:middle-extension-extraction}
Let
\[
  \mathcal M\in
  F\text{-}\Coh^{\fstr}_{\mathfrak K}
  (\mathscr D^\dagger_{\mathfrak X,\mathbb Q}).
\]
Assume, in addition, that the underlying arithmetic
\(\mathscr D^\dagger\)-module of \(\mathcal M\) belongs to the geometrically
overholonomic heart \(\Ovhol(X_s/K)\).  Let \(O\) be an orbit open in
\(\Supp(\mathcal M)\).  Put
\[
  Z:=\Supp(\mathcal M)\setminus O,
  \qquad
  Q:=\mathcal M/\Gamma_Z(\mathcal M),
  \qquad
  E:=Q|_O\simeq\mathcal M|_O.
\]
Then \(E\) is an equivariant overconvergent \(F\)-isocrystal and \(Q\) contains
a canonical Frobenius filtered strongly equivariant submodule
\[
  J_O(Q)\simeq j_{O,!+}(E)
\]
whose quotient \(Q/J_O(Q)\) is supported on \(Z\).
\end{lemma}

\begin{proof}
By \Cref{lem:invariant-support-truncation}, the submodule
\(\Gamma_Z(\mathcal M)\) is stable under the weak equivariant structure of
\(\mathcal M\), and the quotient \(Q\) has no nonzero coherent
\(\mathscr D^\dagger_{\mathfrak X,\mathbb Q}\)-submodule supported on \(Z\).
Because \(\mathcal M\) lies in the overholonomic heart, the support/localization
formalism recalled in \Cref{prop:standard-facts} produces the degree-zero
supported subobject attached to \(Z\) inside that heart.  Its underlying
holonomic coherent submodule is the maximal supported submodule
\(\Gamma_Z(\mathcal M)\) by \Cref{prop:support-truncation-standard}.  Hence
\(\Gamma_Z(\mathcal M)\) lies in the overholonomic heart, and the quotient
\(Q\) is the corresponding cokernel in that abelian heart.  Thus both
\(\Gamma_Z(\mathcal M)\) and \(Q\) are geometrically overholonomic.  Moreover,
\(Q\in F\text{-}\Coh^{\fstr}_{\mathfrak K}\) by
\Cref{lem:fstr-subquotient-closure}\textup{(ii)}.  Since \(O\) is open in
\(\Supp(Q)\) and \(Q\) is filtered strongly equivariant, the open-orbit
restriction lemma, \Cref{lem:restriction-to-orbit-isocrystal}, shows that
\(E=Q|_O\) is an equivariant overconvergent \(F\)-isocrystal.

The cokernel just constructed is the same whether computed in the holonomic
heart or in the overholonomic heart.  Indeed, the inclusion of the
overholonomic derived category into the holonomic one is triangulated, fully
faithful, and closed under cones as recalled in \Cref{rem:ovhol-heart-use};
therefore a short exact sequence in the holonomic heart whose first two terms
are overholonomic has an overholonomic cone concentrated in degree zero.  This
identifies the quotient \(\mathcal M/\Gamma_Z(\mathcal M)\) with the cokernel in
the overholonomic heart.

The comparison with the intermediate extension is now made entirely inside the
overholonomic heart.  This is the point of the extra hypothesis: no
adjunction from the overholonomic six-functor formalism is applied to an object
which has not yet been placed in that heart.

Let \(j=j_O:O\hookrightarrow X_s\).  Since \(Q\in\Ovhol(X_s/K)\), the standard
adjunctions in the overholonomic heart give canonical morphisms
\[
  c_Q:j^0_!E\longrightarrow Q,
  \qquad
  u_Q:Q\longrightarrow j^0_+E,
\]
whose restrictions to \(O\) are the identity of \(E\).  The kernel of \(u_Q\)
restricts to zero on \(O\); hence its support is contained in
\(\Supp(Q)\setminus O\subseteq Z\).  Since \(Q\) has no nonzero submodule
supported on \(Z\), the morphism \(u_Q\) is a monomorphism.

Under the adjunction identifications, the composite
\[
  j^0_!E\xrightarrow{\ c_Q\ } Q
  \xrightarrow{\ u_Q\ } j^0_+E
\]
is precisely the canonical morphism
\(\theta^0_{j,E}:j^0_!E\to j^0_+E\).  By definition,
\[
  j_{O,!+}(E)=\operatorname{Im}(\theta^0_{j,E})
\]
in the overholonomic heart.  Since \(u_Q\) is a monomorphism, the image
\[
  J_O(Q):=\operatorname{Im}(c_Q)\subseteq Q
\]
is identified by \(u_Q\) with \(j_{O,!+}(E)\subseteq j^0_+E\).  Thus
\[
  J_O(Q)\simeq j_{O,!+}(E)
\]
as an object of the overholonomic heart and as a coherent arithmetic
\(\mathscr D^\dagger\)-submodule of \(Q\).

The quotient \(Q/J_O(Q)\) restricts to zero on \(O\), because both \(Q\) and
\(J_O(Q)\) restrict to the same object \(E\).  Hence
\[
  \Supp(Q/J_O(Q))\subseteq Z.
\]
The construction of \(J_O(Q)\) as the image of the adjunction morphism
\(j^0_!E\to Q\) is functorial in \(Q\) and in the identity of \(E\) after
restriction to \(O\).  Applying the same construction to \(a^\sharp Q\) and
\(p^\sharp Q\), and using the weak equivariant isomorphism
\(a^\sharp Q\simeq p^\sharp Q\), identifies \(a^\sharp J_O(Q)\) with
\(p^\sharp J_O(Q)\).  Hence \(J_O(Q)\) is weakly equivariant.  The same
functoriality shows that the Frobenius structure on \(Q\) restricts to
\(J_O(Q)\) and descends to \(Q/J_O(Q)\).  Finally, since \(Q\) is Frobenius
filtered strongly equivariant, \Cref{lem:fstr-subquotient-closure} gives the
filtered strong structures on \(J_O(Q)\) and on \(Q/J_O(Q)\).  This proves the
claim.
\end{proof}

\subsection{Equivariant orbital devissage}

The following proposition is the central structural result used for the orbit
classification.  It is stated with an explicit geometric-overholonomicity
hypothesis; this prevents the overholonomic adjunctions defining \(j_{!+}\)
from being applied outside their natural domain.  It is the arithmetic
equivariant analogue of the usual stratified devissage for regular holonomic
\(D\)-modules.

\begin{proposition}[Frobenius equivariant orbital devissage]
\label{prop:equivariant-orbital-devissage}
Let
\[
  \mathcal M\in
  F\text{-}\Coh^{\fstr}_{\mathfrak K}
  (\mathscr D^\dagger_{\mathfrak X,\mathbb Q}).
\]
Assume, in addition, that the underlying arithmetic
\(\mathscr D^\dagger\)-module of \(\mathcal M\) is geometrically
overholonomic.  Then \(\mathcal M\) admits a finite filtration by Frobenius filtered strongly
\(\mathfrak K\)-equivariant coherent submodules
\[
  0=\mathcal M_0\subseteq \mathcal M_1\subseteq \cdots
  \subseteq \mathcal M_r=\mathcal M
\]
such that every successive quotient is of the form
\[
  \mathcal M_i/\mathcal M_{i-1}
  \simeq
  j_{O_i,!+}(E_i),
\]
where \(O_i\subseteq X_s\) is a \(\mathcal K_s\)-orbit and \(E_i\) is an
equivariant overconvergent \(F\)-isocrystal satisfying
\[
  [E_i]\in
  \Adm^{\Phi,\fstr}_{\mathcal K}(O_i,\overline O_i).
\]
The objects \(E_i\) need not be simple.
\end{proposition}

\begin{proof}
We argue by induction on the number of \(\mathcal K_s\)-orbits contained in
\(\Supp(\mathcal M)\).  The zero module is clear.  If \(\mathcal M\neq0\),
choose an orbit \(O\) open in \(\Supp(\mathcal M)\) and set
\[
  Z:=\Supp(\mathcal M)\setminus O.
\]
The closed subset \(Z\) is \(\mathcal K_s\)-stable and contains strictly fewer
orbits.  Put
\[
  Q:=\mathcal M/\Gamma_Z(\mathcal M),
  \qquad
  E:=Q|_O.
\]
By \Cref{lem:middle-extension-extraction}, \(Q\) contains a canonical Frobenius
filtered strongly equivariant submodule
\[
  J_O(Q)\simeq j_{O,!+}(E),
\]
and \(Q/J_O(Q)\) is supported on \(Z\).  The piece \(j_{O,!+}(E)\) already has
the required form and satisfies
\([E]\in\Adm^{\Phi,\fstr}_{\mathcal K}(O,\overline O)\) by construction.  The
objects \(\Gamma_Z(\mathcal M)\) and \(Q/J_O(Q)\) are Frobenius filtered strongly
equivariant, geometrically overholonomic, and supported on \(Z\), hence involve
strictly fewer orbits than \(\mathcal M\).  The induction hypothesis applies to
them.

We now form the filtration explicitly.  First filter
\(\Gamma_Z(\mathcal M)\) by the induction hypothesis.  Next lift the subobject
\(J_O(Q)\subseteq Q\) to \(\mathcal M\) through the quotient map
\(\mathcal M\twoheadrightarrow Q\); this contributes the single graded piece
\(j_{O,!+}(E)\).  Finally lift a filtration of \(Q/J_O(Q)\), again supplied by
induction because its support is contained in \(Z\).  All subobjects used in
this construction are coherent and Frobenius filtered strongly equivariant by
\Cref{lem:fstr-subquotient-closure}.  The resulting finite filtration of
\(\mathcal M\) has only quotients of the required form \(j_{O',!+}(E')\), with
\([E']\in\Adm^{\Phi,\fstr}_{\mathcal K}(O',\overline O')\).  The induction
terminates because the number of orbits in the support strictly decreases
whenever a term is supported on \(Z\).
\end{proof}

\subsection{The regularity theorem}

\begin{theorem}[Arithmetic Kashiwara regularity]
\label{thm:arithmetic-kashiwara-regularity}
Assume \Cref{setup:basic}.  Let \(\mathcal K\subseteq\mathcal G\) be a smooth
closed subgroup scheme acting on \(\mathfrak X\), and assume that
\(\mathcal K_s\) acts on \(X_s\) with finitely many orbits and satisfies
\Cref{hyp:orbit-separability}.  Then:
\begin{enumerate}[label=\textup{(\roman*)}]
\item every object of
\[
  \Coh^{\fstr}_{\mathfrak K}
  \bigl(\mathscr D^\dagger_{\mathfrak X,\mathbb Q}\bigr)
\]
is holonomic;
\item every object of
\[
  F\text{-}\Coh^{\fstr}_{\mathfrak K}
  \bigl(\mathscr D^\dagger_{\mathfrak X,\mathbb Q}\bigr)
\]
is geometrically overholonomic.
\end{enumerate}
\end{theorem}

\begin{proof}
Let \(\mathcal M\) be filtered strongly \(\mathfrak K\)-equivariant and coherent.
The conormal-containment argument gives
\[
  \Car(\mathcal M)\subseteq \bigcup_O T^*_OX_s.
\]
Since there are finitely many orbits, \Cref{cor:strong-equivariant-holonomic}
shows that \(\mathcal M\) is holonomic.  This proves \textup{(i)}.

Assume now that \(\mathcal M\) also carries a Frobenius structure compatible
with the filtered strongly equivariant structure.  By \textup{(i)}, its
underlying coherent arithmetic \(\mathscr D^\dagger\)-module is holonomic.
Caro's stability of Frobenius holonomicity over smooth projective formal schemes, recorded in
\Cref{prop:standard-facts}\textup{(iii)}, then implies that \(\mathcal M\) is
overholonomic in the Frobenius sense.  Huyghe--Schmidt's stable-after-base-change
convention, recalled in \Cref{prop:ovhol-stability}, says that the base-change
condition is automatic for Frobenius overholonomic objects.  Hence
\(\mathcal M\) is geometrically overholonomic.  This proves \textup{(ii)}.
\end{proof}

\begin{definition}[Frobenius overholonomic strongly equivariant objects]
\label{def:ovhol-fstr}
We write
\[
  F\text{-}\Ovhol^{\fstr}_{\mathfrak K}
  (\mathscr D^\dagger_{\mathfrak X,\mathbb Q})
\]
for the full subcategory of
\(F\text{-}\Coh^{\fstr}_{\mathfrak K}(\mathscr D^\dagger_{\mathfrak X,\mathbb Q})\)
whose underlying arithmetic \(\mathscr D^\dagger\)-module belongs to the
geometrically overholonomic heart \(\Ovhol(X_s/K)\).
\end{definition}

\begin{corollary}
\label{cor:coh-equals-ovhol}
Under the hypotheses of \Cref{thm:arithmetic-kashiwara-regularity},
\[
  F\text{-}\Coh^{\fstr}_{\mathfrak K}
  (\mathscr D^\dagger_{\mathfrak X,\mathbb Q})
  =
  F\text{-}\Ovhol^{\fstr}_{\mathfrak K}
  (\mathscr D^\dagger_{\mathfrak X,\mathbb Q}).
\]
The corresponding statement without Frobenius is not asserted.
\end{corollary}

\begin{proof}
The inclusion from right to left is immediate from the definition.  The other
inclusion is \Cref{thm:arithmetic-kashiwara-regularity}.
\end{proof}

\begin{remark}
\label{rem:regularity-strong-not-weak}
The theorem is a statement about filtered strong equivariance.  We do not claim
that an arbitrary weakly equivariant coherent arithmetic
\(\mathscr D^\dagger\)-module is geometrically overholonomic.  The infinitesimal
compatibility, realized at finite level as in \textup{(SE2)}, is what forces the
moment-zero characteristic containment used above.
\end{remark}

\section{Orbit classification of simple equivariant arithmetic
\texorpdfstring{\(\mathscr D^\dagger\)}{Ddagger}-modules}
\label{sec:classification}

Throughout this section we keep the hypotheses of
\Cref{thm:arithmetic-kashiwara-regularity}. In particular, \(\mathcal K_s\) acts
on \(X_s\) with finitely many orbits and satisfies
\Cref{hyp:orbit-separability}.  The classification is stated only in the
Frobenius filtered strongly equivariant category, where
\Cref{thm:arithmetic-kashiwara-regularity} gives geometric overholonomicity.

\subsection{The support of a simple equivariant object}

\begin{lemma}[Support and reconstruction for a simple Frobenius equivariant object]
\label{lem:support-simple}
Let
\[
  \mathcal M\in
  F\text{-}\Coh^{\fstr}_{\mathfrak K}
  \bigl(\mathscr D^\dagger_{\mathfrak X,\mathbb Q}\bigr)
\]
be simple. Then there exists a unique \(\mathcal K_s\)-orbit \(O\subseteq X_s\)
such that
\[
  \Supp(\mathcal M)=\overline O.
\]
Moreover, \(O\) is open and dense in \(\Supp(\mathcal M)\), and there is a
canonical isomorphism
\[
  \mathcal M\simeq j_{O,!+}(\mathcal M|_O).
\]
\end{lemma}

\begin{proof}
By \Cref{thm:arithmetic-kashiwara-regularity}, \(\mathcal M\) is geometrically
overholonomic.  By \Cref{cor:support-union-orbits}, its support is a finite
union of \(\mathcal K_s\)-orbit closures.  Choose an orbit \(O\) open in
\(\Supp(\mathcal M)\); the argument below will show that this choice is forced.
Put
\[
  Z:=\Supp(\mathcal M)\setminus O.
\]
Then \(Z\) is closed in \(X_s\), \(\mathcal K_s\)-stable, and contains fewer
orbits than \(\Supp(\mathcal M)\).  By
\Cref{lem:invariant-support-truncation}, \(\Gamma_Z(\mathcal M)\) is a Frobenius
filtered strongly equivariant subobject of \(\mathcal M\).  It cannot equal
\(\mathcal M\), because \(\mathcal M|_O\neq0\) while
\(\Gamma_Z(\mathcal M)|_O=0\).  Simplicity therefore gives
\(\Gamma_Z(\mathcal M)=0\).

Apply \Cref{lem:middle-extension-extraction} to \(\mathcal M\).  Since
\(\Gamma_Z(\mathcal M)=0\), the quotient \(Q\) in that lemma is just
\(\mathcal M\).  Hence \(\mathcal M\) contains a canonical Frobenius filtered
strongly equivariant subobject
\[
  J_O(\mathcal M)\simeq j_{O,!+}(\mathcal M|_O),
\]
whose quotient is supported on \(Z\).  This subobject is nonzero, because it
restricts to \(\mathcal M|_O\) on \(O\).  By simplicity, it must be all of
\(\mathcal M\).  Thus
\[
  \mathcal M\simeq j_{O,!+}(\mathcal M|_O).
\]
Taking supports and using \Cref{lem:support-middle-extension} gives
\(\Supp(\mathcal M)=\overline O\).  The orbit \(O\) is therefore dense and open
in the support.  It remains only to justify uniqueness.  For algebraic group
orbits in the finite orbit stratification used here, the frontier relation is
strict: since \(O\) is smooth, locally closed, and open dense in \(\overline O\),
the closed boundary \(\overline O\setminus O\) has dimension strictly smaller
than \(O\).  Hence every orbit contained in that boundary has smaller dimension
and cannot be dense in \(\overline O\).  Thus no orbit distinct from \(O\) can
be dense in \(\overline O\), and the open orbit attached to \(\mathcal M\) is
unique.
\end{proof}

\subsection{Restriction to the open orbit}

\begin{lemma}[Irreducibility on the open orbit]
\label{lem:restriction-irreducible}
Let \(\mathcal M\in F\text{-}\Coh^{\fstr}_{\mathfrak K}\) be simple and let
\(O\subseteq \Supp(\mathcal M)\) be the open orbit. Then \(\mathcal M|_O\) is an
irreducible object of
\[
  F\text{-}\Isoc^{\dagger}_{\mathcal K}(O,\overline O).
\]
\end{lemma}

\begin{proof}
Put \(E=\mathcal M|_O\).  By \Cref{thm:arithmetic-kashiwara-regularity},
\(\mathcal M\) is geometrically overholonomic, and by
\Cref{lem:restriction-equvariant-isocrystal}, \(E\) is a \(\mathcal K\)-equivariant
overconvergent \(F\)-isocrystal on \((O,\overline O)\).  By
\Cref{lem:support-simple}, we have a canonical identification
\[
  \mathcal M\simeq j_{O,!+}(E).
\]
Suppose that
\(0\neq E'\subsetneq E\) is a proper subobject in
\(F\text{-}\Isoc^{\dagger}_{\mathcal K}(O,\overline O)\).  The intermediate-extension
construction is functorial in the overholonomic heart, so the inclusion
\(E'\hookrightarrow E\) induces a morphism
\[
  j_{O,!+}(E')\longrightarrow j_{O,!+}(E)\simeq\mathcal M.
\]
The source carries the Frobenius weak equivariant structure supplied by
\Cref{prop:eq-intermediate-extension}, and the morphism is equivariant by
functoriality.  Its restriction to \(O\) is the original nonzero inclusion
\(E'\hookrightarrow E\); therefore its image \(I\subseteq\mathcal M\) is nonzero.
Because \(I\) is the image of an equivariant morphism, it is weakly
\(\mathfrak K\)-equivariant; as a coherent \(\mathscr D^\dagger\)-submodule of
\(\mathcal M\), it is then endowed with the induced Frobenius filtered strongly
equivariant structure by \Cref{lem:fstr-subquotient-closure}.  Simplicity of
\(\mathcal M\) gives \(I=\mathcal M\).  Restriction to the open orbit is exact
on the relevant overholonomic heart, as recalled in \Cref{app:t-structure}, and
is compatible with images; hence the restricted image is
\[
  I|_O=E.
\]
But the same restricted image is the image of \(E'\hookrightarrow E\), namely
\(E'\).  Thus \(E'=E\), contradicting the assumption that \(E'\) was proper.
Therefore \(E\) is irreducible in
\(F\text{-}\Isoc^{\dagger}_{\mathcal K}(O,\overline O)\).
\end{proof}

\subsection{Reconstruction from the open orbit}

\begin{lemma}[Reconstruction from the open orbit]
\label{lem:reconstruction}
Let \(\mathcal M\in F\text{-}\Coh^{\fstr}_{\mathfrak K}\) be simple and let \(O\subseteq\Supp(\mathcal M)\) be the open
orbit. Then there is a canonical isomorphism
\[
  \mathcal M\simeq j_{O,!+}(\mathcal M|_O).
\]
\end{lemma}

\begin{proof}
This is the reconstruction statement already proved in
\Cref{lem:support-simple}.  The canonicity is the canonicity of the extraction
morphism in \Cref{lem:middle-extension-extraction}, applied after the vanishing
\(\Gamma_{\Supp(\mathcal M)\setminus O}(\mathcal M)=0\) forced by simplicity.
\end{proof}

\subsection{Restriction image and classification theorem}

\begin{theorem}[Restriction, reconstruction and Frobenius filtered orbit classification]
\label{thm:orbit-classification}
For a simple object
\(\mathcal M\in F\text{-}\Coh^{\fstr}_{\mathfrak K}
(\mathscr D^\dagger_{\mathfrak X,\mathbb Q})\), let \(O(\mathcal M)\) be the
unique open orbit in its support.  The assignment
\[
  \rho(\mathcal M):=\bigl(O(\mathcal M),\mathcal M|_{O(\mathcal M)}\bigr)
\]
defines an injective map of isomorphism classes
\[
  \rho:
  \Irr F\text{-}\Coh^{\fstr}_{\mathfrak K}
  \bigl(\mathscr D^\dagger_{\mathfrak X,\mathbb Q}\bigr)
  \longrightarrow
  \coprod_{O\in \mathcal K_s\backslash X_s}
  \Irr F\text{-}\Isoc^{\dagger}_{\mathcal K}(O,\overline O).
\]
Its image is exactly
\[
  \coprod_{O\in \mathcal K_s\backslash X_s}
  \Adm^{\Phi,\fstr}_{\mathcal K}(O,\overline O),
\]
where \(\Adm^{\Phi,\fstr}\) is the Frobenius image locus of
\Cref{def:F-filtered-admissible-isocrystal}.  On this image the inverse of
\(\rho\) is the intermediate-extension construction
\[
  (O,E)\longmapsto j_{O,!+}(E).
\]
Equivalently, the assignment \((O,E)\mapsto j_{O,!+}(E)\) gives a bijection
from the displayed image locus onto
\(\Irr F\text{-}\Coh^{\fstr}_{\mathfrak K}
(\mathscr D^\dagger_{\mathfrak X,\mathbb Q})\).
\end{theorem}

\begin{remark}[Comparison with Huyghe--Schmidt]
\label{rem:comparison-HS-theorem-8}
Huyghe--Schmidt \cite[Theorem 2.3.4]{HuygheSchmidtIntermediate} classify
irreducible arithmetic \(\mathscr D^\dagger\)-modules which are overholonomic
after any base change by smooth locally closed supports together with
overconvergent isocrystals already belonging to the corresponding overholonomic
coefficient category.
\Cref{thm:orbit-classification} is different in logical direction.  In the
Frobenius filtered strongly equivariant range, \Cref{thm:arithmetic-kashiwara-regularity}
first gives holonomicity by the moment-map argument and then geometric
overholonomicity by Caro's smooth-projective holonomicity-stability theorem recalled in
\Cref{prop:standard-facts}.  The finite-orbit and orbit-separability hypotheses
then force the support of a simple equivariant object to be the closure of a
single \(\mathcal K_s\)-orbit.  Only at that point does the
Huyghe--Schmidt intermediate-extension classification enter, now restricted to
the \(\mathcal K_s\)-orbit stratification and to equivariant \(F\)-coefficient
classes satisfying the Frobenius admissibility condition.
\end{remark}

\begin{proof}
Let \(\mathcal M\) be a simple Frobenius filtered strongly equivariant coherent
arithmetic \(\mathscr D^\dagger\)-module. By \Cref{lem:support-simple}, there is
a unique open orbit \(O=O(\mathcal M)\subseteq\Supp(\mathcal M)\). By
\Cref{lem:restriction-irreducible}, \(\mathcal M|_O\) is irreducible in
\(F\text{-}\Isoc^{\dagger}_{\mathcal K}(O,\overline O)\). Thus \(\rho\) is well-defined.
By \Cref{lem:reconstruction},
\[
  \mathcal M\simeq j_{O,!+}(\mathcal M|_O).
\]
Consequently \([\mathcal M|_O]\in
\Adm^{\Phi,\fstr}_{\mathcal K}(O,\overline O)\), so the image of \(\rho\) is contained
in the displayed admissible locus.

Conversely, if \([E]\in\Adm^{\Phi,\fstr}_{\mathcal K}(O,\overline O)\), then by
\Cref{def:F-filtered-admissible-isocrystal} the object \(j_{O,!+}(E)\) is
filtered strongly equivariant, and by \Cref{prop:simplicity-intermediate-extension}
it is simple.  Moreover \Cref{lem:restriction-intermediate-extension} gives
\[
  \rho\bigl(j_{O,!+}(E)\bigr)=(O,E).
\]
Therefore the image of \(\rho\) is exactly the admissible locus and
\((O,E)\mapsto j_{O,!+}(E)\) is a right inverse to \(\rho\) on that locus.

Finally, if two simple Frobenius filtered strongly equivariant modules have the same
image under \(\rho\), then both are isomorphic to the same intermediate
extension by \Cref{lem:reconstruction}.  Hence \(\rho\) is injective.  Equivalently,
if \(j_{O,!+}(E)\simeq j_{O',!+}(E')\), then taking supports and using
\Cref{lem:support-middle-extension} gives \(\overline O=\overline{O'}\).  In a
finite orbit stratification an orbit is the unique open orbit in its own
closure, so \(O=O'\); restriction to \(O\) then gives \(E\simeq E'\).
\end{proof}

\begin{corollary}[Full Frobenius coefficient form under filtered extension admissibility]
\label{cor:classification-full-isocrystals-under-FEH}
Assume, in addition, that for every orbit \(O\) every irreducible equivariant overconvergent \(F\)-isocrystal has Frobenius
filtered strongly equivariant intermediate extension, equivalently
\[
  \Adm^{\Phi,\fstr}_{\mathcal K}(O,\overline O)
  =
  \Irr F\text{-}\Isoc^{\dagger}_{\mathcal K}(O,\overline O).
\]
Then the classification of \Cref{thm:orbit-classification} becomes
\[
  \coprod_{O\in \mathcal K_s\backslash X_s}
  \Irr F\text{-}\Isoc^{\dagger}_{\mathcal K}(O,\overline O)
  \xrightarrow{\;\sim\;}
  \Irr F\text{-}\Coh^{\fstr}_{\mathfrak K}
  \bigl(\mathscr D^\dagger_{\mathfrak X,\mathbb Q}\bigr).
\]
\end{corollary}

\begin{remark}
\label{rem:classification-interpretation}
The classification theorem above is the arithmetic analogue, at trivial
character, of the classical orbit classification of equivariant \(D\)-modules:
\(K\)-orbits and equivariant local systems are replaced by \(\mathcal K_s\)-orbits
and equivariant overconvergent \(F\)-isocrystals whose intermediate extensions are Frobenius filtered strongly equivariant.
\end{remark}

\section{Crystalline distribution algebras at trivial character}
\label{sec:distributions}

Set
\[
  A_0:=
  \Gamma\bigl(\mathfrak X,
  \mathscr D^\dagger_{\mathfrak X,\mathbb Q}\bigr).
\]
By \Cref{prop:ABB-trivial}, this algebra is identified with the corresponding
crystalline distribution algebra at trivial central character.  In this
section no additional equivariant localization theorem is asserted: all
algebraic equivariance is obtained by transporting the geometric category
through the arithmetic Beilinson--Bernstein equivalence.

Let
\[
  \Gamma:
  \Coh\bigl(\mathscr D^\dagger_{\mathfrak X,\mathbb Q}\bigr)
  \longrightarrow
  A_0\operatorname{-mod}_{\mathrm{fp}}
\]
be the global section functor of \Cref{prop:ABB-trivial}, and let \(\Loc\)
denote its quasi-inverse.

\begin{definition}[Frobenius filtered strongly equivariant \(A_0\)-modules]
\label{def:strong-equivariant-A0-modules}
A Frobenius filtered strongly \(\mathfrak K\)-equivariant finitely presented
left \(A_0\)-module is a finitely presented left \(A_0\)-module \(M\), together
with a chosen object structure on \(\Loc(M)\) in
\[
  F\text{-}\Coh^{\fstr}_{\mathfrak K}
  \bigl(\mathscr D^\dagger_{\mathfrak X,\mathbb Q}\bigr).
\]
Morphisms are the \(A_0\)-linear maps whose localizations are morphisms in this
geometric Frobenius filtered strongly equivariant category.  The resulting
transported category is denoted by
\[
  A_0\operatorname{-mod}^{\Phi,\fstr}_{\mathfrak K,\mathrm{fp}}.
\]
Equivalently, it is the essential image of
\(F\text{-}\Coh^{\fstr}_{\mathfrak K}(\mathscr D^\dagger_{\mathfrak X,\mathbb Q})\)
under the global section functor, with its transported structure.  Thus two
different Frobenius/equivariant structures on the same underlying
\(A_0\)-module are not identified unless they are isomorphic in the transported
category.
\end{definition}

\begin{proposition}[Equivariant arithmetic localization as transported structure]
\label{prop:equivariant-ABB-trivial}
The arithmetic Beilinson--Bernstein equivalence at trivial character restricts,
by transport of structure, to an exact equivalence
\[
  \Gamma:
  F\text{-}\Coh^{\fstr}_{\mathfrak K}
  \bigl(\mathscr D^\dagger_{\mathfrak X,\mathbb Q}\bigr)
  \xrightarrow{\;\sim\;}
  A_0\operatorname{-mod}^{\Phi,\fstr}_{\mathfrak K,\mathrm{fp}}.
\]
The exact structure, kernels, cokernels, and simple objects on the right are
those transported from the geometric category through \(\Gamma\).  Its
quasi-inverse is \(\Loc\).
\end{proposition}

\begin{proof}
The ordinary arithmetic Beilinson--Bernstein theorem gives an equivalence
\[
  \Gamma:
  \Coh\bigl(\mathscr D^\dagger_{\mathfrak X,\mathbb Q}\bigr)
  \xrightarrow{\;\sim\;}
  A_0\operatorname{-mod}_{\mathrm{fp}},
\]
with quasi-inverse \(\Loc\). By definition,
\(A_0\operatorname{-mod}^{\Phi,\fstr}_{\mathfrak K,\mathrm{fp}}\)
is the essential image of the Frobenius filtered strongly equivariant coherent category under
this equivalence. Hence, if
\(M\in A_0\operatorname{-mod}^{\Phi,\fstr}_{\mathfrak K,\mathrm{fp}}\), then
there exists
\[
  \mathcal N\in
  F\text{-}\Coh^{\fstr}_{\mathfrak K}
  \bigl(\mathscr D^\dagger_{\mathfrak X,\mathbb Q}\bigr)
\]
with \(M\simeq \Gamma(\mathcal N)\), together with the transported
Frobenius filtered strongly equivariant structure. Since \(\Loc\) is the
quasi-inverse of \(\Gamma\), we have a canonical isomorphism
\[
  \Loc(M)\simeq \Loc\Gamma(\mathcal N)\simeq \mathcal N,
\]
and therefore \(\Loc(M)\) belongs to the filtered strongly equivariant
geometric category by definition of the transported algebraic category.
Conversely, if \(\mathcal M\) is Frobenius filtered strongly equivariant, then
\(\Gamma(\mathcal M)\), equipped with the transported structure, belongs to the
essential image by construction. Thus \(\Gamma\) and \(\Loc\) restrict to
mutually quasi-inverse exact equivalences.  No independent finite-level
construction on the algebra side is asserted here; the algebraic equivariant
condition is a transport of structure through localization.
\end{proof}

\begin{corollary}[Distribution-side classification]
\label{thm:distribution-classification}
The global section functor induces a bijection
\[
  \coprod_{O\in \mathcal K_s\backslash X_s}
  \Adm^{\Phi,\fstr}_{\mathcal K}(O,\overline O)
  \xrightarrow{\;\sim\;}
  \Irr
  \bigl(A_0\operatorname{-mod}^{\Phi,\fstr}_{\mathfrak K,\mathrm{fp}}\bigr),
\]
where \(\Irr\) refers to simple objects in the transported Frobenius filtered
strongly equivariant category of \Cref{def:strong-equivariant-A0-modules}.  The
object corresponding to \((O,E)\) has underlying finitely presented
\(A_0\)-module
\[
  \Gamma\bigl(\mathfrak X,j_{O,!+}(E)\bigr),
\]
equipped with the transported Frobenius filtered strongly equivariant
structure.
\end{corollary}

\begin{proof}
Transport \Cref{thm:orbit-classification} through the equivalence of
\Cref{prop:equivariant-ABB-trivial}. The displayed formula is precisely the
global section module of the corresponding simple arithmetic
\(\mathscr D^\dagger\)-module.
\end{proof}

\begin{remark}
\label{rem:distribution-side-interpretation}
\Cref{thm:distribution-classification} is the representation-theoretic form of
the orbit classification. It says that the simple objects in the transported
Frobenius filtered strongly \(\mathfrak K\)-equivariant category over the
crystalline distribution algebra at trivial character are controlled by the
geometry of the \(\mathcal K_s\)-orbits on the special fiber of the formal flag
variety, together with irreducible Frobenius equivariant overconvergent
coefficient data satisfying the displayed admissibility condition on those
orbits.  The statement is not an additional classification of all finitely
presented modules over the crystalline distribution algebra; it is exactly the
transport through arithmetic localization of the geometric classification in
\Cref{thm:orbit-classification}.
\end{remark}

\section{Examples and geometric models}
\label{sec:examples}

The purpose of this section is to spell out the admissible-locus form of the
classification in several basic geometric situations.  These examples are not
intended to reprove the Bruhat-cell calculations of Huyghe--Schmidt
\cite[Section~5]{HuygheSchmidtIntermediate}.  In the Borel case, the
constant-coefficient Bruhat-cell objects recover the intermediate extensions
already identified there; the additional information in the present framework
is the filtered strongly equivariant structure, the stabilizer-equivariant
coefficient data, and the restriction to the admissible locus.  The examples
are deliberately elementary, but they serve two roles. First, they show how the
actual parameter set used in the theorem,
\[
  \coprod_{O\in \mathcal K_s\backslash X_s}
  \Adm^{\Phi,\fstr}_{\mathcal K}(O,\overline O),
\]
is read in practice in terms of orbits, stabilizers and intermediate
extensions.  Thus every coefficient object displayed below is understood either
as an explicitly verified admissible Frobenius coefficient object or as an
object already lying in the admissible locus.  Second, they illustrate how the same geometric data produce simple
modules over the crystalline distribution algebra
\[
  A_0=\Gamma(\mathfrak X,\mathscr D^\dagger_{\mathfrak X,\mathbb Q})
\]
by taking global sections.

\subsection{The transitive case: \texorpdfstring{\(\mathcal K=\mathcal G\)}{K=G}}

Let
\[
  \mathfrak X=\widehat{\mathcal G/\mathcal B}
\]
be the formal flag variety and take
\[
  \mathcal K=\mathcal G.
\]
Then \(\mathcal G_s\) acts transitively on
\[
  X_s=\mathcal G_s/\mathcal B_s.
\]
There is therefore a single orbit,
\[
  O=X_s.
\]
In this case \(\overline O=O=X_s\) and the intermediate extension functor is
just the identity functor. Therefore Theorem~\ref{thm:orbit-classification}
gives
\[
  \operatorname{Irr}
  F\text{-}\operatorname{Coh}^{\fstr}_{\widehat{\mathcal G}}
  \bigl(\mathscr D^\dagger_{\mathfrak X,\mathbb Q}\bigr)
  \cong
  \Adm^{\Phi,\fstr}_{\mathcal G}(X_s,X_s).
\]
Equivalently, every simple Frobenius filtered strongly \(\widehat{\mathcal G}\)-equivariant
coherent arithmetic \(\mathscr D^\dagger\)-module is already a Frobenius equivariant
overconvergent isocrystal on the whole special fiber.

Fix the base point
\[
  x_0=e\mathcal B_s\in \mathcal G_s/\mathcal B_s.
\]
Its stabilizer is \(\mathcal B_s\).  Intrinsically, the theorem uses the
admissible locus \(\Adm^{\Phi,\fstr}_{\mathcal G}(X_s,X_s)\).  If one also
uses the auxiliary stabilizer description of
Proposition~\ref{prop:stabilizer-description}, under its explicit descent
hypotheses, the corresponding coefficient data are read as irreducible
Frobenius overconvergent coefficient data at \(x_0\) with compatible
\(\mathcal B_s\)-action.  In this optional stabilizer-admissible reading, the transitive case reduces to the homogeneous-space principle
\[
  \mathcal G_s\text{-equivariant coefficients on }\mathcal G_s/\mathcal B_s
  \quad\leftrightarrow\quad
  \mathcal B_s\text{-coefficient data at }x_0.
\]
On the distribution side, the corresponding simple objects are
\[
  \Gamma(\mathfrak X,E),
  \qquad
  [E]\in \Adm^{\Phi,\fstr}_{\mathcal G}(X_s,X_s).
\]
This example has no boundary contribution: the entire classification comes
from coefficients on a single homogeneous orbit.  It is therefore the cleanest
example illustrating the regularity nature of the theorem: no intermediate
extension or Bruhat-cell calculation is involved, and the content is that the
filtered strong equivariance condition places coherent objects in the holonomic
range by the microlocal argument.

\subsection{The rank-one Bruhat stratification: \texorpdfstring{\(\mathrm{SL}_2\)}{SL2} and \texorpdfstring{\(\mathcal K=\mathcal B\)}{K=B}}

Let
\[
  \mathcal G=\mathrm{SL}_{2,\mathcal V},
\]
let \(\mathcal B\subseteq\mathcal G\) be the upper triangular Borel subgroup,
and let
\[
  \mathfrak X=\widehat{\mathbb P^1_{\mathcal V}}.
\]
Then
\[
  X_s\simeq \mathbb P^1_k.
\]
We write homogeneous coordinates as \([X:Y]\), and we let
\[
  \mathcal B_s=
  \left\{
  \begin{pmatrix}
  a & b \\
  0 & a^{-1}
  \end{pmatrix}
  \right\}
\]
act on column vectors. With this convention, the point
\[
  \infty=[1:0]
\]
is fixed by \(\mathcal B_s\), and the complement
\[
  O_{\mathrm{op}}=\mathbb P^1_k\setminus\{\infty\}
  \simeq \mathbb A^1_k
\]
is the open Bruhat cell. Thus the \(\mathcal B_s\)-orbit decomposition is
\[
  \mathbb P^1_k=O_{\mathrm{op}}\sqcup O_{\mathrm{cl}},
  \qquad
  O_{\mathrm{cl}}=\{\infty\}.
\]

\begin{lemma}[Stabilizer admissibility in the rank-one open cell]
\label{lem:SL2-open-cell-admissible}
For the open orbit \(O_{\mathrm{op}}\simeq\mathbb A^1_k\) with base point
\([0:1]\), the orbit map
\[
  q:\mathcal B_s\longrightarrow O_{\mathrm{op}},
  \qquad
  g\longmapsto g\cdot[0:1],
\]
is stabilizer-admissible in the sense used in
\Cref{app:stabilizer-description}.
\end{lemma}

\begin{proof}
Write an element of \(\mathcal B_s\) as
\[
  \begin{pmatrix} a&b\\0&a^{-1}\end{pmatrix}.
\]
On the affine coordinate \(z=X/Y\) of
\(\mathbb P^1_k\setminus\{\infty\}\), the point \([0:1]\) is \(z=0\) and the
orbit map is
\[
  q(a,b)=ab.
\]
The stabilizer of \(0\) is the diagonal torus \(\mathcal T_s\).  The morphism
\[
  \mathcal G_{m,k}\times\mathbb A^1_k\longrightarrow\mathcal B_s,
  \qquad
  (a,z)\longmapsto
  \begin{pmatrix}a&z/a\\0&a^{-1}\end{pmatrix}
\]
identifies \(q\) with the projection
\[
  \mathcal G_{m,k}\times\mathbb A^1_k\longrightarrow\mathbb A^1_k,
  \qquad (a,z)\longmapsto z,
\]
and the right action of \(\mathcal T_s\) with multiplication on the
\(\mathcal G_m\)-factor.  Hence the groupoid
\[
  \mathcal B_s\times\mathcal T_s\rightrightarrows\mathcal B_s
\]
identifies with
\[
  \mathcal B_s\times_{O_{\mathrm{op}}}\mathcal B_s
  \rightrightarrows \mathcal B_s.
\]
This proves the torsor condition.

For effective descent of overconvergent isocrystals on the pair
\((\mathbb A^1_k,\mathbb P^1_k)\), note that the projection above has the
section \(z\mapsto(1,z)\), and both projection and section are morphisms of
pairs
\[
  (\mathcal G_m\times\mathbb A^1,\mathcal G_m\times\mathbb P^1)
  \rightleftarrows
  (\mathbb A^1,\mathbb P^1).
\]
Thus descent along this projection is explicit: an object with descent datum
is recovered by pulling back along the section, and the cocycle condition gives
the unique comparison isomorphism after pulling back to
\(\mathcal G_m\times\mathbb A^1\).  This elementary split-descent argument is
compatible with overconvergence along the boundary \(\{\infty\}\), because the
boundary on the source is exactly \(\mathcal G_m\times\{\infty\}\).  Therefore
the descent condition required in the stabilizer-admissibility definition is
satisfied.
\end{proof}

\begin{remark}[Filtered admissibility in the rank-one examples]
\label{rem:SL2-SE2}
For the rank-one character objects explicitly displayed below, condition
\textup{(SE2)} is checked directly.  On the open cell the preceding proof
identifies the orbit map with the split projection
\(\mathcal G_m\times\mathbb A^1\to\mathbb A^1\).  A character coefficient pulls
back to a rank-one constant object on this split product with integral weight;
in the affine coordinate used above, the fundamental vector fields preserve the
standard coherent generator, so the practical criterion of
\Cref{prop:practical-SE2-criterion} applies.  The closed orbit has no boundary,
so the same assertion is immediate there.  Thus the rank-one character objects
appearing below define classes in the corresponding admissible loci
\(\Adm^{\Phi,\fstr}_{\mathcal B}\).  We do not use this remark as an automaticity
theorem for arbitrary coefficient objects.
\end{remark}

Theorem~\ref{thm:orbit-classification} gives
\[
  \operatorname{Irr}
  F\text{-}\operatorname{Coh}^{\fstr}_{\widehat{\mathcal B}}
  \bigl(\mathscr D^\dagger_{\widehat{\mathbb P^1},\mathbb Q}\bigr)
  \cong
  \Adm^{\Phi,\fstr}_{\mathcal B}(\mathbb A^1_k,\mathbb P^1_k)
  \sqcup
  \Adm^{\Phi,\fstr}_{\mathcal B}(\{\infty\},\{\infty\}).
\]
Thus every simple Frobenius filtered strongly \(\widehat{\mathcal B}\)-equivariant coherent
arithmetic \(\mathscr D^\dagger\)-module on \(\widehat{\mathbb P^1}\) belongs
to one of two geometric families: it is either the intermediate extension of
an equivariant overconvergent isocrystal from the open Bruhat cell, or it is
supported on the closed fixed point.

\begin{example}[Open Bruhat cell]
\label{ex:SL2-open-cell}
Let
\[
  j:\mathbb A^1_k\hookrightarrow \mathbb P^1_k
\]
be the open immersion. The point \([0:1]\in\mathbb A^1_k\) has stabilizer the
diagonal torus
\[
  \mathcal T_s=
  \left\{
  \begin{pmatrix}
  a & 0 \\
  0 & a^{-1}
  \end{pmatrix}
  \right\}
  \subseteq \mathcal B_s.
\]
After choosing this base point, the Frobenius equivariant
overconvergent isocrystals on the open cell used in the theorem are controlled,
for the split stabilizer-generated subcategory of \Cref{rem:SL2-SE2}, by
coefficient data with compatible \(\mathcal T_s\)-action.

The constant overconvergent isocrystal
\[
  \mathcal O^\dagger_{\mathbb A^1_k}
\]
with trivial stabilizer action gives the simple object
\[
  j_{!+}\bigl(\mathcal O^\dagger_{\mathbb A^1_k}\bigr).
\]
More generally, every irreducible Frobenius equivariant coefficient
datum on the open orbit whose class lies in
\(\Adm^{\Phi,\fstr}_{\mathcal B}(\mathbb A^1_k,\mathbb P^1_k)\) gives a simple arithmetic
\(\mathscr D^\dagger\)-module by the same middle-extension construction. For instance, if
\[
  \chi:\mathcal T_s\longrightarrow \mathbb G_{m,k}
\]
is a character and \(V_\chi\) is its one-dimensional coefficient space, then
\Cref{prop:stabilizer-description} gives the corresponding equivariant
overconvergent isocrystal on the open cell.  For example, for any integer
\(n\), the character
\[
  \chi_n:
  \begin{pmatrix}a&0\\0&a^{-1}\end{pmatrix}
  \longmapsto a^n
\]
produces a one-dimensional coefficient object \(E_n\); the associated simple
object is \(j_{!+}(E_n)\). More precisely, the notation below
is shorthand for the object obtained by applying the stabilizer equivalence of
\Cref{prop:stabilizer-description} to the coefficient datum \(V_\chi\). In the
notation
\[
  E_\chi\simeq \mathcal B_s\times^{\mathcal T_s}V_\chi
\]
we are not forming an associated vector bundle in the classical sense; rather,
this denotes the object of
\(\Isoc^{\dagger}_{\mathcal B}(\mathbb A^1_k,\mathbb P^1_k)\) obtained from
the one-dimensional stabilizer coefficient datum \(V_\chi\) through the
stabilizer equivalence of \Cref{prop:stabilizer-description}. The corresponding
simple object is
\[
  j_{!+}(E_\chi).
\]
If \(\chi\neq1\), then \(E_\chi\) is not isomorphic to the trivial object
in the Frobenius equivariant isocrystal category: its fiber at the base point carries the
nontrivial stabilizer character \(\chi\).  In this split rank-one example the
underlying non-equivariant connection may be trivial, but the equivariant
coefficient object, and hence the simple Frobenius filtered equivariant arithmetic
\(\mathscr D^\dagger\)-module \(j_{!+}(E_\chi)\), is genuinely different
from the trivial-coefficient intermediate extension.  The trivial character
recovers the constant object, while nontrivial characters produce nontrivial
equivariant coefficient objects on the open Bruhat cell.  By
\Cref{rem:SL2-SE2}, each character object \(E_\chi\), and in particular each
\(E_n\), defines a class
\[
  [E_\chi]\in
  \Adm^{\Phi,\fstr}_{\mathcal B}(\mathbb A^1_k,\mathbb P^1_k).
\]
Thus these are concrete nontrivial points of the admissible locus appearing in
the classification theorem.
\end{example}

\begin{example}[Closed Bruhat cell]
\label{ex:SL2-closed-orbit}
Let
\[
  i:\{\infty\}\hookrightarrow \mathbb P^1_k
\]
be the closed immersion. Since the orbit is a point, an overconvergent
isocrystal on \((\{\infty\},\{\infty\})\) is finite-dimensional coefficient
data at \(\infty\). The stabilizer of \(\infty\) is all of \(\mathcal B_s\).
Consequently, the Frobenius equivariant coefficient objects used in the
classification are irreducible coefficient data with compatible
\(\mathcal B_s\)-action, viewed inside
\[
  \Adm^{\Phi,\fstr}_{\mathcal B}(\{\infty\},\{\infty\}).
\]
The corresponding simple arithmetic \(\mathscr D^\dagger\)-modules are
\[
  i_+(V)=i_{!+}(V),
  \qquad
  [V]\in
  \Adm^{\Phi,\fstr}_{\mathcal B}(\{\infty\},\{\infty\}).
\]
Indeed, for a closed immersion the extension by zero and the direct image
coincide on objects supported on the closed stratum; there is no boundary
across which to take a nontrivial minimal extension. Thus the intermediate
extension agrees with the closed direct image:
\[
  i_{!+}(V)=i_+(V).
\]
These are skyscraper-type arithmetic \(\mathscr D^\dagger\)-modules at the
closed orbit, equipped with their induced \(\widehat{\mathcal B}\)-equivariant
structures.
\end{example}

\subsection{A torus action on \texorpdfstring{\(\widehat{\mathbb P^1}\)}{P1}}

We keep \(\mathcal G=\mathrm{SL}_{2,\mathcal V}\) and
\(\mathfrak X=\widehat{\mathbb P^1_{\mathcal V}}\), but now take
\[
  \mathcal K=\mathcal T
\]
to be the diagonal torus. On the affine coordinate
\[
  z=X/Y
\]
of \(\mathbb A^1_k=\mathbb P^1_k\setminus\{\infty\}\), an element
\[
  \begin{pmatrix}
  a & 0 \\
  0 & a^{-1}
  \end{pmatrix}
\]
acts by
\[
  z\longmapsto a^2 z.
\]
If the square character has nonzero differential, for instance if
\(p\neq2\), the special fiber decomposes into three separable
\(\mathcal T_s\)-orbits:
\[
  \mathbb P^1_k=\{0\}\sqcup \mathbb G_{m,k}\sqcup \{\infty\}.
\]
Under this separability assumption the finite-orbit and orbit-separability hypotheses are satisfied, and the orbit classification
becomes
\[
\begin{aligned}
  \operatorname{Irr}
  F\text{-}\operatorname{Coh}^{\fstr}_{\widehat{\mathcal T}}
  \bigl(\mathscr D^\dagger_{\widehat{\mathbb P^1},\mathbb Q}\bigr)
  \cong{}&
  \Adm^{\Phi,\fstr}_{\mathcal T}(\{0\},\{0\}) \\
  &\sqcup
  \Adm^{\Phi,\fstr}_{\mathcal T}(\mathbb G_{m,k},\mathbb P^1_k) \\
  &\sqcup
  \Adm^{\Phi,\fstr}_{\mathcal T}(\{\infty\},\{\infty\}).
\end{aligned}
\]
The three families of simple objects are therefore
\[
  i_{0,+}(V_0),
  \qquad
  j_{!+}(E),
  \qquad
  i_{\infty,+}(V_\infty),
\]
where \(V_0\) and \(V_\infty\) are irreducible admissible Frobenius equivariant
coefficient data at the two fixed points and \(E\) is an irreducible Frobenius
\(\mathcal T\)-equivariant overconvergent isocrystal on
\((\mathbb G_{m,k},\mathbb P^1_k)\) whose class lies in
\(\Adm^{\Phi,\fstr}_{\mathcal T}(\mathbb G_{m,k},\mathbb P^1_k)\).

This example shows that the theorem is not tied only to the Bruhat
stratification; the actual coefficient parameter set is the admissible locus
inside the Frobenius equivariant isocrystals on each orbit. It applies to any
smooth subgroup whose special fiber has a finite orbit decomposition on the
flag variety.  Thus the examples should be read as tests of the orbitwise
admissible-locus formalism rather than as a replacement for the explicit
Bruhat-cell constructions already available in the literature.

\subsection{Distribution-side interpretation}

The examples above have immediate algebraic counterparts. Let
\[
  A_0=\Gamma(\mathfrak X,\mathscr D^\dagger_{\mathfrak X,\mathbb Q}).
\]
For any smooth subgroup \(\mathcal K\) satisfying the finite-orbit and orbit-separability hypotheses,
Theorem~\ref{thm:distribution-classification} identifies the simple objects of
\[
  F\text{-}A_0\operatorname{-mod}^{\fstr}_{\widehat{\mathcal K},\mathrm{fp}}
\]
with the same Frobenius orbitwise coefficient data satisfying the admissibility condition. The simple module corresponding to
\[
  (O,E)
\]
is
\[
  \Gamma\bigl(\mathfrak X,j_{O,!+}(E)\bigr).
\]
The geometry of the special-fiber orbit decomposition therefore organizes
simple modules over the crystalline distribution algebra at trivial character.

For instance, in the rank-one Bruhat example with \(\mathcal K=\mathcal B\),
the simple Frobenius filtered strongly \(\widehat{\mathcal B}\)-equivariant finitely presented
\(A_0\)-modules split into the two families
\[
  \Gamma\bigl(\widehat{\mathbb P^1_{\mathcal V}},
  j_{!+}(E)\bigr),
  \qquad
  [E]\in
  \Adm^{\Phi,\fstr}_{\mathcal B}(\mathbb A^1_k,\mathbb P^1_k),
\]
and
\[
  \Gamma\bigl(\widehat{\mathbb P^1_{\mathcal V}},
  i_+(V)\bigr),
  \qquad
  [V]\in
  \Adm^{\Phi,\fstr}_{\mathcal B}(\{\infty\},\{\infty\}).
\]
In the torus example, the algebraic category is similarly divided into three
geometric families corresponding to the two fixed points and the open torus
orbit.

\begin{remark}[No explicit distribution-module calculation in the examples]
\label{rem:no-explicit-global-section-calculation}
The examples are meant to identify the orbitwise coefficient data that label
simple objects and their images under global sections.  We do not compute the
modules
\[
  \Gamma\bigl(\widehat{\mathbb P^1_{\mathcal V}},j_{!+}(E_n)\bigr)
\]
explicitly inside the crystalline distribution algebra.  Such a computation
would require a separate analysis of the resulting global-section modules and
is not used in the proof of the classification theorem.
\end{remark}

\subsection{Comparison with the classical orbit classification}

The following table summarizes the dictionary between the classical complex
picture and the arithmetic picture developed in this article. The third
column records the role played by each object in the classification.
\begin{center}
\footnotesize
\renewcommand{\arraystretch}{1.18}
\begin{tabularx}{\textwidth}{@{}>{\raggedright\arraybackslash}p{0.27\textwidth}
                        >{\raggedright\arraybackslash}p{0.34\textwidth}
                        >{\raggedright\arraybackslash}X@{}}
\toprule
\textbf{Classical complex setting} &
\textbf{Arithmetic setting} &
\textbf{Role in the theorem} \\
\midrule
\(D_{G/B}\) &
\(\mathscr D^\dagger_{\widehat{\mathcal G/\mathcal B},\mathbb Q}\) &
Sheaf of differential operators whose coherent equivariant modules are studied. \\
\addlinespace[0.35em]
\(K\)-orbits on \(G/B\) &
\(\mathcal K_s\)-orbits on \(X_s\) &
Geometric strata controlling supports and conormal directions. \\
\addlinespace[0.35em]
\(K\)-equivariant local systems &
\(\mathcal K\)-equivariant overconvergent isocrystals &
Coefficient objects attached to each orbit. \\
\addlinespace[0.35em]
\(j_{!*}\) &
\(j_{!+}\) &
Minimal extension from an orbit to its closure. \\
\addlinespace[0.35em]
regularity class of equivariant \(D\)-modules &
geometrically overholonomic arithmetic modules &
Arithmetic finiteness and six-functor stability class obtained in the Frobenius filtered strongly equivariant range. \\
\addlinespace[0.35em]
\(U(\mathfrak g)\)-modules with central character &
\(A_0\)-modules at trivial character &
Global-section side of the orbit classification. \\
\bottomrule
\end{tabularx}
\end{center}
Here ``geometrically overholonomic'' is not meant as a literal synonym for
classical regular holonomicity.  It is the arithmetic stability condition used
by the six-functor formalism and is stronger than the holonomicity obtained
from the characteristic-variety estimate alone.  The theorem proved in this
article is therefore not merely a formal translation of terminology. The arithmetic proof requires Berthelot's theory
of \(\mathscr D^\dagger\)-modules, the characteristic-variety formalism at
finite level, the arithmetic intermediate extension, and the equivariant
orbital devissage proved above. The examples in this section show how the
abstract classification becomes a concrete orbit-by-orbit description of
simple arithmetic \(\mathscr D^\dagger\)-modules and of their global sections,
while keeping separate the objects already constructed in the Bruhat-cell
setting by Huyghe--Schmidt from the equivariant admissibility information added
here.

\section{Further directions}
\label{sec:further-directions}

The results proved in this article suggest several continuations.  We describe
some of them in a form that keeps the technical constraints of the arithmetic
setting visible.  The point is not only to enlarge the class of examples, but
to understand which parts of the classical representation-theoretic picture can
be transported to Berthelot's arithmetic theory without losing the finiteness
properties encoded by overholonomicity.

A preliminary problem is to clarify the exact size of the filtered strongly
equivariant category.  The present paper treats filtered finite-level
realizability as an explicit hypothesis.  By \Cref{prop:practical-SE2-criterion},
a positive automaticity theorem would follow from the existence, after raising
the level locally, of coherent generators stable under the finitely many
fundamental vector fields.  The local rank-one calculation of
\Cref{prop:rank-one-integral-exponent-SE2} shows that, once a geometric
intermediate-extension model with integral logarithmic exponents is already in
hand, filtered realizability may be checked by an explicit generator.  The
remaining open question is the abstract finite-level coherence problem: whether
ordinary infinitesimal strong equivariance alone produces such stable coherent
generators.  A positive answer would upgrade the main theorem to the unfiltered
strongly equivariant category.  We leave this problem open.

A first direction is the extension of the present orbit classification to
twisted arithmetic differential operators.  The article has deliberately worked
at trivial infinitesimal character, where the sheaf
\(\mathscr D^\dagger_{\mathfrak X,\mathbb Q}\) and its global section algebra
\(A_0\) give the cleanest form of arithmetic localization.  For a twisting
parameter \(\lambda\), a parallel theory would have to be developed for
\[
  \mathscr D^\dagger_{\mathfrak X,\lambda},
  \qquad
  A_\lambda=
  \Gamma\bigl(\mathfrak X,
  \mathscr D^\dagger_{\mathfrak X,\lambda}\bigr),
\]
provided \(\lambda\) lies in the range where arithmetic localization is an
equivalence.  The geometric part of such an extension would require a separate proof: one
would have to establish the corresponding moment-map containment for twisted
arithmetic differential operators and then use finite orbit type plus orbit separability to obtain
conormal containment.  The genuinely new issue is the compatibility
between the twisting datum and the infinitesimal \(\mathfrak K\)-action.  In
particular, one must formulate a twisted version of strong equivariance in
which the Lie algebra action on the module agrees with the action by twisted
arithmetic differential operators.  A satisfactory twisted Frobenius theory would require a separately proved
localization and overholonomicity statement before one could assert a bijection
of the form
\[
  \coprod_{O\in\mathcal K_s\backslash X_s}
  \operatorname{Irr}
  F\text{-}\Isoc^{\dagger}_{\mathcal K,\lambda}(O,\overline O)
  \;\cong\;
  \operatorname{Irr}
  F\text{-}\Coh^{\fstr}_{\mathfrak K}
  \bigl(\mathscr D^\dagger_{\mathfrak X,\lambda}\bigr),
\]
where the coefficient category on each orbit records the restriction of the
arithmetic twist to the pair \((O,\overline O)\).  This would be the arithmetic
counterpart of the usual passage from untwisted equivariant \(D\)-modules to
localization at arbitrary infinitesimal character.

A second direction concerns the representation-theoretic meaning of the
classification on the crystalline distribution side.  The equivalence used in
\Cref{sec:distributions} transports a simple Frobenius filtered strongly equivariant arithmetic
\(\mathscr D^\dagger\)-module to a simple module over
\[
  A_0=
  \Gamma\bigl(\mathfrak X,
  \mathscr D^\dagger_{\mathfrak X,\mathbb Q}\bigr).
\]
It is natural to ask how these modules sit inside categories of locally
analytic representations of \(p\)-adic groups.  In the classical complex
picture, the orbit classification of equivariant \(D\)-modules is tied to
Harish--Chandra modules, standard modules, and their irreducible quotients.
An arithmetic analogue would require a comparison between
\(A_0\)-modules arising from
\(\Gamma(\mathfrak X,j_{O,!+}(E))\) and representation-theoretic objects
constructed from locally analytic distribution algebras.  One possible route is
to study whether the orbit filtration of a strongly equivariant
\(\mathscr D^\dagger\)-module induces a geometric filtration on the associated
crystalline distribution module, whose graded pieces are controlled by
stabilizer data on the special-fiber orbits:
\[
  \operatorname{gr}\,
  \Gamma(\mathfrak X,\mathcal M)
  \quad\text{built from}\quad
  \Gamma\bigl(\mathfrak X,j_{O,!+}(E)\bigr).
\]
Such a comparison would clarify which arithmetic sheaves correspond to
geometric avatars of locally analytic principal series, and which new simple
objects appear because the coefficients are overconvergent isocrystals rather
than ordinary local systems.  In rank one, the concrete objects
\[
  \Gamma\bigl(\widehat{\mathbb P^1_{\mathcal V}},j_{!+}(E_n)\bigr)
\]
attached to the characters \(\chi_n:\operatorname{diag}(a,a^{-1})\mapsto a^n\)
are natural first test cases.  One may ask whether, after a suitable comparison
between Berthelot's crystalline distribution algebra and Schneider--Teitelbaum
style locally analytic distribution algebras, these modules are related to
arithmetic or crystalline avatars of locally analytic principal series.  We do
not assert such an identification here; it would require a separate comparison
theorem between the crystalline and locally analytic distribution settings.

A third direction is the development of arithmetic Harish--Chandra sheaves for
symmetric pairs.  Suppose that \(\mathcal K\) is the fixed subgroup of an
involution of \(\mathcal G\), or more generally that the pair
\((\mathcal G,\mathcal K)\) has finitely many orbits on the flag variety.  The
present theorem applies to such a situation whenever \(\mathcal K\) is smooth
and the special-fiber orbit condition holds.  The next step is to study the
geometry of these orbit closures, their conormal varieties, and the resulting
arithmetic standard and costandard objects.  In analogy with the classical
Beilinson--Bernstein and Kashiwara--Schmid picture, a possible arithmetic extension would study
Harish--Chandra sheaves to be organized by orbit closures and equivariant
coefficient data:
\[
  (O,E)
  \quad\longmapsto\quad
  j_{O,!}(E),\quad j_{O,+}(E),\quad j_{O,!+}(E).
\]
The simple objects are given by the middle extensions in this article, but the
standard and costandard extensions may contain additional information about
multiplicities, extension groups, and possible arithmetic analogues of
Kazhdan--Lusztig phenomena.

A fourth direction is the study of characteristic cycles and multiplicities in
the arithmetic setting.  The proof of regularity uses the inclusion
\[
  \Car(\mathcal M)
  \subseteq
  \bigcup_{O\in\mathcal K_s\backslash X_s}T^*_O X_s,
\]
but it does not attempt to compute the multiplicity with which each conormal
component occurs.  For a geometrically overholonomic filtered strongly equivariant module, one may study characteristic cycles of the form
\[
  \operatorname{CC}(\mathcal M)
  =
  \sum_{O\in\mathcal K_s\backslash X_s}
  m_O(\mathcal M)\,[T^*_O X_s],
  \qquad
  m_O(\mathcal M)\in\mathbb Z_{\ge0}.
\]
Understanding the integers \(m_O(\mathcal M)\) for
\(\mathcal M=j_{O,!+}(E)\) would refine the orbit classification from a
classification of simple objects to a microlocal invariant theory.  It would
also open the door to arithmetic analogues of characteristic-cycle formulas
for intersection cohomology \(D\)-modules and perverse sheaves.

A final direction suggested by the equivariant devissage is a broader stratified
formalism.  The flag variety is special because arithmetic localization is
available and the relevant orbit stratifications are representation-theoretic.
Nevertheless, the argument isolates a mechanism that may apply in other smooth
proper formal schemes with finitely many orbits satisfying orbit separability under a smooth group action:
\[
\begin{aligned}
  &\text{filtered strong equivariance}
  \Longrightarrow \text{moment-zero containment}
  \Longrightarrow \text{holonomicity},\\
  &\text{Frobenius overholonomicity}
  \Longrightarrow \text{middle-extension extraction}
  \Longrightarrow \text{orbital devissage}.
\end{aligned}
\]
The extent to which this mechanism survives outside the flag variety depends
on the availability of localization, intermediate extension, and a good theory
of equivariant overconvergent coefficients on the strata.  Thus the present
article provides not only an orbit classification theorem for formal
flag varieties, but also as a model for a more general theory of filtered strongly
equivariant arithmetic \(\mathscr D^\dagger\)-modules on stratified formal
schemes.

\appendix
\section{Auxiliary stabilizer description}
\label{app:stabilizer-description}
\label{subsec:stabilizer-description}

This appendix records a stabilizer description of the coefficient category on a single
orbit.  This result is not used in the proof of regularity or in the proof of
the main orbit classification, but it makes the classification more concrete
when the required descent hypotheses are available.  Because overconvergent isocrystals are
objects on pairs, the statement must include the boundary in the descent datum.
For this reason, the stabilizer category below is not the category of abstract
representations of the stabilizer.  It is the category of stabilizer coefficient
data whose descent is effective in the overconvergent category of the pair
\((O,\overline O)\).

Let \(x\in O(\overline{k})\) be a geometric point and let
\[
  H_x:=\operatorname{Stab}_{\mathcal K_{s,\overline k}}(x)
\]
be its stabilizer.  We shall use the orbit morphism
\[
  q_x:
  \mathcal K_{s,\overline{k}}
  \longrightarrow
  O_{\overline{k}},
  \qquad
  g\longmapsto g\cdot x.
\]
The natural right action of \(H_x\) on \(\mathcal K_{s,\overline k}\) is given
by multiplication, and the fibers of \(q_x\) are the corresponding
\(H_x\)-orbits.  In the descent formalism used below, we require this orbit map
to be an effective descent cover for overconvergent isocrystals on the relevant
pair.  This requirement is made explicit in
Definition~\ref{def:stabilizer-admissible-orbit}; no automatic descent property
is used without being included in that hypothesis.

To speak about overconvergent isocrystals one must realize \(q_x\) as a
morphism of pairs
\[
  q_x^\natural:
  (U_x,\overline U_x)
  \longrightarrow
  (O_{\overline{k}},\overline O_{\overline{k}}),
\]
whose open part is the orbit map
\(q_x:\mathcal K_{s,\overline k}\to O_{\overline k}\).  We require this
realization to be compatible with the left \(\mathcal K_{s,\overline k}\)-action
and the right \(H_x\)-action used in the torsor groupoid.  The choice of
\(\overline U_x\) and these compatibilities are part of the
stabilizer-admissibility hypothesis.  This is
important: the boundary \(\overline O_{\overline k}\setminus O_{\overline k}\)
must be visible after pullback.  In the rank-one example of
\Cref{lem:SL2-open-cell-admissible}, for instance, the split realization is
\((\mathcal G_m\times\mathbb A^1,\mathcal G_m\times\mathbb P^1)\to
(\mathbb A^1,\mathbb P^1)\).  The effectiveness condition in the following
definition is exactly the point at which overconvergence along the boundary
enters.

\begin{definition}[Stabilizer-admissible orbit]
\label{def:stabilizer-admissible-orbit}
The orbit \(O\) is called \emph{stabilizer-admissible at \(x\)} if the
following two conditions hold.

First, the canonical morphism
\[
  \mathcal K_{s,\overline k}\times H_x
  \longrightarrow
  \mathcal K_{s,\overline k}\times_{O_{\overline k}}
  \mathcal K_{s,\overline k},
  \qquad
  (g,h)\longmapsto(g,gh),
\]
is an isomorphism of groupoids over \(\mathcal K_{s,\overline k}\).  Equivalently,
\(q_x\) is a right \(H_x\)-torsor in the topology used for descent.

Second, the morphism of pairs \(q_x^\natural\) is an effective descent morphism
for overconvergent isocrystals.  Thus pullback along \(q_x^\natural\) induces an
equivalence between
\[
  \Isoc^\dagger(O_{\overline k},\overline O_{\overline k})
\]
and the category of overconvergent isocrystals on
\[
  (U_x,\overline U_x)
\]
equipped with descent data for the groupoid
\[
  \mathcal K_{s,\overline k}\times_{O_{\overline k}}
  \mathcal K_{s,\overline k}
  \rightrightarrows
  \mathcal K_{s,\overline k}.
\]
We also require the corresponding Galois descent from \(\overline k\) to \(k\)
to be effective for the objects under consideration.  These are the
effective-descent properties for overconvergent isocrystals on pairs used in
this theorem; see \cite[Sections 7--8]{LeStumOverconvergentSite},
\cite[Theorem 5.1]{LazdaDescent}, and
\cite[\S2]{KedlayaNotesIsocrystals}.
\end{definition}

\begin{remark}
\label{rem:admissibility-is-hypothesis}
Stabilizer-admissibility is an explicit hypothesis in the stabilizer theorem
below.  The main classification theorem is stated intrinsically in terms of
\(\Isoc^\dagger_{\mathcal K}(O,\overline O)\), and no later argument depends on
every orbit being stabilizer-admissible.  When the effective descent statement
for the chosen pair is available, the following result gives the corresponding
stabilizer interpretation.
\end{remark}

Under the torsor identification in
Definition~\ref{def:stabilizer-admissible-orbit}, the descent groupoid is
\[
  \mathcal K_{s,\overline{k}}\times H_x
  \rightrightarrows
  \mathcal K_{s,\overline{k}},
\]
with structure maps
\[
  r_1(g,h)=g h,
  \qquad
  r_2(g,h)=g.
\]

\begin{definition}[Boundary-admissible stabilizer coefficient data]
\label{def:stabilizer-coefficients}
Assume that \(O\) is stabilizer-admissible at \(x\).  We denote by
\(\Coeff^\dagger(H_x)\) the following category.  An object is a pair
\((V,\theta_V)\), where \(V\) is a finite-dimensional coefficient space over
the coefficient field and
\[
  \mathcal E_V:=
  \mathcal O^\dagger_{U_x}\otimes V
\]
is the constant overconvergent isocrystal on the chosen source pair
\[
  (U_x,\overline U_x),
\]
together with an \(H_x\)-descent datum
\[
  \theta_V:
  r_2^*\mathcal E_V
  \xrightarrow{\;\sim\;}
  r_1^*\mathcal E_V
\]
on \(\mathcal K_{s,\overline k}\times H_x\), satisfying the identity and
cocycle conditions.  The descent datum is required to be effective for the pair
\[
  (O_{\overline k},\overline O_{\overline k}),
\]
that is, the descended object must belong to
\[
  \Isoc^\dagger(O_{\overline k},\overline O_{\overline k}).
\]
If the geometric point \(x\) is not defined over \(k\), we further include the
compatible Galois descent datum from \(\overline k\) to \(k\).  Morphisms
\[
  (V,\theta_V)\longrightarrow(W,\theta_W)
\]
are linear maps \(V\to W\) whose induced morphisms
\[
  \mathcal E_V\longrightarrow \mathcal E_W
\]
commute with the \(H_x\)-descent data and with the Galois descent data.
\end{definition}

\begin{remark}
\label{rem:associated-bundle-notation}
When \((V,\theta_V)\in\Coeff^\dagger(H_x)\), we sometimes write
\[
  \mathcal K_{s,\overline k}\times^{H_x}V
\]
for the associated object on \(O_{\overline k}\).  This notation is only an
abbreviation for the object obtained by effective descent from
\(\mathcal E_V\) and \(\theta_V\) in the overconvergent category.  It is not an
ordinary associated vector bundle construction.  The overconvergence along
\(\overline O_{\overline k}\setminus O_{\overline k}\) is part of the descended
object on the pair \((O_{\overline k},\overline O_{\overline k})\).
\end{remark}

\begin{proposition}[Stabilizer description]
\label{prop:stabilizer-description}
Assume that \(O\) is geometrically homogeneous under \(\mathcal K_s\) and
stabilizer-admissible at \(x\) in the sense of
Definition~\ref{def:stabilizer-admissible-orbit}.  Then there is an equivalence
of categories
\begin{equation}
\label{eq:stabilizer-equivalence}
  \Isoc^{\dagger}_{\mathcal K}(O,\overline O)
  \xrightarrow{\;\sim\;}
  \Coeff^{\dagger}(H_x).
\end{equation}
Under this equivalence, irreducible equivariant overconvergent isocrystals
correspond exactly to irreducible boundary-admissible stabilizer coefficient
data.
\end{proposition}

\begin{proof}
We first work over \(\overline k\).  By stabilizer-admissibility, pullback along
\(q_x^\natural\) identifies overconvergent isocrystals on
\((O_{\overline k},\overline O_{\overline k})\) with overconvergent isocrystals
on the chosen source pair \((U_x,\overline U_x)\) equipped with descent data for
\[
  \mathcal K_{s,\overline k}\times_{O_{\overline k}}
  \mathcal K_{s,\overline k}
\rightrightarrows \mathcal K_{s,\overline k}.
\]
The torsor condition identifies this groupoid with
\[
  \mathcal K_{s,\overline k}\times H_x
\rightrightarrows
  \mathcal K_{s,\overline k},
  \qquad
  r_1(g,h)=gh,
  \quad r_2(g,h)=g.
\]
Thus the only point to check is how the \(\mathcal K_{s,\overline k}\)-equivariant
structure is read after this pullback.

Let \((E,\beta_E)\) be a \(\mathcal K_{s,\overline k}\)-equivariant
overconvergent isocrystal on the pair.  Pulling back along \(q_x^\natural\)
gives an object
\[
  F:=q_x^{\natural,*}E\in \Isoc^\dagger(U_x,\overline U_x)
\]
with a left \(\mathcal K_{s,\overline k}\)-equivariant structure.  Since the
open part of \(U_x\) is \(\mathcal K_{s,\overline k}\) and the pair-level
realization is assumed compatible with left translations, this left-equivariant
structure canonically trivializes \(F\) from its fiber at the identity:
\[
  F\simeq \mathcal O^\dagger_{U_x}\otimes V,
  \qquad V:=e^*F.
\]
The descent datum for the torsor, transported through
\(\mathcal K_{s,\overline k}\times H_x\simeq
\mathcal K_{s,\overline k}\times_{O_{\overline k}}\mathcal K_{s,\overline k}\),
becomes precisely an \(H_x\)-descent datum
\[
  \theta_V:r_2^*(\mathcal O^\dagger_{U_x}\otimes V)
  \xrightarrow{\sim}
  r_1^*(\mathcal O^\dagger_{U_x}\otimes V).
\]
This defines the functor
\[
  \Phi:
  \Isoc^\dagger_{\mathcal K_{s,\overline k}}(O_{\overline k},\overline O_{\overline k})
  \longrightarrow \Coeff^\dagger(H_x).
\]

Conversely, an object \((V,\theta_V)\in\Coeff^\dagger(H_x)\) is, by definition,
a constant overconvergent isocrystal \(\mathcal O^\dagger_{U_x}\otimes V\) on
the chosen source pair with an effective descent datum for \(q_x^\natural\).
Effective descent gives an overconvergent isocrystal \(E_V\) on
\((O_{\overline k},\overline O_{\overline k})\).  The tautological left
\(\mathcal K_{s,\overline k}\)-equivariant structure on the constant object
commutes with the right \(H_x\)-descent datum, because left and right
translations commute on the open torsor; by the compatibility built into the
pair-level realization, it descends to a \(\mathcal K_{s,\overline k}\)-equivariant
structure on \(E_V\).  This gives the functor
\[
  \Psi:\Coeff^\dagger(H_x)
  \longrightarrow
  \Isoc^\dagger_{\mathcal K_{s,\overline k}}(O_{\overline k},\overline O_{\overline k}).
\]

The two functors are quasi-inverse because, after pullback along
\(q_x^\natural\), both composites identify the same constant object together
with the same torsor descent datum.  Full faithfulness and essential surjectivity
are exactly the effective descent condition in
Definition~\ref{def:stabilizer-admissible-orbit}.  Finally, the passage from
\(\overline k\) to \(k\) is part of the same stabilizer-admissibility hypothesis:
we require the relevant Galois descent data to be effective.  Exactness of
pullback, descent, and the fiber functor then identifies irreducible objects on
both sides.
\end{proof}

\begin{remark}
\label{rem:stabilizer-description-use}
The equivalence \eqref{eq:stabilizer-equivalence} is the precise arithmetic
analogue of the classical equivalence between equivariant local systems on a
homogeneous space and representations of the stabilizer.  The boundary
admissibility in Definition~\ref{def:stabilizer-coefficients} is the arithmetic
feature: it records the overconvergence condition along
\(\overline O\setminus O\).  The abstract orbit classification below does not
require choosing stabilizer data, but when the orbit is stabilizer-admissible
the right-hand side of the classification can be read orbit by orbit through
\(\Coeff^\dagger(H_x)\).
\end{remark}

\section{Logical structure of the proof}
\label{app:logical-structure}

This appendix records the logical structure of the proof in a textual form.
It is included to make the architecture of the argument transparent without
introducing additional notation.  The main point is that the proof separates
three mechanisms which are often intertwined in the classical theory: a
microlocal mechanism, a stratified devissage mechanism, and a localization
mechanism.

The first mechanism is microlocal.  Filtered strong equivariance gives two
infinitesimal actions of \(\operatorname{Lie}(\mathcal K)\otimes_{\mathcal V}K\)
on a coherent arithmetic \(\mathscr D^\dagger\)-module \(\mathcal M\): one
obtained by differentiating the equivariant structure, and one obtained by the
fundamental vector fields on \(\mathfrak X\).  The defining conditions
\textup{(SE1)--(SE2)} require these two actions to agree and require this
agreement to be realized on a good finite-level model with a good filtration
stable under the differentiated infinitesimal action.  From these defining
conditions, \Cref{lem:symbols-fundamental-fields} deduces that the principal
symbols of all fundamental vector fields annihilate
\(\operatorname{gr}\mathcal M\).  Equivalently,
\[
  \Car(\mathcal M)
  \subseteq
  \mu_{\mathcal K}^{-1}(0),
\]
where \(\mu_{\mathcal K}:T^*X_s\to\mathfrak k_s^*\) is the moment map of the
special-fiber action.  As in \Cref{lem:moment-zero-conormals}, at the level of reduced underlying
closed subsets one has
\[
  \mu_{\mathcal K}^{-1}(0)
  =
  \bigcup_{O\in\mathcal K_s\backslash X_s}T^*_O X_s.
\]
The finite-orbit and orbit-separability hypotheses therefore turn the moment-zero containment into a
finite conormal containment.  Bernstein's inequality then gives holonomicity.

The second mechanism is stratified.  Once holonomicity is known, one isolates the maximal coherent submodule
supported on a closed \(\mathcal K_s\)-stable subset \(Z\subseteq X_s\):
\[
  \Gamma_Z(\mathcal M)\subseteq \mathcal M.
\]
This is a noetherian construction in the holonomic coherent category, and it has
the universal property specified in \Cref{app:t-structure}.  If \(O\) is
an open orbit in \(\Supp(\mathcal M)\) and
\(Z=\Supp(\mathcal M)\setminus O\), then
\[
  Q=\mathcal M/\Gamma_Z(\mathcal M)
\]
has the same restriction to \(O\) as \(\mathcal M\) and no nonzero subobject
supported on the boundary.  The conormal containment restricts to the zero
section over \(O\); hence \(Q|_O\), and therefore \(\mathcal M|_O\), is a
Frobenius equivariant overconvergent isocrystal.  The middle-extension
comparison extracts from \(Q\) a canonical open-orbit piece
\(j_{O,!+}(Q|_O)\), and the boundary-supported remainder is handled by induction
on the number of orbits.  After the Frobenius regularity theorem has placed the
object in the overholonomic heart, this proves the equivariant orbital
devissage: every Frobenius filtered strongly equivariant coherent module admits
a finite filtration whose graded pieces are of the form
\[
  j_{O,!+}(E),
  \qquad
  [E]\in
  \Adm^{\Phi,\fstr}_{\mathcal K}(O,\overline O),
\]
where the coefficient objects \(E\) need not be simple.  Thus the devissage is
a structural filtration mechanism, not the source of the Frobenius
holonomic-to-overholonomic upgrade.

The third mechanism is classificatory.  If \(\mathcal M\) is simple, its
support is the closure of a unique open \(\mathcal K_s\)-orbit \(O\).  The
restriction \(\mathcal M|_O\) is then irreducible in the Frobenius equivariant
isocrystal category, and the minimality of intermediate extension gives
\[
  \mathcal M
  \simeq
  j_{O,!+}(\mathcal M|_O).
\]
In the opposite direction, the intermediate extension of an irreducible
Frobenius equivariant overconvergent isocrystal whose class lies in the admissible
locus \(\Adm^{\Phi,\fstr}\) is simple in the filtered strongly equivariant
category.
This proves the orbit classification.  Finally, arithmetic localization at
trivial character transports the same classification to the crystalline
distribution algebra
\[
  A_0=
  \Gamma\bigl(\mathfrak X,
  \mathscr D^\dagger_{\mathfrak X,\mathbb Q}\bigr).
\]

The proof therefore has the following dependency order.  The definition of
filtered strong equivariance is used first to obtain moment-zero characteristic
containment.  Moment-zero containment, together with finite orbit type and the
orbit-separability hypothesis, gives holonomicity.  In the Frobenius range,
holonomicity is upgraded to geometric overholonomicity by Caro's stability of Frobenius holonomicity over smooth projective formal schemes.  The noetherian support truncation is used only after
holonomicity is known; the overholonomic adjunctions are used only after the
Frobenius object has been placed in the overholonomic heart.  Restriction to the
open orbit produces an equivariant overconvergent \(F\)-isocrystal, intermediate
extension extracts the corresponding open-orbit piece, and induction on the
number of orbits yields the equivariant devissage.  The classification theorem and its distribution-side form are
consequences of this devissage together with the minimality of \(j_{!+}\) and
arithmetic Beilinson--Bernstein localization.

\section{The relevant overholonomic t-structure}
\label{app:t-structure}

This appendix fixes the categorical convention used whenever the main text
speaks about images, kernels, quotients, intermediate extensions and support
truncations.  Two abelian environments occur in the proof, and they must not be
conflated.  Before the Frobenius regularity theorem is invoked, support
truncation is used only in the holonomic coherent category.  Intermediate
extensions, their adjunctions and their minimality are used only after the
object under consideration has been placed in the overholonomic heart.  The
proof does not require forming an intersection between an overholonomic object
and a merely holonomic object.

Let \((Y,\overline Y)\) be a smooth locally closed couple realized inside the
fixed smooth proper formal scheme \(\mathfrak X\).  We write
\[
  D^b_{\mathrm{ovhol}}(Y/K)
\]
for the corresponding triangulated category of geometrically overholonomic
complexes, stable after extension of the base field, and we denote by
\[
  \Ovhol(Y/K)
\]
the heart of its canonical t-structure.  This is the heart used in
\cite[\S2.1, first paragraphs]{HuygheSchmidtIntermediate}; its abelian and
extension-closed nature is used repeatedly below.  When \(Y=X_s\), the notation
agrees with the category of geometrically overholonomic arithmetic
\(\mathscr D^\dagger_{\mathfrak X,\mathbb Q}\)-modules used in the body of the
paper.

All images in the definition of intermediate extension are images in this
overholonomic heart.  In the notation of Huyghe--Schmidt, if \(u:Y\to Y'\) is a
\(c\)-immersion and \(E\in\Ovhol(Y/L)\), then
\[
  u_{!+}(E)=\operatorname{Im}_{\Ovhol(Y'/L)}
  \bigl(\theta^0_{u,E}:u^0_!E\longrightarrow u^0_+E\bigr),
\]
where \((- )^0=H^0_t(-)\) is taken with respect to the canonical
\(t\)-structure.  For \(c\)-affine immersions, the functors \(u_!\) and
\(u_+\) are \(t\)-exact, and the same construction is literally the image of
\(u_!E\to u_+E\) in the heart.  For the locally closed couples occurring in the
classification theorem, Huyghe--Schmidt use the associated \(c\)-locally closed
immersion and the same degree-zero convention; see
\cite[\S2.2, definition of \(u_{!+}\), and \S2.3]{HuygheSchmidtIntermediate}.  Consequently, if
\(j:U\hookrightarrow Y\) is a smooth locally closed immersion of couples and
\(E\in\Ovhol(U/K)\) is a coefficient object in the overholonomic heart--in
particular in the Frobenius coefficient range used in the classification--the
notation \(j_!E\), \(j_+E\), and \(j_{!+}(E)\) in the body of the paper refers
to these degree-zero objects and images whenever intermediate extension or
minimality is invoked.  The phrase
``no boundary subobject or quotient'' means no such object in the abelian heart
\(\Ovhol(Y/K)\).

For a closed subset \(Z\subseteq Y\), the support truncation used before any
overholonomic adjunction is invoked is the maximal coherent submodule
\[
  \Gamma_Z(\mathcal M)\subseteq \mathcal M
\]
supported on \(Z\).  Whenever this notation is used in the proof, the object
under consideration has already been placed in the holonomic range by the
moment-map argument.  The construction is noetherian: every coherent subobject
supported on \(Z\) factors through \(\Gamma_Z(\mathcal M)\).

The open-orbit extraction step is deliberately performed only after the
Frobenius regularity theorem has placed the object in the overholonomic heart.
For the quotient \(Q=\mathcal M/\Gamma_Z(\mathcal M)\), the morphisms
\[
  j^0_!(Q|_O)\longrightarrow Q,
  \qquad
  Q\longrightarrow j^0_+(Q|_O)
\]
are therefore the standard adjunction morphisms in \(\Ovhol(Y/K)\).  The second
map is a monomorphism because its kernel is supported on the boundary and
\(Q\) has no nonzero subobject supported there.  The image of the first map is
then identified, through this monomorphism, with the defining image of
\(j^0_!(Q|_O)\to j^0_+(Q|_O)\), namely \(j_{O,!+}(Q|_O)\).  Thus no intersection
of an overholonomic object with a merely holonomic object is used, and no
uncited ambient subquotient-stability assertion is required.

Finally, the equivariant structures do not change the ambient abelian category.
A filtered strongly equivariant object is an object of the underlying holonomic or
overholonomic heart, according to the stage of the proof, together with the
additional equivariant data of \Cref{def:strong-equivariance}.  Kernels,
cokernels, images and quotients in the equivariant category are computed in the
underlying heart and then endowed with the induced equivariant structure.  This
is the convention used in \Cref{lem:strong-equivariant-abelian},
\Cref{lem:invariant-support-truncation},
\Cref{lem:middle-extension-extraction}, and
\Cref{prop:equivariant-orbital-devissage}.


\section*{Notation index}
\addcontentsline{toc}{section}{Notation index}

\begin{center}
\begin{tabularx}{\textwidth}{@{}lX@{}}
\toprule
Symbol & Meaning \\
\midrule
\(\mathcal V\) & Complete discrete valuation ring of mixed characteristic \((0,p)\). \\
\(K\) & Fraction field of \(\mathcal V\). \\
\(k\) & Perfect residue field of \(\mathcal V\). \\
\(\mathcal G\) & Split connected reductive group scheme over \(\mathcal V\). \\
\(\mathcal B\) & Fixed Borel subgroup scheme of \(\mathcal G\). \\
\(\mathfrak X\) & Formal flag variety \(\widehat{\mathcal G/\mathcal B}\). \\
\(X_s\) & Special fiber of \(\mathcal G/\mathcal B\). \\
\(\mathcal K\) & Smooth closed subgroup scheme of \(\mathcal G\). \\
\(\mathfrak K\) & Formal completion \(\widehat{\mathcal K}\). \\
\(\mathcal K_s\) & Special fiber of \(\mathcal K\), acting on \(X_s\). \\
\(\mathscr D^\dagger_{\mathfrak X,\mathbb Q}\) & Berthelot's sheaf of overconvergent arithmetic differential operators. \\
\(\Car(\mathcal M)\) & Characteristic variety of a coherent arithmetic \(\mathscr D^\dagger\)-module. \\
\(\Supp(\mathcal M)\) & Support of \(\mathcal M\) on the special fiber. \\
\(\Gamma_Z(\mathcal M)\) & Maximal coherent submodule of \(\mathcal M\) supported on \(Z\), for holonomic \(\mathcal M\). \\
\(\Isoc^\dagger(Y,\overline Y)\) & Overconvergent isocrystals on the pair \((Y,\overline Y)\). \\
\(j_{O,!+}(E)\) & Arithmetic intermediate extension of \(E\) from the orbit \(O\). \\
\(A_0\) & Global section algebra \(\Gamma(\mathfrak X,\mathscr D^\dagger_{\mathfrak X,\mathbb Q})\). \\
\bottomrule
\end{tabularx}
\end{center}


\end{document}